\documentclass[10pt]{article}
\usepackage{amsthm,amsfonts,amsmath,amssymb}
\usepackage{hyperref}
\usepackage{graphics} 
\usepackage{epstopdf} 
\usepackage{xypic}
\usepackage[all,arrow,matrix]{xy}

\theoremstyle{definition}

\newtheorem{thm}{Theorem}
\newtheorem{defn}[thm]{Definition}
\newtheorem{prop}[thm]{Proposition}
\newtheorem{cor}[thm]{Corollary}
\newtheorem{rema}[thm]{Remark}
\newtheorem{lemma}[thm]{Lemma}
\newtheorem{ass}[thm]{Assumption}

\newtheorem{conv}[thm]{Convention}
\newtheorem{conj}[thm]{Conjecture}

\numberwithin{equation}{section}
\numberwithin{thm}{section}

\newcommand\void[1]       {}

\newcommand\nn             {\nonumber \\}

\newcommand\be            {\begin{equation}}
\newcommand\ee            {\end{equation}}
\newcommand\bea           {\begin{eqnarray}}
\newcommand\eea         {\end{eqnarray}}
\newcommand\bnu          {\begin{enumerate}}
\newcommand\enu          {\end{enumerate}}

\newcommand{\ble}{\begin{lemma}}
\newcommand{\ele}{\end{lemma}}
\newcommand{\bpf}{\begin{proof}}
\newcommand{\epf}{\end{proof}}

\newcommand\bearll        {\begin{array}{ll}\displaystyle}

\newcommand\eear          {\end{array}}

\newcommand\CHom      {\mathcal{H}\mathrm{om}}
\newcommand\Hom           {\mathrm{Hom}}

\newcommand\id            {\mathrm{id}}
\newcommand\one           {{\bf1}}
\newcommand\op          {\mathrm{op}}
\newcommand\ev          {\mathrm{ev}}
\newcommand\comp          {\mathrm{C}}

\newcommand{\pf}{\begin{proof}}

\newcommand\CC           {\mathcal{C}}
\newcommand\CD           {\mathcal{D}}

\newcommand\CF          {\mathcal{F}}
\newcommand\CG         {\mathcal{G}}

\newcommand\CL           {\mathcal{L}}
\newcommand\CM           {\mathcal{M}}
\newcommand\CN           {\mathcal{N}}

\newcommand\CP           {\mathcal{P}}

\newcommand\CZ           {\mathcal{Z}}

\newcommand\BM           {\mathbf{M}}
\newcommand\BZ           {\mathbf{Z}}
\newcommand\BD          {\mathbf{D}}

\newcommand\bfC    {\mathbf{C}}
\newcommand\bfD    {\mathbf{D}}
\newcommand\bfF     {\mathbf{F}}
\newcommand\bfG    {\mathbf{G}}

\newcommand{\Bm}{\mathbf{m}}

\newcommand{\End}{\mathcal{E}\hspace{-.7pt}\mathrm{nd}}

\newcommand{\Mor}{\mathcal{M}\mathrm{or}}
\newcommand{\Vect}{\mathrm{Vect}}
\newcommand{\Coco}{\mathbf{C}}
\newcommand{\Fun}{\mathcal{F}\mathrm{un}}
\newcommand{\Nat}{\mathrm{Nat}}

\newcommand{\cosp}{\mathcal{C}\hspace{-1pt}osp}
\newcommand{\tdiag}{\mathcal{D}\hspace{-.5pt}iag}
\newcommand{\alg}{\mathcal{A}\hspace{.2pt}lg}
\newcommand{\calg}{\mathcal{C}alg}

\newcommand{\Cosp}{\mathbf{Cosp}}				
\newcommand{\CAlg}{\mathbf{C\hspace{-1pt}Alg}}		
\newcommand{\Mod}{\mathbf{Mod}}

\newcommand{\CALG}{\mathrm{C\hspace{-.8pt}A\hspace{-.3pt}L\hspace{-.4pt}G}}		

\newcommand{\tdiagu}{{\underline{\mathcal{D}\hspace{-.5pt}ia\phantom{a}}}\hspace*{-0.5em}g}
\newcommand{\Cospu}{{\underline{\mathbf{Cos\phantom{a}}}}\hspace*{-0.5em}\mathbf{p}}
\newcommand{\CALGu}{{\underline\CALG}}		

\newcommand{\Alg}{\mathbf{Alg}}

\newcommand\nxt{\noindent\raisebox{.08em}{\rule{.44em}{.44em}}\hspace{.4em}}

\newcommand{\da}{\mbox{-}}
\newcommand{\se}{\mbox{:}}

\begin{document}

\begin{center} \LARGE
Functoriality of the center of an algebra
\end{center}

\vskip 2em
\begin{center}
{\large 
Alexei Davydov$^{a}$,\, Liang Kong$^{b,c}$,\,  Ingo Runkel$^{d}$,\, ~\footnote{Emails: 
{\tt  davydov@ohio.edu, kong.fan.liang@gmail.com, ingo.runkel@uni-hamburg.de}}}
\\[1em]
\it$^a$ Department of Mathematics, Ohio University \\
Athens, OH 45701, USA
\\[1em]
$^b$ Institute for Advanced Study (Science Hall) \\ 
Tsinghua University, Beijing 100084, China
\\[1em]
$^c$ Department of Mathematics and Statistics\\
University of New Hampshire, Durham, NH 03824, USA
\\[1em]
$^d$ Fachbereich Mathematik, Universit\"at Hamburg\\
  Bundesstra\ss e 55, 20146 Hamburg, Germany
\end{center}

\vskip 4em

\begin{abstract}
The notion of the center of an algebra over a 
field $k$ has a far reaching generalization to algebras in monoidal categories. The center then lives in the monoidal center of the original category. This generalization plays an important role in the study of bulk-boundary duality of rational conformal field theories. 

In this paper, we study functorial properties of the center. We show that it gives rise to a 2-functor from the bicategory of semisimple indecomposable module categories over a fusion category to the bicategory of commutative algebras in the monoidal center of this fusion category. Morphism spaces of the latter bicategory are extended from algebra homomorphisms to certain categories of cospans. We conjecture that the above 2-functor arises from a lax 3-functor between tricategories, and that in this setting one can relax the conditions from fusion categories to finite tensor categories.

We briefly outline how one is naturally lead to the above 2-functor when studying rational conformal field theory with defects of all codimensions. For example, the cospans of the target bicategory correspond to spaces of defect fields and to the bulk-defect operator product expansions.
\end{abstract}

\newpage

\tableofcontents

\newpage

\section{Introduction}

The {\em full center} of an algebra $A$ in a monoidal category $\CC$ is a commutative algebra 
$Z(A)$ in the monoidal center $\CZ(\CC)$ of $\CC$, defined in terms of $A$ via a universal property \cite{da}. 
When $\CC$ is the category of finite-dimensional vector spaces over a field $k$,  we have $\CZ(\CC) \cong \CC$ and the full center coincides with the usual center of an algebra over $k$. 
The full center shares in particular the following property
	with
the latter: if two algebras $A,B \in \CC$ are Morita-equivalent, then their full centers $Z(A)$ and $Z(B)$ are isomorphic, and under additional assumptions on $A,B$ and $\CC$, the converse holds as well \cite{morita,da}.
The full center was first introduced in the context of two-dimensional rational conformal field theory (CFT) to capture the relation between boundary fields and bulk fields  \cite{ffrs3,cardy-sew-1}. In \cite{da}, the construction of the full center was formulated in a way applicable to algebras in monoidal categories in general. 

\medskip
In this paper we study functorial properties of the full center. Note that even the classical center construction is not functorial in a straightforward way. Indeed, let $k$ be a field and $\Vect_k$ the category of vector spaces over $k$. Consider the category $\alg(\Vect_k)$ of algebras in $\Vect_k$ (as objects) and algebra homomorphisms (as morphisms). With respect to this category, the assignment $A \mapsto Z(A)$ is not functorial because a homomorphism of algebras $A\to B$ does not induce in general 
a homomorphism of their centers $Z(A)\to Z(B)$. In this paper we propose to remedy this by enlarging the spaces of morphisms between commutative algebras. 
In more detail, we will introduce a bicategory $\CAlg(\Vect_k)$ whose objects are commutative $k$-algebras and whose 
categories of morphisms are certain categories of cospans of algebras. The assignment $A \mapsto Z(A)$, together with properly defined maps on morphisms, gives a lax functor $\BZ: \alg(\Vect_k)\to \CAlg(\Vect_k)$. The vector space example, together with its physical meaning in two-dimensional topological field theory, was explained in detail in \cite{dkr2}; we summarize the results in Sections~\ref{sec:intro-1}.

\medskip

In the present paper we generalize the domain category first to $\alg(\CC)$ for a monoidal category $\CC$, and then further to (a subcategory of) the bicategory $\Mod(\CC)$ of $\CC$-modules (as objects), $\CC$-module functors (as 1-morphisms) and $\CC$-module natural transformations (as 2-morphisms). 
The codomain category is generalized first to $\CAlg(\CZ)$, where $\CZ$ is an abelian braided monoidal category $\CZ$ with right exact tensor product, and then further to a bicategory $\CALGu(\CZ)$ which is truncated from a conjectural tricategory $\CALG(\CZ)$. Our main result (Theorem \ref{thm:center-lax-functor-2}) is that the full center construction provides a lax 2-functor
\be\label{eq:intro-Z}
  \BZ : \BM(\CC) \to \CALGu(\CZ(\CC))  \ ,
\ee
where $\BM(\CC)$ is a suitable sub-bicategory of $\Mod(\CC)$. The construction of $\BZ$ is summarized in Section~\ref{sec:intro-2}. 
We believe that the lax 2-functor $\BZ$ can be lifted to a lax 3-functor from $\Mod(\CC)$ to $\CALG(\CZ(\CC))$ if all internal homs exist, 
but we do not address this question in the present paper.
We prove that $\BZ$ in \eqref{eq:intro-Z} becomes a non-lax 2-functor in (at least) the following two situations:
\begin{itemize}
\item if we restrict the domain to the maximal 2-groupoid $\BM(\CC)^\times$ inside of $\BM(\CC)$ (Theorem \ref{thm:funct-eq}), and 
\item if $\CC$ is a fusion category and we take $\BM(\CC) = \Mod^o(\CC)$, the full sub-bicategory of $\Mod(\CC)$ consisting of only semisimple and indecomposable $\CC$-modules (Theorem \ref{thm:fusion-cat-Z-functor}). 
\end{itemize}

As explained in \cite{dkr2}, the lax functor $\BZ: \alg(\Vect_k)\to \CAlg(\Vect_k)$ arises naturally in two-dimensional {\em topological} field theory. In the application to rational {\em conformal} field theory, $\Vect_k$ gets replaced by a modular tensor category $\CC$, which is the category of modules 
of a rational vertex operator algebra (VOA) (i.e.\ a VOA that satisfies the conditions in \cite{huang-mtc}). The various ingredients of the lax 2-functor $\BZ$ have natural interpretations in rational CFT with defect lines and defect fields. This is outlined in Section~\ref{sec:intro-3}.

\subsection{The center of an algebra over a field}\label{sec:intro-1}

Fix a field $k$; all algebras in this subsection will be unital associative algebras over $k$. 
Recall that the center of an algebra $A$ is the commutative subalgebra: $$Z(A)=\{z\in A\,|\, za = az\ \forall a\in A\}.$$
An algebra homomorphism $f:A\to B$ does in general not give an algebra homomorphism from $Z(A)$ to $Z(B)$, 
	e.g.\ take the inclusion of diagonal $2{\times}2$-matrices into all $2{\times}2$-matrices. 
To remedy this, we will use certain cospans as morphisms between commutative algebras. This is motivated by the following construction:
The centralizer of the image of $f$ defined by
\be  \label{eq:Z-f}
  Z(f) = C_B(f(A)) =\{z\in B \,|\, zf(a) = f(a)z\ \forall a\in A\}
\ee
is also an algebra. Note also that $Z(B)$ is a subalgebra of $Z(f)$, and $f$ 
maps
$Z(A)$ into $Z(f)$, and both $Z(B)$ and $f(Z(A))$ are in the center of $Z(f)$. We summarize this by the following diagram (a {\em cospan} in the category of $k$-algebras)
\be\label{cf}
\raisebox{1em}{\xymatrix@R=0.5em{  
  & Z(f) & \\ Z(A) \ar[ur]^f & & Z(B) \ar[ul]
  }}
\ee
Given two successive algebra homomorphisms $A\stackrel{f}{\to}B\stackrel{g}{\to}C$ we obtain the commutative diagram
\be
\raisebox{1.2em}{\xymatrix@R=0.5em{
& & Z(gf) & & \\ & Z(f) \ar[ur]^g & & Z(g) \ar[ul] & \\ Z(A) \ar[ur]^f & & Z(B) \ar[ul] \ar[ur]^g & &
Z(C) \ar[ul] }} 
\ee
which shows that there is an algebra homomorphism from the push-out $Z(f)\otimes_{Z(B)}Z(g)$ to $Z(gf)$. This structure suggests 
a bicategory $\CAlg(\Vect_k)$ as follows:

\medskip\noindent
\nxt {\bf Objects}: commutative algebras $C,D,\dots$ in $\Vect_k$.

\medskip\noindent
\nxt {\bf 1-morphisms}:
A 1-morphism $C \rightarrow D$ in $\CAlg(\Vect_k)$ is a triple $(S,c,d)$, where $S$ is an algebra (not necessarily commutative) in $\Vect_k$ and $c : C \rightarrow S$ and $d : D \rightarrow S$ are homomorphisms of algebras, whose images lie in the center $Z(S)$ of $S$. This is represented by the following diagram
\be
\xymatrix{
Z(S) \ar[r] & S & Z(S) \ar[l] \\ C\ar[ru]_c \ar[u] && D \ar[lu]^d \ar[u] }
\ee
The composition of $C\xrightarrow{(S,-,-)} D \xrightarrow{(T,-,-)} E$ is defined by the pushout $(C \xrightarrow{(S\otimes_D T,-,-)} E)$:
\be
\raisebox{2em}{\xymatrix@R=0.5em{  
& & \hspace{-1em} S \otimes_D T \hspace{-1em} & & \\ & S \ar[ur] & & T \ar[ul] & \\ C \ar[ur] & & D \ar[ul] \ar[ur] & &
E \ar[ul] }} 
\ee

\medskip\noindent
\nxt {\bf 2-morphisms}: a 2-morphism between two cospans $C \xrightarrow{(S,-,-)} D$ and $C \xrightarrow{(T,-,-)} D$ is a homomorphism of algebras $S \xrightarrow{f} T$ such that the following diagram commutes:
\be
\raisebox{2.2em}{\xymatrix@R=0.7em{  
& S \ar[dd]^f & \\ C \ar[ru] \ar[rd] &  & D \ar[ul] \ar[dl]\\
& T }}
\ee

\medskip

Our goal -- to make the assignment of the center functorial -- is now achieved. More precisely, we obtain a lax functor $Z: \alg(\Vect_k) \to \CAlg(\Vect_k)$ by assigning $A \mapsto Z(A)$ and $(A \xrightarrow{f} B) \mapsto \mbox{cospan (\ref{cf})}$, and defining the unit transformation by identity 2-morphisms and the multiplication transformations by the following 2-morphisms: 
$$
\raisebox{2.7em}{\xymatrix@R=1em{  
 & Z(f)\otimes_{Z(B)}Z(g) \ar[dd] &  \\  Z(A) \ar[ur] \ar[dr]  & & Z(B)~. \ar[ul] \ar[dl] \\  &  Z(gf)
}} 
$$ 
The proof of this claim can be found in \cite[Thm.\,4.12]{dkr2}, but it is also a special case of the general construction we turn to now.

\subsection{The full center for module categories}\label{sec:intro-2}

Motivated by the application in rational CFTs, we generalize the full center construction in two steps.
First we extend the category $\alg(\Vect_k)$ to the bicategory $\Alg(\Vect_k)$ of $k$-algebras (as objects), bimodules (as 1-morphisms) and bimodule maps (as 2-morphisms), and then we generalize further to $\Mod(\CC)$ for a monoidal category $\CC$. Each step brings a modification to the previous constructions. 

\medskip
For an object $A$ in $\Alg(\Vect_k)$, the notion of 
the
center can be defined equivalently by
$$
Z(A) := \Hom_{A\otimes A^\op}(A, A)  \ ,
$$ 
where $\Hom_{A\otimes A^\op}(A, A)$ is the space of $A$-$A$-bimodule maps from $A$ to $A$, and $\text{End}(\id_{\text{$A$-mod}})$ the space of $k$-linear natural transformations from the identity functor to itself. Therefore, for an object $A$ in $\Alg(\Vect_k)$, we can assign its center, i.e. $A \mapsto Z(A)$.

For a $1$-morphism $A \xrightarrow{M} B$ in $\Alg(\Vect_k)$, i.e.\ an $A$-$B$-bimodule $M$, we  define the value of the center functor on $M$ by: 
\be \label{def:Z-M}
Z(M) := \Hom_{A\otimes A^\op}(M, M).
\ee
It fits into a cospan: 
\be \label{eq:Z-M}
\raisebox{1em}{\xymatrix@R=0.5em{  
  & Z(M) & \\ Z(A) \ar[ur] & & Z(B) \ar[ul]
  }}
\ee
where, using $Z(A)=\Hom_{A\otimes A^\op}(A,A)$, the two algebra maps in the cospan 
are defined by 
\bea
(A \xrightarrow{a} A) &\mapsto& (M \simeq A\otimes_A M \xrightarrow{a\otimes_A \id_M} A\otimes_A M \simeq M),  \nn
(B \xrightarrow{b} B) &\mapsto& (M \simeq M\otimes_B B \xrightarrow{\id_m\otimes b} M\otimes_B B \simeq M). \nonumber
\eea
Similar to (\ref{cf}), we assign the cospan (\ref{eq:Z-M}) to a 1-morphism $A\xrightarrow{M} B$ in $\Alg(\Vect_k)$.

Note that for a 2-morphism $\phi$ in $\Alg(\Vect_k)$, i.e. a bimodule map $\phi: M \to N$, there is no algebra map from $Z(M)$ to $Z(N)$ associated to $\phi$ unless $\phi$ is invertible. So 
2-morphisms in $\CAlg(\Vect_k)$ must be replaced by something else in order to obtain a functor. As we 
show in Section~\ref{sec:intro-3}, the intuition from defect junction in RCFT provides a solution to the problem. More precisely, one can consider the following diagram: 
\be \label{eq:2-diag-0}
\raisebox{3em}{\xymatrix@R=2em{  
  & Z(M) \ar[d]^{[M, \phi]} & \\ Z(A) \ar[ur] \ar[rd] & [M,N] & Z(B) \ar[ul] \ar[dl]  \\
  & Z(N) \ar[u]^{[\phi,N]} 
  }}
\ee
where $[M,N]:=\Hom_{A\otimes A^\op}(M, N)$, $[M,\phi]:=\phi\circ -$ and $[\phi, N]:=-\circ \phi$. We show in later sections that the above diagram is actually commutative, and $[M,\phi]$ is a right $Z(M)$-module map, and $[\phi, N]$ is a left $Z(N)$-module map, satisfying additional properties. This leads us to a new definition of a 2-morphism in the target category. We give the precise definition of such a 2-morphism, called a {\it 2-digram}, in Definition~\ref{def:2-diagram+3-cell}. Together with appropriate 3-morphisms between such 2-diagrams (see Definition~\ref{def:2-diagram+3-cell}), we conjecturally obtain a  tricategory $\CALG(\Vect_k)$ and a lax 3-functor
from $\Alg(\Vect_k)$ to $\CALG(\Vect_k)$. 
Now we look at
the second generalization, where we replace $\Vect_k$ by more general monoidal categories.

\medskip
Let $\CC$ be a monoidal category. Instead of working with the bicategory of algebras, bimodules and bimodule maps, we 
pass to the bicategory $\Mod(\CC)$ of $\CC$-module categories.
For a monoidal category $\CC$, the notion of a module category over $\CC$ or a $\CC$-module is simply the categorification of the notion of a module over a ring \cite{ostrik}. In more detail, a left $\CC$-module is a category $\mathcal{M}$ together with a bifunctor $\CC \times \mathcal{M} \rightarrow \mathcal{M}$ and associator and unit transformations subject to coherence conditions (see Definition~\ref{def:mod-cat}). The notion of a $\CC$-module functor and that of a $\CC$-module natural transformation can be defined accordingly (see Definition~\ref{def:C-module-functor} and \ref{def:C-mod-nat-xfer}). $\Mod(\CC)$ is the bicategory of $\CC$-modules (as objects), $\CC$-module functors (as 1-morphisms) and $\CC$-module natural transformations (as 2-morphisms).

For a given monoidal category $\CC$, one cannot define the notion of full center for a $\CC$-algebra (or for a $\CC$-module) as an object in $\CC$ because $\CC$ is not braided in general. But what one can do is to define the center as an object of the monoidal center $\CZ(\CC)$ of $\CC$. A definition of full center for a $\CC$-algebra (or a $\CC$-module) was introduced 
in \cite{da}. In particular, for a $\CC$-module $\CM$, the full center $Z(\CM)$ is defined by an internal hom $[\id_\CM, \id_\CM]$ valued in $\CZ(\CC)$.  This definition also works for any $\CC$-module functor $F: \CM \to \CN$. Namely, we can define the centralizer of a $\CC$-module functor $F$ as an internal hom $Z(F):=[F,F]$ valued in $\CZ(\CC)$. This generalize both (\ref{eq:Z-f}) and (\ref{def:Z-M}). We recall the definition of $[F,F]$ in Section~\ref{sec:cent-module-fun}. Therefore, the cospan (\ref{eq:Z-M}) is generalized to a new cospan:
\be  \label{diag:Z-F}
\raisebox{1em}{\xymatrix@R=0.5em{  
  & Z(F)  & \\ Z(\CM) \ar[ur] & & Z(\CN)~. \ar[ul] }}
\ee
In this case, the construction of 2-diagrams is similar (see the right diagram in (\ref{eq:Z-functor})). Altogether, our full center construction can be summarized in the following multilayered assignment:
\be \label{eq:Z-functor}
\BZ: \,\, \xymatrix{ \CM \ar@/^12pt/[rr]^F \ar@/_12pt/[rr]_G & \Downarrow \phi & \CN}
\quad \longmapsto \quad
\raisebox{3em}{\xymatrix@R=1.3em{  
  & Z(F) \ar[d]^{[F, \phi]} & \\ 
  Z(\CM) \ar@/^8pt/[ur] \ar@/_8pt/[rd] & [F,G] & Z(\CN) \ar@/_8pt/[ul] \ar@/^8pt/[dl]  \\
  & Z(G) \ar[u]^{[\phi,G]} 
  }} \ .
\ee
With a proper definition of 3-morphisms between 2-diagrams (see Definition~\ref{def:2-diagram+3-cell}), we believe that we obtain a tricategory $\CALG(\CZ(\CC))$ of commutative $\CZ(\CC)$-algebras (as objects), cospans 
\eqref{diag:Z-F} (as 1-morphisms), 2-diagrams (as 2-morphisms) and properly defined 3-morphisms, and that \eqref{eq:Z-functor} leads to a lax 3-functor from $\Mod(\CC)$ to $\CALG(\CZ(\CC))$ with appropriate assumptions on $\CC$ and $\CZ(\CC)$. 
We do not deal with the 3-category extension in this paper. Instead we
prove in Section~\ref{sec:under-CALG-Z} that we obtain a bicategory $\CALGu(\CZ(\CC))$ if take 2-morphisms to be  equivalence classes of 2-diagrams, and in Section~\ref{sec:center-func} that $\BZ$ defines a lax 2-functor.

\subsection{Relation to two-dimensional field theory}\label{sec:intro-3}

Physical motivations for the definition of $\CALGu(\CZ(\CC))$ and the full center construction $\BZ$ comes from two-dimensional rational CFTs. A 
natural
way to compare two CFTs (or any two field theories of a given dimension), is to study 
domain walls between them. These are codimension 1 submanifolds on a
two-dimensional surface on which the CFTs are defined.  Similarly a natural
way to compare two domain walls between a fixed pair of CFTs is to study domain walls between domain walls. These are codimension 2 submanifolds which lie inside the codimension 1 domain walls. For two-dimensional theories we have now reached dimension zero, 
and we have to stop here.
For higher-dimensional field theories one can go on to higher codimensions (see for example \cite{kitaev-kong,Fuchs:2012dt,u-lw-mod}).

Let us describe CFTs with domain walls in more detail and show the relation to $\CALGu$ and the full center construction. In doing so we 
gloss over a number of subtleties, which we briefly comment on at the end of this section.

\medskip\noindent
\nxt {\bf CFT without domain walls}: A CFT can be defined in terms of its correlation functions, which are subject to certain consistency conditions called sewing constraints. We use an axiomatic approach via the representation theory of vertex operator algebra (VOA), see for example \cite{kong2} for more details and references therein.

To an oriented compact closed Riemann surface $\Sigma$ with an ordered list of marked points (a point on $\Sigma$ together with a local coordinate and labelled `in' or `out'), a CFT assigns a multilinear map $C(\Sigma) : B \times \cdots \times B^* \times \cdots \rightarrow \mathbb{C}$, called the {\em correlator}. Here $B$ is a vector space (graded by scaling dimension), called the {\em space of bulk fields}, and $B^*$ is its graded dual. The domain of the correlation function $C(\Sigma)$ has one factor of $B$ for each marked point labelled `in' and one factor $B^*$ for each point labelled `out'. The consistency conditions arise via gluing surfaces by identifying neighborhoods of an in-going and out-going marked point using the local coordinates.

To make the connection to the construction in Sections~\ref{sec:intro-1} and \ref{sec:intro-2}, we will consider rational CFTs. In these cases, the subspace of holomorphic fields contains a rational vertex operator algebra (VOA) $V_L$ and the subspace of antiholomorphic fields contains a rational VOA $V_R$ called the left- and right-moving chiral symmetry algebras\footnote{A VOA is defined in terms of formal variables. One can chose to replace the formal variables by holomorphic complex variables $z_1, z_2, \cdots$ (or its complex conjugate $\bar{z}_1, \bar{z}_2, \cdots$) to obtain 
holomorphic fields (or 
anti-holomorphic fields).} . (Anti-)holomorphic fields have the defining property that any correlator involving such a field depends (anti-)holomorphically on the corresponding marked point. By a rational VOA we mean a VOA satisfying certain conditions such that the category $\text{Rep}(V)$ of $V$-modules is a modular tensor category \cite{huang-mtc}. We set $\text{Rep}(V)_+:=\text{Rep}(V)$ and $\text{Rep}(V)_-$ to be the same category as $\text{Rep}(V)$ but with braidings and twists replaced by anti-braidings and anti-twists, respectively. It is clear that $\text{Rep}(V)_-$ is also a modular tensor category. For two rational VOAs $V_L$ and $V_R$, the Deligne product $\text{Rep}(V_L)_+ \boxtimes \text{Rep}(V_R)_-$,  is also a modular tensor category. We abbreviate $\CD = \text{Rep}(V_L)_+ \boxtimes \text{Rep}(V_R)_-$. The space of bulk fields $B$ is an object in $\CD$. 

The correlation functions endow $B$ with an additional structure, for example the Riemann sphere with three marked points labelled `in', `in', and `out' defines an associative product on $B$ in the category $\CD$, and the sphere with one out-going marked point provides a unit, turning $B$ into a (unital, associative) commutative algebra; for more details see for example \cite{frs, kong2}. Thus a CFT without domain walls with chiral symmetry $V_L\otimes_\mathbb{C} V_R$ gives an object in $\CALGu(\CD)$. 

\medskip\noindent
\nxt {\bf CFT with domain walls}: A CFT with domain walls assigns correlators to Riemann surfaces $\Sigma$ with marked points that in addition have an embedded smooth oriented submanifold $\Delta$ of codimension 1, whose connected components are called {\em domain walls} or {\em defect lines}. For the present exposition we restrict our attention to situations where the submanifold has empty boundary, and where the surface is bicolorable in the following sense. 
To each connected component of $\Sigma \setminus \Delta$ we assign one of two possible colors, say `1' and `2', such that `to the left' of each defect line we have color 1 and `to the right' we have color 2 (the orientation of the surface and the submanifold determines `left' and `right'). A marked point may lie in an area of color 1 or 2, or it may lie on a defect line.
The correlators are again subject to sewing constraints. A formulation of CFT with defect lines similar to the functorial approach by Segal \cite{segal} has been given in \cite{r-suszek} and \cite{dkr2}.

Here we restrict our attention to such CFTs with domain walls, for which correlators depend on the position of the embedded submanifold only up to homotopy
\footnote{The homotopy is required to keep marked points fixed, and the submanifold must not intersect itself or cross marked points during the homotopy.}. 
Such domain walls are called `topological' (in the same sense that a field theory is called `topological' if its correlation functions depend on the marked points only up to homotopy).

Depending on wether a marked point lies in an area of color 1 or 2 or on a defect line, the corresponding argument of the correlator lies in one of three state spaces (or their duals), which we denote by $B_1$, $B_2$, and $D$, respectively; $D$ is called the {\em space of defect fields}. Note that if the Riemann surface contains no defect lines (i.e.\ the embedded submanifold is empty), we can color the entire surface either by 1 or 2. Thus a CFT with domain walls contains as a part of its data {\em two} CFTs without domain walls. 

\begin{figure}[bt]
\begin{center}
\raisebox{8em}{a)}
\hspace{.5em}
\scalebox{0.6}{\includegraphics{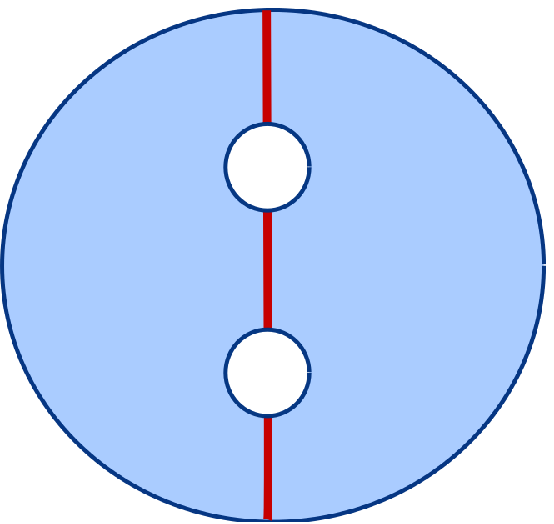}}
\hspace{6em}
\raisebox{8em}{b)}
\hspace{.5em}
\scalebox{0.6}{\includegraphics{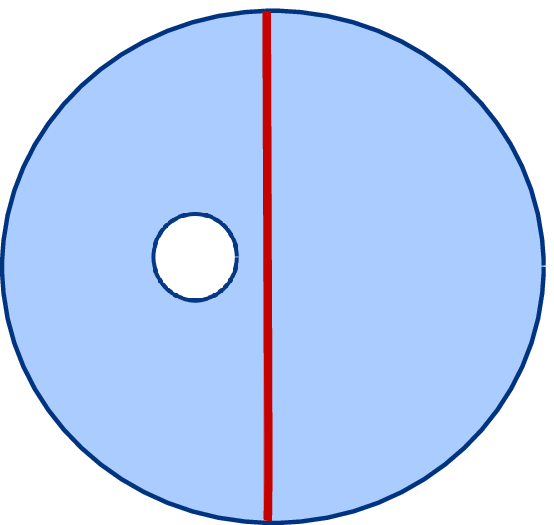}}
\end{center}
\caption{In these figures we cut out little discs around the marked points and mapped the picture to the plane. Figure a) shows the surface which provides the multiplication on the space $D$. Namely, it gives rise to a morphism $D \otimes D \to D$ in the category $\CD = \text{Rep}(V_L)_+ \boxtimes \text{Rep}(V_R)_-$. Figure b) shows the surface that defines the morphism $B_1 \to D$.}
\label{fig:defect-mult+inclusion}
\end{figure}

In the setting of rational CFT we demand in addition that $B_1$ and $B_2$ contain the same chiral symmetry, i.e.\ they are both objects in $\CD$, and we demand that the defect lines are `transparent' to $V_L \otimes_\mathbb{C} V_R$ (the defect line can cross a marked point without affecting the value of the correlation function if the corresponding argument is taken from $V_L \otimes_\mathbb{C} V_R$), which then implies that also $D$ is an object in $\CD$. 

The correlators endow the state spaces with an additional structure. As before, $B_1$ and $B_2$ are commutative associative algebras, and $D$ is a (not necessarily commutative) algebra. For example, the multiplication on $D$ is obtained from a sphere with a defect line placed on the equator and three marked points on the equator labelled `in', `in', and `out', see Figure~\ref{fig:defect-mult+inclusion}\,a.

\begin{figure}[bt]
\begin{center}
\raisebox{8em}{a)}
\hspace{.5em}
\scalebox{0.6}{\includegraphics{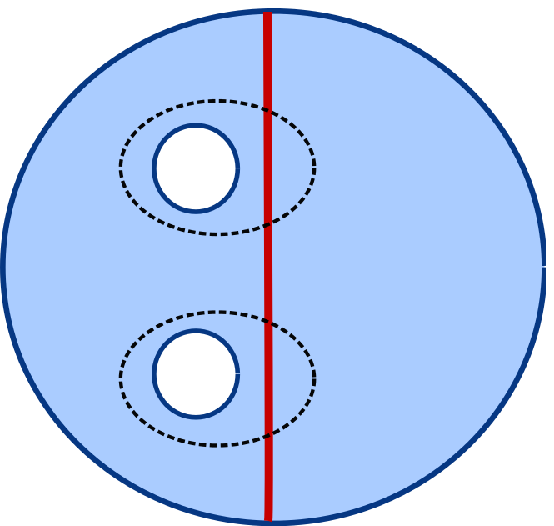}}
\hspace{6em}
\raisebox{8em}{b)}
\hspace{.5em}
\scalebox{0.6}{\includegraphics{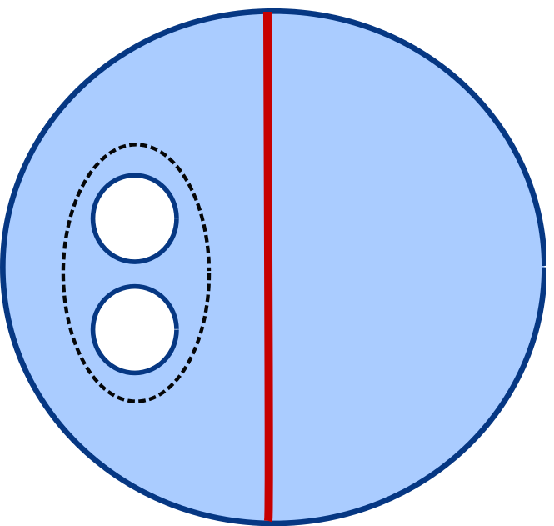}}
\end{center}
\caption{Surfaces can be sewn together using the local coordinates around an in- and out-going marked points. We symbolize this procedure in the above pictures by the dashed circles. For example, figure a) shows a sphere with two in-going and one out-going marked point which has been obtained by sewing two copies of the spheres in Figure~\ref{fig:defect-mult+inclusion}\,b to the in-going punctures of the sphere in Figure~\ref{fig:defect-mult+inclusion}\,a. The two decompositions at different points in the moduli space shown in figure a) and b) above result in the statement that the morphism $B_1 \to D$ in $\CD$ is compatible with the multiplication on $B_1$ and $D$.}
\label{fig:bulk-defect-map-mult} 
\end{figure}

Consider once more the sphere with an equatorial defect line, and take the upper and lower hemispheres to be of color 1 and 2, respectively. Putting an in-going marked point on the north pole and an out-going marked point on the equator defines a map $B_1 \rightarrow D$ (Figure~\ref{fig:defect-mult+inclusion}\,b); the sewing constraints imply that this has to be an algebra map (place two marked points in the upper hemisphere and decompose the surface in different ways -- see Figure~\ref{fig:bulk-defect-map-mult}). Similarly one obtains an algebra map $B_2 \rightarrow D$. 

Thus a CFT with topological domain walls and chiral symmetry $V_L \otimes_\mathbb{C} V_R$ gives a cospan
\be \label{eq:algcospan-via-CFT}
\raisebox{1em}{\xymatrix@R=0.5em{  
& D \\ B_1 \ar[ru] && B_2 \ar[lu]}}
\ee
i.e.\ a 1-morphism in $\CALGu(\CD)$.

\medskip\noindent
\nxt {\bf CFT with domain walls between domain walls}: Next we enlarge the allowed decoration data of 
Riemann surfaces to which a
CFT assigns correlators by including also codimension 2 submanifolds lying inside the codimension 1 submanifolds, i.e.\ distinguished points on the defect lines. These distinguished points are different from the marked points: they do not carry a local coordinate and they do not give rise to a factor in the product of vector spaces that forms the domain of the multilinear map given by the correlator. 

We demand that the defect lines are bicolorable in the following sense. Denote the two colors by `$a$' and `$b$', say. Each connected component of a defect line minus the marked and distinguished points carries color `$a$' or `$b$'. 
The color may change across a distinguished or marked point from `$a$' to `$b$' or vice versa.

As above, we restrict our attention to topological domain walls, and to topological domain walls between domain walls, i.e.\ the correlators do not depend on the position of the distinguished points on the defect lines, as long as one does not move a distinguished point past a marked point. 

For such a CFT we already have a rather long list of state spaces: $B_1$ and $B_2$ for marked points in areas of color 1 or 2; $D_a$ and $D_b$ for marked points on a defect line with color $a$ or $b$ on both sides of the marked point; $D_{ab}$ for a marked point on a defect line across which the color changes from `$a$` to `$b$'; $D_{ab}$ is called the space of defect changing fields. Note also that a CFT with domain walls between domain walls contains as part of its data {\em two} CFTs with domain walls, namely we can color defect lines without distinguished points by `$a$' or `$b$'.

The correlators endow the state spaces with addition structure. 
As before we have cospans $B_1 \rightarrow D_a \leftarrow B_2$ and $B_1 \rightarrow D_b \leftarrow B_2$. 
In addition, take a sphere with equatorial defect, hemispheres colored `1' and `2', and the defect itself split into two semicircles colored `a' and 'b'. These are joined by a distinguished point (for the transition from `$a$' to `$b$') and by an out-going marked point (for the transition from `$b$' to `$a$'). Placing an additional in-going marked point on the semi-circle colored `$a$' (resp.\ `$b$') gives a map $D_a \rightarrow D_{ab}$ (resp,\ $D_b \rightarrow D_{ab}$). These maps can be collected in the diagram
\be
\raisebox{3em}{\xymatrix@R=1.3em{  
& D_a \ar[d] & \\ B_1 \ar[ru] \ar[rd] & D_{ab} & B_2 \ar[ul] \ar[dl]\\
& D_b \ar[u]}}
\ee
The sewing constraints imply that $D_{ab}$ is a $D_a$-$D_b$-bimodule and that the above diagram satisfies all the defining properties of a 2-diagram (see Definition~\ref{def:2-diagram+3-cell}). Therefore, a CFT with chiral symmetry $V_L \otimes_\mathbb{C} V_R$ and with topological domain walls between domain walls, considered up to isomorphism, gives a 2-morphism in $\CALGu(\CD)$.

\medskip

\begin{table}[bt]
\begin{center}
\begin{tabular}{p{10em}p{26em}}
type of CFT 
  & algebraic data in $\mathrm{Rep}(V)$ \\[.3em]
\hline \\[-.7em]
without domain walls 
  & $A$, a (special symmetric Frobenius) algebra. In this case the space of bulk fields is $B=Z(A) \in \CD$, the full center of $A$.
\\[.3em]
with domain walls 
  & $(A_1,A_2,M)$, where $A_1$, $A_2$ are (special symmetric Frobenius) algebras, and $M$ is an $A_1$-$A_2$-bimodule. The space of defect fields $D$ is the internal hom $[F,F]\in \CD$, where $F = (-) \otimes_{A_1} M$ is a functor from $\mbox{$A_1$-mod}$ to $\mbox{$A_2$-mod}$. Both categories are left module categories over $\CD$.
\\[.3em]
\raggedright
  with domain walls \hbox{between} domain walls 
  & $(A_1,A_2,M_a,M_b,\phi)$, where $A_1$, $A_2$ are as above, $M_a$, $M_b$ are $A_1$-$A_2$-bimodules, and $\phi : M_a \rightarrow M_b$ is a bimodule intertwiner. The map $\phi$ provides a natural transformation between the functors $F = (-) \otimes_{A_1} M_a$ and $G = (-) \otimes_{A_2} M_b$. The space of defect changing fields $D_{ab}$ is the internal hom $[F,G]\in \CD$.
\end{tabular}  
\end{center} 
\caption{Different types of CFTs and the algebraic data needed to define them in the description via 3d\,TFT. Since the full center $Z(A)$ of $A$ only depends on the Morita class of $A$ \cite{morita,da},  
we can replace $A_1$ by $\CM:=\mbox{$A_1$-mod}$ and $A_2$ by $\CN:=\mbox{$A_2$-mod}$. The relations between the above data then organize themselves into the two diagrams given in \eqref{eq:Z-functor}.}
\label{tab:cft-via-3dtft-data}
\end{table}

So far we have outlined how the bicategory $\CALGu(\CD)$ appears in the study of CFT in the presence of domain walls. It turns out that the full center construction also arises in this context, as we now describe. Examples of the three types of CFTs discussed above can be constructed with the help of three-dimensional topological field theory (3d\,TFT). For this construction to apply, the chiral symmetry of the CFT needs to obey $V_L = V_R$, where $V := V_L = V_R$ is a rational VOA. In this case, the Deligne produce $\CD=\text{Rep}(V)_+ \boxtimes \text{Rep}(V)_-$ is canonically equivalent to the monoidal center $\CZ(\text{Rep}(V))$ of $\text{Rep}(V)$ \cite{mueger}. The 3d\,TFT in question is the one associated to the category $\mathrm{Rep}(V)$ via the construction of Turaev \cite{turaev}. CFT correlators are constructed as invariants of three-manifolds (with non-empty boundary) and embedded ribbon graphs \cite{frs,ffrs2,Fjelstad:2012mj}. The coloring of the ribbon graph depends on additional algebraic data in $\mathrm{Rep} V$ as listed in Table~\ref{tab:cft-via-3dtft-data}, and the relations between these pieces of data are exactly as given in (\ref{eq:Z-functor}). From this we see that the 3d\,TFT construction of rational CFTs in the presence of domain walls contains within it the full center construction. 

\begin{rema}
It is curious to note that, on the one hand, the functoriality of full center needs defects of all codimensions; on the other hand, the simple statement of this functoriality also summarizes all local structures in an RCFT with topological defects efficiently. This functoriality of full center in rational CFT is not an isolated phenomenon. A categorification of it \cite{u-lw-mod}, which is associated to the extended Turaev-Viro topological field theories, is still conjectural 
but obvious by physical intuition.  In \cite{u-lw-mod}, it was conjectured that this functoriality (with additional nice properties) holds for all extended topological field theories. A related but different 2-functor was constructed in \cite{ENO2009}. Another related result is presented in \cite[Cor.\,2.5.13]{Lurie:2009b}, where to an $\mathbb{E}[k]$-algebra (in a symmetric monoidal $\infty$-category) one assigns its center, which is an $\mathbb{E}[k{+}1]$-algebra in the same category.
\end{rema}

\medskip
This ends our short exposition of the relation between CFT and the full center construction. To conclude let us just list -- without comment -- some of the issues we have left aside to keep the presentation short. For higher genus surfaces one either has to include a line bundle over the moduli space of Riemann surfaces or content oneself with obtaining rays of multilinear maps as correlators (this is due to the conformal anomaly; it can be avoided at genus 0). The formulation of the sewing constraints for CFTs with domain walls requires the introduction of defect junctions of higher valencies, whereas above we have restricted ourselves to valency 2. The construction of correlators via 3d\,TFT is proved only for surfaces of genus 0 and 1, for surfaces of higher genus, it relies on the conjectural equivalence of two modular functors, one obtained from conformal blocks and one from the 3d\,TFT.

\bigskip

The rest of the paper is organized as follows. In Section \ref{sec:act-int-hom} we review the notion of action internal hom; in Section \ref{sec:full-center} we will review the notion of full center of an algebra in a monoidal category $\CC$ and 
set our notations; in Section \ref{sec:two-cat} we will construct two bicategories: $\CAlg(\CZ)$ and $\CALGu(\CZ)$; in Section \ref{sec:center-func} we will present the full center construction; in Section \ref{sec:auto} and Section \ref{sec:fusion-cat} we restrict the domain of the full center construction to $\Mod^\times(\CC)$ and $\Mod^o(\CC)$, respectively, and show that we obtain non-lax 2-functors in both cases. 

\bigskip\noindent
{\bf Acknowledgement}:
We would like to thank Victor Ostrik for explaining some of his earlier works to us. AD thanks Max Planck Institut f\"ur Mathematik (Bonn) for hospitality and excellent working conditions. LK is supported by the Basic Research Young Scholars Program and the Initiative Scientific Research Program of Tsinghua University, and NSFC under Grant No. 11071134. AD and IR thank Tsinghua University for hospitality during a visit where part of this work was completed.
IR is supported in part by the German Science Foundation (DFG) within the Collaborative Research Center 676 ``Particles, Strings and the Early Universe''.

\section{Action internal homs}  \label{sec:act-int-hom}

In this section we recall the definition of action internal homs. Some of their properties are best stated in the language of enriched categories and we will therefore also use this language, even if enriched categories do not feature in the rest of this paper. Most of the results collected in this section can be found e.g.\ in \cite{janelidze:2001,ostrik}.

\subsection{Module categories and internal homs}

Let $\CC$ be a monoidal category with tensor product $\otimes$ and tensor unit $1_\CC$ and let $\CM$ be a left module category over $\CC$ (or a $\CC$-{\em module}, for short). We briefly review the definition of module categories, module functors and natural transformations between them in Appendix~\ref{app:module-cat}. Action internal homs are defined by a universal property as follows.

\begin{defn} \label{def:action-internal-hom}
For $M,N\in \CM$ the {\em action internal hom} $[M,N]$ is an object of $\CC$ equipped with a map $\ev_M: [M, N] \ast M \rightarrow N$ such that $([M, N], \ev_M)$ is terminal among the pairs $(U,f)$, where $U\in\CC$ and $f:U\ast M\to N$ is a morphism in $\CM$. That is, for each such pair there is a unique morphism $\underline{f}: U\to[M,N]$ which makes the diagram
\be  \label{eq:underline-f}
\xymatrix{
U\ast M \ar[dr]_f \ar@{-->}[rr]^{\underline{f}\ast \id_M}  && [M,N]\ast M \ar[dl]^{\ev_M}  \\
& N
}
\ee
commute. Equivalently, the action internal hom $[M, N]$ is the terminal object in the comma-category $A_M\,{\downarrow}\,N$, where $A_M: \CC \rightarrow \CM$ is the functor sending $X \in \CC$ to $X \ast M$. 
\end{defn} 

The notion of action internal homs is a generalization of that of internal hom of a monoidal category. The latter one corresponds to the regular action of a monoidal category on itself.  
For brevity, we will refer to both as `internal hom'. 

\begin{rema}  
(i) The assignment $f \mapsto \underline{f}$ in Definition~\ref{def:action-internal-hom} gives a natural isomorphism 
\be  \label{eq:def-int-hom-cl}
\Hom_\CM( X \ast M, N)\,\, \cong \,\, \Hom_\CC(X, [M, N]). 
\ee
This shows the equivalence of our formulation to that used in \cite{ostrik}, where the internal hom functor $[M, -]: \CM \to \CC$ was defined as the right adjoint of $-\ast M: \CC \to \CM$. 
\\[.5em]
(ii) Let $\Vect_k$ be the monoidal category of (not necessarily finite dimensional) vector spaces over a field $k$. There is a canonical isomorphism $\Hom(U \otimes V , W) \cong \Hom(U, \Hom(V,W))$. Thus, if we consider $\Vect_k$ as a module category over itself, the internal hom is just the usual hom, $[U,V] = \Hom(U,V)$.
\\[.5em]
(iii) 
Internal homs may or may not exist. For example, consider a monoidal category $\CC$ as a module category over itself, and let $\mathcal{G}$ be the full monoidal subcategory of $\CC$ of invertible objects (or even just the tensor unit itself). Then $\CC$ is also a $\mathcal{G}$-module, but may not have internal homs in $\mathcal{G}$. For example, this happens for $\Vect_k$ where $\mathcal{G}$ consists of 1-dimensional vector spaces.
\end{rema}

\begin{defn}  \label{def:C-closed}
A subcategory $\CN\subset\CM$ of a $\CC$-module $\CM$ is called $\CC$-{\em closed} iff the internal hom $[M, N]$ exists for all $M, N \in \CN$. ($\CN$ is not itself required to be a $\CC$-module; it is understood that the condition \eqref{eq:underline-f} -- or \eqref{eq:def-int-hom-cl} -- is to be applied to all morphisms in $\CM$, not only to those in $\CN$.)
\end{defn}

\begin{rema}
(i) The above definition allows to select a sub-class of objects in $\CM$ for which internal Homs exist, the condition of being $\CC$-closed is independent of the morphism sets in $\CN$ (but of course not of those of $\CM$ and $\CC$). This is helpful because we will be interested mainly in module categories that in turn arise as categories of $\CC$-module functors between given $\CC$-modules, and we may want to single out certain sub-categories of functors. For example, if $\CG$ denotes the sub-category of invertible such functors with invertible natural transformations, we can say `assume $\CG$ is closed' rather than `assume that the class of objects in $\CG$ is a closed sub-class' or `assume that the full subcategory containing the same objects as $\CG$ is closed'. Note that in this example, $\CG$ will typically not itself be a module category.
\\[.5em]
(ii) The above example also shows that for a $\CC$-closed $\CC$-module $\CM$, it is possible to have a 
monoidal subcategory
$\CC'$ of $\CC$ such that $\CM$ is not $\CC'$-closed. 
\end{rema}

\begin{conv} \label{conv:f1-noassoc}
In order not to make the commutative diagrams unnecessarily large, we adopt the following conventions.\\
(i) We write $f1$ as an abbreviation for $f\ast \id$ or $f\otimes \id$, and $fg$ for $f\ast g$ or $f\otimes g$. 
\\
(ii) We will usually not spell out the associator and unit isomorphisms in monoidal categories or module categories. For example, given a map $f: V \ast M \rightarrow N$ we may write $(U \otimes V) \ast M \xrightarrow{~1f~} U \ast N$ instead of $(U \otimes V) \ast M \xrightarrow{a_{U,V,M}^{-1}} U \ast  (V \ast M) \xrightarrow{~\id_U \ast f~} U \ast N$.
\end{conv}

\begin{lemma}  \label{lemma:int-hom-bifunctor}  
If $\CM$ is a $\CC$-closed subcategory of a $\CC$-module, the internal hom $[-, -]$ gives a functor from the product category $\CM^\op \times \CM$ into $\CC$. 
\end{lemma}
\pf
We start by construction two families of functors, $[M,-] : \CM \rightarrow \CC$ and $[-,M] : \CM^\op \to \CC$ for all $M \in \CM$.
On objects both are given by the internal hom. 

Let $f: N \to N'$ and $g: M \to M'$ be morphisms in $\CM$. We define $[M,f]: [M, N] \to [M, N']$ and $[g,N]: [M', N] \to [M, N]$ via the universal property,
\be  \label{eq:def-fast-gast}
\raisebox{2em}{\xymatrix{
[M, N] \ast M \ar[d]^{\ev_M} \ar@{-->}[rr]^{\exists !\, [M,f] 1} && [M, N'] \ast M \ar[d]^{\ev_M} \\
N \ar[rr]^f  && N' 
}}
\quad\quad
\raisebox{2em}{\xymatrix{
[M', N] \ast M \ar[d]^{1g}   \ar@{-->}[rr]^{\exists !\, [g,N] 1} && [M, N] \ast M \ar[d]^{\ev_M} \\
[M', N] \ast M' \ar[rr]^{\ev_{M'}} && N 
}}
\ee
Using the uniqueness requirement in the universal property, one checks that $[M,\id_N]=\id_{[M,N]}$, $[\id_M, N]=\id_{[M,N]}$, 
$[M, f]\circ [M, f'] = [M, f\circ f']$ and $[g', N]\circ [g, N] = [g \circ g', N]$ where $f': N'\to N''$ and $g':M'\to M''$. 

Next we will show that the two functors commute in the sense that $[g,N'] \circ [M',f] = [M,f] \circ [g,N]$, where both are maps $[M',N] \rightarrow [M,N']$. Consider the  diagram
\be
\raisebox{5em}{\xymatrix{
[M', N] \ast M  \ar[ddd]_{[M',f]1} \ar[dr]^{1g} \ar[rrr]^{[g,N]1}  & & & [M,N]\ast M \ar[ddd]_{[M,f]1} \ar[ld]_{\ev_{M}} \\
& [M',N] \ast M' \ar[r]^{~~~~\ev_{M'}}  \ar[d]_{[M',f]1} & N \ar[d]_f \\
& [M',N'] \ast M' \ar[r]^{~~~~\ev_{M'}}  & N'  \\
[M', N'] \ast M  \ar[ur]^{1g} \ar[rrr]^{[g,N']1}  & & & [M,N']\ast M \ar[ul]_{\ev_{M}} \ .
}}
\ee
All but the leftmost square commute by the commutativity of the two diagrams in (\ref{eq:def-fast-gast}). The commutativity of the leftmost square follows from the statement that $\ast$ is a functor $\CC \times \CM \rightarrow \CM$. Thus the two maps $[M', N] \ast M \rightarrow N'$ defined by the outer square (composed with $\ev_{M}$) are equal and the universal property implies $[g,N'] \circ [M',f] = [M,f] \circ [g,N]$.

Altogether this shows that the two families of functors $[M,-]$ and $[-,M]$ define a functor $\CM^\op \times \CM \rightarrow \CC$ (cf.\ \cite[Sect.\,II.3, Prop.\,1]{catforwormath}).
\epf

In $\Vect_k$ there is a canonical homomorphism $X \otimes \Hom(L,M) \rightarrow \Hom(L,X\otimes M)$ taking $x \otimes f$ to the map $l \mapsto x \otimes f(l)$. This map is an isomorphism if $X$ is finite dimensional (choose a basis of $X$), but in general it is not (take $X$ and $L$ infinite dimensional and $M=k$). General internal homs behave similarly. For $X\in \CC$ and $L, M\in \CM$, we have a map 
\be
  \gamma_X: X\otimes [L, M] \to [L, X * M]
\ee  
defined by the universal property of the internal hom: 
\be  \label{diag:def-beta}
\raisebox{2em}{\xymatrix{
(X\otimes [L, M]) \ast L \ar[rd]_{1\ev_L} \ar@{-->}[rr]^{\exists ! \, \gamma_X 1} 
&& [L, X\ast M]\ast L \ar[ld]^{\ev_L}  \\
  & X\ast M
}}
\ee
Here the arrow labelled $1\,\ev_L$ makes sense in view of Convention~\ref{conv:f1-noassoc}.
For $f: X\to X'$, the universal property implies commutativity of the diagram
\be  \label{diag:X-X'-beta}
\raisebox{2em}{\xymatrix{
X \otimes [L, M] \ar[r]^{\gamma_X} \ar[d]_{f1} & [L, X\ast M] \ar[d]^{[L, f1]}  \\
X' \otimes [L, M] \ar[r]^{\gamma_{X'}} & [L, X'\ast M]\, ,
}}
\ee
so that $\gamma$ is a natural transformation from $(-) \otimes [L,M] \rightarrow [L,(-)\ast M]$, both of which are endofunctors of $\CC$.

An object of $\Vect_k$ has a right (or left) dual iff it is finite-dimensional. For general internal homs we have the following

\begin{lemma} \label{lemma:beta-iso}
Let $\CM$ be a $\CC$-module, and let $X\in\CC$ have a right dual $X^\vee$. Suppose that the internal homs $[L,M]$, $[L,X\ast M]$, $[L,X^\vee \ast M]$, $[L,(X^\vee \otimes X)\ast M]$ exist.
Then $\gamma_X : X \otimes [L,M] \to [L,X \ast M]$ is an isomorphism. 
\end{lemma}
\pf
The following identities for $\gamma_{(-)}$ follow immediately from the universal property:
\be \label{eq:gamma-prop}
\gamma_{1_\CC} =\id_{[L,M]},
\quad  \quad
\gamma_{X\otimes Y}=\gamma_X \circ (\id_X \otimes \gamma_Y).
\ee
Denote by $1_\CC \xrightarrow{b_X} X\otimes X^\vee$ and $X^\vee \otimes X \xrightarrow{d_X} 1_\CC$ the duality maps. We have the commutative diagram
\be  \label{diag:bar-beta-1}
\raisebox{4em}{\xymatrix{
[L, X\ast M] \ast L \ar[rr]^{\ev_L} \ar[d]_{b_X11} && X\ast M \ar[d]_{b_X11} \ar@{=}[rrd]  \\
(X\otimes X^\vee \otimes [L, X\ast M])\ast L \ar[rr]^{~~11\ev_L} \ar[d]_{1\gamma_{X^\vee}1} &&
(X\otimes X^\vee \otimes X)\ast M  \ar[rr]^{~~~1d_X1} & &  X\ast M  \\
(X\otimes [L, (X^\vee {\otimes} X) \ast M]) \ast L \ar[rr]^{~~~1[L,d_X1] 1} \ar[rru]^{1\ev_L} 
& & (X\otimes [L, M])\ast L \ar[urr]_{1\ev_L} & & .  
}}
\ee
Thus, if we define the map $\bar{\gamma}_X: [L, X\ast M] \to X\otimes [L, M]$ as
$\bar{\gamma}_X := (1[L,d_X1]) \circ  (1\gamma_{X^\vee}) \circ (b_X 1)$,  the above
commutative diagram implies commutativity of
\be \label{diag:bar-beta-2}
\raisebox{2em}{\xymatrix{
& [L, X\ast M] \ast L \ar[rd]^{\ev_L} \ar[ld]_{\bar{\gamma}_X1}  \\
(X\otimes [L, M]) \ast L \ar[rr]^{1\ev_L} & & X\ast M \ .
}}
\ee
The universal property of the internal hom immediately gives $\gamma_X \circ \bar{\gamma}_X = \id_{[L, X\otimes M]}$. For the composition in the opposite order, we compute (again not writing out associators and unit isomorphisms)
\bea
\bar{\gamma}_X  \circ \gamma_X 
&\overset{(1)}{=}& (1[L,d_X1]) \circ  (1\gamma_{X^\vee}) \circ (b_X 1) \circ \gamma_X \nn
&\overset{(2)}{=}& (1[L,d_X1]) \circ  (1\gamma_{X^\vee}) \circ (11\gamma_X) \circ (b_X11) \nn
&\overset{(3)}{=}& (1[L,d_X1]) \circ (1\gamma_{X^\vee\otimes X}) \circ (b_X11) \nn
&\overset{(4)}{=}& (1\gamma_{1_\CC}) \circ (1d_X1) \circ (b_X11) 
~\overset{(5)}{=}~ \id_{X\otimes [L, M]}  \ .
\eea
Here step 1 is the definition of $\bar\gamma_X$, step 2 is functoriality of the tensor product of $\CC$,
step 3 uses \eqref{eq:gamma-prop}, step 4 is the naturality \eqref{diag:X-X'-beta} of $\gamma$, and step 5 follows from $\gamma_{1_\CC}=\id$ and the properties of duality maps.
\epf

As we have seen, in $\Vect_k$ internal homs are just homs, which can be composed. Analogously, for general internal homs the universal property allows to define a composition morphism 
\be \label{eq:compos}
\comp_M \equiv \comp_{N,M,L}  : \,\, [M, L] \otimes [N, M] \rightarrow [N, L]
\ee
in terms of the commutative diagram
\be \label{diag:def-compos}
\raisebox{2em}{\xymatrix{
([M, L] \otimes [N, M]) \ast N \ar[d]^{1\ev_N} \ar@{-->}[rr]^(.6){\exists ! \, \comp_M 1} 
 && [N,L]\ast N \ar[d]^{\ev_N} \\
[M, L]\ast M  \ar[rr]^{\ev_M} && L \ .
}}
\ee

\begin{lemma}  \label{lem:composition-properties}
Let $\CM$ be a $\CC$-module, and let $K,K',L,L',M,N \in \CM$.
\\[.5em]
(i) If all internal homs in \eqref{diag:beta-Nev=ev} exist, the composition morphism $\comp_M$ can be factorized as
\be \label{diag:beta-Nev=ev}
\xymatrix{
& [N, [M, L]\ast M] \ar[dr]^{[N, \ev_M]} & \\
[M, L] \otimes [N, M] \ar[ur]^{\gamma_{[M, L]}} \ar[rr]^{\comp_M} & & [N, L]
}
\ee
(ii) If all internal homs in \eqref{eq:asso-int-hom} exist, the composition morphisms are associative in the sense that the diagram
\be \label{eq:asso-int-hom}
\xymatrix{
([M, N] \otimes [K, M]) \otimes [L, K]  \ar[rr]^{\id} \ar[d]_{\comp_M 1} & &  [M, N] \otimes ([K, M] \otimes [L, K])  \ar[d]^{1\comp_K} \\
[K,N] \otimes [L, K] \ar[r]^{\comp_K} & [L,N] & [M,N] \otimes [L,M] \ar[l]_{\comp_M} 
}
\ee
commutes.
\\[.5em]
(iii) Suppose all internal homs in \eqref{eq:hom-comp-compat} exist. Then for $f : M \to M'$ and $g : K' \to K$ we have the commuting squares
\be \label{eq:hom-comp-compat}
\xymatrix{
[L,M]{\otimes}[K,L] \ar[rr]^{[L,f]\,1~} \ar[d]^{\comp_L}
  && [L,M']{\otimes}[K,L] \ar[d]^{\comp_L} \\
[K,M] \ar[rr]^{[K,f]} && [K,M'] \ ,
}
 \hspace{1em}
\xymatrix{
[L,M]{\otimes}[K,L] \ar[rr]^{1\,[g,L]~} \ar[d]^{\comp_L}
  && [L,M]{\otimes}[K',L] \ar[d]^{\comp_L} \\
[K,M] \ar[rr]^{[g,M]} && [K',M]\, .
}
\ee
\end{lemma}

Note that, while the proof of Lemma~\ref{lemma:beta-iso} required more internal homs to exist than just the ones appearing as source and target of the map $\gamma_X : X \otimes [L,M] \to [L,X \ast M]$, the proof of Lemma~\ref{lem:composition-properties} only needs the existence of the internal homs appearing explicitly in the diagrams \eqref{diag:beta-Nev=ev}, \eqref{eq:asso-int-hom}, and \eqref{eq:hom-comp-compat}.

\pf
(i) By the universal property of internal homs it is enough to check that the two maps $\ev_N \circ (\comp_M1)$ and $\ev_N \circ ([N,\ev_M]1) \circ (\gamma_{[M,L]}1)$ from $([M,L]\otimes [N, M])\ast N$ to $L$ agree. This in turn can be quickly verified by inserting the definitions \eqref{eq:def-fast-gast}, \eqref{diag:def-beta} and \eqref{diag:def-compos} of $[N,-]$, $\gamma$ and $\comp$.
\\[.5em]
(ii) It suffices to check that the two morphisms from $(([M, N] \otimes [K, M]) \otimes [L, K]) \ast L$ to $N$ given by
$a = \ev_L \circ (\comp_M 1) \circ (1 \comp_K 1) $ and $b = \ev_L \circ (\comp_K 1) \circ (\comp_M 11)$ agree. This can be checked by substituting the definition \eqref{diag:def-compos} twice into $a$, while for $b$ one needs to use \eqref{diag:beta-Nev=ev} to replace $(\comp_M 11)$. It is then straightforward to verify $a=b$ from the definition of $\gamma_{[M,N]}$ and $[L,\ev_M]$.
\\[.5em]
(iii) The first identity follows from the universal property together with the commutativity of (we omit `$\otimes$', `$\ast$' and brackets)
\be
\xymatrix{
[L,M][K,L]K \ar[rr]^{[L,f]\,11} \ar[dr]^{1\,\ev_{K}} \ar[dd]^{\comp_L 1} 
  && [L,M'][K,L]K \ar[d]_{1\,\ev_K} \ar[dr]^{\comp_L 1}
  \\
  & [L,M]L \ar[r]^{[L,f]\,1} \ar[d]_{\ev_L}
  & [L,M']L \ar[d]_{\ev_L} & [K,M']K \ar[dl]_{\ev_K} \\
[K,M]K \ar[r]^{\ev_K} \ar[dr]^{[K,f]\,1}
 & M \ar[r]^{f} 
 & M' 
 \\
& [K,M']K \ar[ru]^{\ev_K}
}
\ee
The commutativity of the individual cells holds by definition of $[K,f]$, $[L,f]$ and $\comp_L$. The second square can be checked analogously. 
\epf

Recall the notation $\underline{f}$ from \eqref{eq:underline-f}. From $1_\CC \ast M \xrightarrow{\id_M} M$ (using Conv.\,\ref{conv:f1-noassoc}) we get the map $\underline{\id_M}: 1_\CC \rightarrow [M, M]$, and it is straightforward to check that it makes the following diagram commute,
\be  \label{eq:unit-int-hom}
\xymatrix{
 [N, N] \otimes [M, N] \ar[r]^{\hspace{0.8cm} \comp_N} &   [M, N] & [M, N] \otimes [M, M] \ar[l]_{\hspace{-0.8cm} \comp_M} \\
1_\CC \otimes [M,N] \ar[u]^{\underline{\id_N}1} \ar[ur]_{\id_N} & & [M, N]\otimes 1_\CC  \ar[u]_{1\,\underline{\id_M}} \ar[ul]^{\id_M}\ .
}
\ee

\begin{rema} \label{rema:iso-int-hom}  
(i) An immediate consequence of the above observations is that if the internal hom $[M,M]$ exists, 
$([M, M], \comp_M, \underline{\id_M})$ is an algebra in $\CC$. 
Analogously, $[N, M]$ is a left-$[M,M]$ and right-$[N,N]$ module. 
\\[.5em]
(ii) If $f: M \rightarrow N$ and $f': M' \rightarrow N'$ are isomorphisms in $\CM$, then by Lemma~\ref{lemma:int-hom-bifunctor} $[M, M'] \xrightarrow{[M,f']} [M, N']$ and $[M, N'] \xrightarrow{[f^{-1}, N']} [N, N']$ are isomorphisms in $\CC$. One can also show that the morphism $[f^{-1}, f] : [M, M] \to [N, N]$ is an algebra isomorphism, but we will not need this result. 
\end{rema}

\subsection{Enriched categories}

It is possible to summarize above structures using enriched categories (see \cite{kelly} for definitions). 
We define a category $\CM^\CC$ whose objects are given by objects in $\CM$ and whose morphisms are internal homs, $\Hom_{\CM^\CC}(M, N):= [M, N]$. The compositions is given by the composition morphism \eqref{eq:compos} and the identity is $\underline{\id_M}$. The commutative diagrams \eqref{eq:asso-int-hom} and \eqref{eq:unit-int-hom} simply say that $\CM^\CC$ is a category enriched over $\CC$, or a $\CC$-category for short.

\medskip

Let $\CM$ and $\CM'$ be $\CC$-closed subcategories of two $\CC$-modules $\tilde\CM$ and $\tilde\CM'$, respectively. Let $F: \tilde\CM \rightarrow \tilde\CM'$ be a $\CC$-module functor (see Appendix~\ref{app:module-cat} for definition and conventions), which maps $\CM$ to $\CM'$. 

\medskip

There is a canonical map $[F]_{[M,N]}: [M, N] \rightarrow [F(M), F(N)]$ given by 
\be
  [F]_{[M,N]}=\underline{F(\ev_M) \circ (F_{[M,N],M}^{(2)})^{-1}} \ ,
\ee
i.e.
\be  \label{eq:def-F-MN}
\xymatrix{
[M, N] \ast F(M) 
\ar@{-->}[rr]^(.4){\exists ! \, [F]_{[M,N]}1} 
\ar[d]^{(F_{[M,N],M}^{(2)})^{-1}} 
&& [F(M), F(N)] \ast F(M) \ar[d]^{\ev_{F(M)}} \\
 F([M, N] \ast M) \ar[rr]^{F(\ev_M)} & &  F(N)\ .
}
\ee
Sometimes we abbreviate $[F]_{[M,N]} \equiv [F]$. It is clear that when the functor $F$ is the identity, then $[F]_{[M,N]}$ is the identity map on $[M,N]$.

\begin{lemma}  \label{lemma:[F]-funct-1}
$[F]_{[-,-]}$ satisfies the following functorial properties. For $\phi: K \to L$ and $\psi: M \to N$, the following diagrams
\be  \label{diag:[F]-fun-property}
\xymatrix{
[L, M]  \ar[r]^{\hspace{-0.7cm}[F]}  \ar[d]_{[L, \psi]} & [F(L), F(M)] \ar[d]^{[F(L), F(\psi)]} \\
[L, N]  \ar[r]^{\hspace{-0.7cm}[F]}  & [F(L), F(N)] 
}
\quad \mbox{and} \quad
\xymatrix{
[L,M] \ar[r]^{\hspace{-0.7cm}[F]} \ar[d]_{[\phi,M]} & [F(L), F(M)] \ar[d]^{[F(\phi), F(M)]} \\
[K,M] \ar[r]^{\hspace{-0.7cm}[F]} & [F(K), F(M)] 
}
\ee
are commutative. 
\end{lemma}
\pf
To prove the first property in (\ref{diag:[F]-fun-property}), we consider the following diagram (omitting $\ast$):
$$
\xymatrix{
[L,M]F(L) \ar[rrr]^{[F] 1} \ar[ddd]_{(F^{(2)})^{-1}} \ar[dr]^{[L,\psi]1} & & & 
[F(L),F(M)] F(L) \ar[dl]_{[F(L),F(\psi)]1} \ar[ddd]_{\ev}  \\
& [L,N] F(L) \ar[r]^{\hspace{-0.5cm}[F]1} \ar[d]_{(F^{(2)})^{-1}}  & [F(L),F(N)] F(L) \ar[d]^{\ev} & \\
& F([L,N] L) \ar[r]^{F(\ev)} & F(N) & \\
F([L,M] L) \ar[rrr]^{F(\ev)} \ar[ur]^{F([L,\psi]1)}  & & & F(M) \ar[ul]_{F(\psi)}~.
}
$$
By (\ref{eq:def-F-MN}) and the first commutative diagram in (\ref{eq:def-fast-gast}), all subdiagrams, including the outer square, except the subdiagram on the top are commutative. Hence, we have
$\ev \circ ([F]1) \circ ([L,\psi]1) = \ev \circ ([F(L),F(\psi)]1) \circ ([F]1)$.  The first property in (\ref{diag:[F]-fun-property}) follows by the universal property of the internal hom.

The proof for the second property in (\ref{diag:[F]-fun-property}) is entirely similar, we only mention the difference. Consider the following diagram (omitting $\ast$): 
$$
\xymatrix{
[L,M] F(K) \ar[rrr]^{[F] 1} \ar[ddd]^{(F^{(2)})^{-1}} \ar[dr]^{[\phi,M]1} & & & 
[F(L),F(M)] F(K) \ar[dl]_{[F(\phi),F(M)]1} \ar[dd]_{1F(\phi)}  \\
& [K,M] F(K) \ar[r]^{\hspace{-0.5cm}[F]1} \ar[d]_{(F^{(2)})^{-1}}  & [F(K),F(M)] F(K) \ar[d]^{\ev} &   \\
& F([K,M] K) \ar[r]^{F(\ev)} & F(M) &  [F(L), F(M)] F(L) \ar[l]_{\ev} \\
F([L,M] K) \ar[rr]^{F(1\phi)} \ar[ur]^{F([\phi,M]1)}  & & F([L,M] L) \ar[u]_{F(\ev)} 
\ar[r]^{F^{(2)}} &  [L,M] F(L)~. \ar[u]^{[F]1}
}
$$
Since $F^{(2)} \circ F(1\phi) (F^{(2)})^{-1} = 1F(\phi)$, it is easy to see that 
the outer square is commutative. Then running the same argument as the proof of the first property, we obtain the proof of the second property in (\ref{diag:[F]-fun-property}). 
\epf

\begin{lemma} \label{lemma:F-C-functor}  
$[F]_{[-,-]}$ defines a $\CC$-functor $F^\CC: \CM^\CC \rightarrow \CM'\,^{\CC}$. Equivalently, the two diagrams
\be \label{eq:unit-C-functor}
\xymatrix{
 & 1_\CC \ar[dl]_{\underline{\id_M}} \ar[dr]^{\underline{\id_{F(M)}}}    \\
 [M, M] \ar[rr]^{[F]_{[M,M]}}   & & [F(M), F(M)]   }
\ee
\be \label{eq:C-functor}
\xymatrix{
[N, P] \otimes [M, N]  \ar[d]^{[F] \otimes [F]} 
   \ar[rr]^{\comp_N}
& &  [M, P] \ar[d]_{[F]} \\
[F(N), F(P)] \otimes [F(M), F(N)] 
  \ar[rr]^{\hspace{3em}\comp_{F(N)}} 
& & [F(M), F(P)] 
}
\ee
commute for all $M, N, P\in \CM$. In particular, the morphism $[F]: [M, M] \rightarrow [F(M), F(M)]$ is an algebra homomorphism. 
\end{lemma}
\pf
The commutativity of \eqref{eq:unit-C-functor} follows from the commutative diagram 
\be
\xymatrix{ 
1_\CC \ast F(M) \ar[rr]^{\underline{\id_M}1} \ar[drr]^{F(\underline{\id_M}1)} \ar[ddrr]_{\id}
	& & [M, M] \ast F(M) \ar[rr]^{\hspace{-0.5cm}[F]_{[M,M]}1}  \ar[d]^{(F^{(2)})^{-1}} 
	& &   [F(M), F(M)] \ast F(M) \ar[ddll]^{\ev_{F(M)}} \\
  && F([M,M]\ast M) \ar[d]^{F(\ev_M)}  \\
  && F(M)  }
\ee
together with the universal property of $[F(M), F(M)]$. To check commutativity of \eqref{eq:C-functor}, consider the diagram
(we omit all $\otimes$ and $\ast$ and brackets)
$$
{\small
\xymatrix@C=1em{
[N,P]\,[M,N]\,FM \ar[rr]^{1[F]1} \ar[dr]^{1(F^{(2)})^{-1}} \ar[dd]_{\comp_N 1}
  && [N,P]\,[FM,FN]\,FM \ar[rr]^{[F]11} \ar[d]^{1\, \ev_{FM}}
  && [FN,FP]\,[FM,FN]\,FM \ar[dl]_{1\,\ev_{FM}} \ar[ddd]_{\comp_{FN}1}
\\
& \hspace{-5em}[N,P]\,F([M,N]M) \ar[r]^{1\,F(\ev_M)} \ar[d]^{F(\comp_N) \circ (F^{(2)})^{-1}}
   & [N,P]\,FN \ar[r]^{~[F]1}
   & [FN,FP]\,FN\hspace{-4em} \ar[ddl]^{\ev_{FN}}
\\
[M,P]\,FM \ar[r]^{(F^{(2)})^{-1}} \ar[d]_{[F]1}
  &  F([M,P]\,M) \ar[dr]^{F(\ev_M)}
\\
[FM,FP]\,FM \ar[rr]^{\ev_{FM}}
  & & FP 
  & & [FM,FP]\,FM \ar[ll]_{\ev_{FM}}  \ .
}  }
$$
Here all but the central pentagon are immediate consequences of the definition of $\comp$, $F_{[-,-]}$ and of the fact that $F$ is a $\CC$-module functor. The central pentagon is also easily checked by substituting these definitions.
\epf

\begin{lemma} \label{lemma:G-F=GF}  
Let $\CL$, $\CM$ and $\CN$ be $\CC$-modules and $F: \CL \rightarrow \CM$ and $G: \CM \rightarrow \CN$ be $\CC$-module functors. If the three internal homs in the following diagram exist, then the diagram commutes,
\be
\xymatrix{ 
[M, M']  \ar[r]^(.4){[F]}  \ar[rd]_(.4){[G\circ F]}  &  [F(M), F(M')]  \ar[d]^{[G]}  \\
&  [GF(M),  GF(M')] \ .  
}
\ee
\end{lemma}

\begin{proof}
As usual, this can be seen by composing with $(-)\ast G(F(M))$ and showing the corresponding identity of morphisms
$[M, M'] \ast G(F(M)) \to G(F(M'))$. The latter follows in a straightforward way when inserting the definitions of
$[F]$, $[G]$ and $[G\circ F]$ from \eqref{eq:def-F-MN}.
\end{proof}

Analogously to the definition of $F$ in the beginning of this subsection, let $G : \tilde\CM \to \tilde\CM'$ be a another $\CC$-module functor which maps the $\CC$-closed subcategory $\CM$ to the $\CC$-closed subcategory $\CM'$.
Given a $\CC$-module natural
transformation $F \xrightarrow{\phi} G$, the map $1_\CC \ast F(M) \xrightarrow{\phi_M} G(M)$ induces a morphism $\underline{\phi_M}: 1_\CC \rightarrow [F(M), G(M)]$ as in \eqref{eq:underline-f}. There are two equivalent ways to realize $\underline{\phi_M}$, namely
\bea \label{eq:phi_M-2expressions}
\underline{\phi_M}&=& \!\Big(\, 1_\CC \xrightarrow{~\underline{\id_{F(M)}}~} [F(M), F(M)]  \xrightarrow{[F(M), \phi_M]} [F(M), G(M)] \,\Big) \ , \nn
\underline{\phi_M}&=& \!\Big(\, 1_\CC \xrightarrow{~\underline{\id_{G(M)}}~} [G(M), G(M)] \xrightarrow{[\phi_M, G(M)]} [F(M), G(M)] \,\Big) \ . 
\eea
Both identities can be checked by composing with $(-) \ast F(M)$ and substituting the definitions.

\begin{lemma}  \label{lem:GF-nat-commute}
The following diagram commutes,
\be  \label{diag:GF-nat-commute}
\xymatrix{
[M, N] \ar[rr]^{[F]} \ar[d]_{[G]}
  &&  [F(M), F(N)] \ar[d]^{[F(M), \phi_N]}   \\
[G(M), G(N)] \ar[rr]^{[\phi_M, G(N)]} 
  && [F(M), G(N)] \ .
}
\ee
\end{lemma}

\pf
Composing with $(-)\ast F(M)$ and inserting the definition of $[\phi_M, G(N)]$ and $[F(M), \phi_N]$, one finds that
commutativity of the above square follows from commutativity of
\be
\xymatrix{
& [F(M), F(N)] \ast F(M) \ar[rd]^{~\ev_{G(M)}}
\\
[M, N] \ast F(M) \ar[r]_{(F_{[M,N],M}^{(2)})^{-1}}   \ar[d]_{1\,\phi_M} \ar[ru]^{[F]1} 
  & F([M,N]\ast M) \ar[d]^{\phi_{[M,N]\ast M}} \ar[r]_{F(\ev_M)} 
  & F(N) \ar[d]^{\phi_N}  
\\
[M,N]\ast G(M) \ar[rd]_{[G]1} \ar[r]^{(G_{[M,N],M}^{(2)})^{-1}} 
  & G([M,N]\ast M) \ar[r]^{G(\ev_M)} 
  & G(N) 
\\
& [G(M), G(N)] \ast G(M) \ . \ar[ur]_{~\ev_{G(M)}}
}
\ee
The commutativity of the two triangles amounts to the definition of $[F]$ and $[G]$. The left square is condition \eqref{eq:c-mod-nat-xfer-condition} on a $\CC$-natural transformation, and the right square is the naturalness of $\phi$.
\epf

\begin{lemma}  
$\underline{\phi_M}$ defines a $\CC$-natural transformation  $\phi^\CC: F^\CC \rightarrow G^\CC$. Equivalently,
the diagram 
\be  \label{diag:phi-C}
\xymatrix@=1.5em{
& 1_\CC \otimes [M, N] \ar[rr]^(.35){\underline{\phi_N}\otimes [F]} && [F(N), G(N)] \otimes [F(M), F(N)] \ar[dr]_{C_{F(N)}}    \\
[M, N] \ar[ur]^{\cong} \ar[dr]_{\cong} & & & & \hspace{-2em}[F(M), G(N)] \\
& [M,N]\otimes 1_\CC \ar[rr]^(.35){[G]\otimes \underline{\phi_M}}  && [G(M), G(N)] \otimes [F(M), G(M)] \ar[ur]^{C_{G(M)}} 
}
\ee
is commutative. 
\end{lemma}
\pf
Making use of the identities \eqref{eq:phi_M-2expressions}, we see that the lemma will follow once we establish commutativity of all cells in the diagram (we omit `$\otimes$', `$\ast$' and brackets)
\be
{\small
\xymatrix{
[M,N] \ar[r]^{[F]} \ar[d]^{[G]}
  & [FM,FN] \ar[r]^{\underline{\id_{FN}}1} \ar@{=}[dr]
  &  [FN,FN][FM,FN] \ar[d]_{\comp_{FN}} \ar[dr]^{[FN,\phi_N]\,1} 
  \\
[GM,GN] \ar[d]^{1\,\underline{\id_{GM}}} \ar@{=}[dr]
 & & [FM,FN] \ar[d]_{[FM,\phi_N]}
 & \hspace{-4em}[FN,GN] [FM,FN] \ar[dl]^{\comp_{FN}}
 \\
[GM,GN] [GM,GM] \ar[r]^{\comp_{GM}} \ar[dr]_{1\,[\phi_M,GM]}
 & [GM,GN] \ar[r]^{[\phi_M,GN]}
 & [FM,GN]
 \\
& [GM,GN] [FM,GM] \ar[ur]_{\comp_{GM}}
}}
\ee
The two triangles amount to the identity \eqref{eq:unit-int-hom} and the two outer squares are special cases of Lemma~\ref{lem:composition-properties}\,(iii). The central square (which looks like a hexagon) commutes by Lemma~\ref{lem:GF-nat-commute}.
\epf

In the case that the $\CC$-closed subcategory $\CM$ is already the entire $\CC$-module,
the above results can be summarized into the following Theorem.
\begin{thm} \label{thm:2-functor}  
Let $\CC$ be a monoidal category. 
The assignment $\CM \mapsto \CM^\CC$, $F \mapsto F^\CC$ and $\phi \mapsto \phi^\CC$ defines a functor $(-)^\CC$ from the 2-category of $\CC$-closed $\CC$-modules into the 2-category of $\CC$-categories.
\end{thm}

\section{Module categories and the monoidal center}  \label{sec:full-center}

In this section, we recall the construction of the full center for a module category over a monoidal category $\CC$ and the centralizer of a $\CC$-module functor. We will also prove some basic properties of them. In particular, the commutativity proved in Proposition\,\ref{prop:comm} will be important for the constructions in later sections.

\subsection{The full center of a module category}\label{sec:full-center-module}

The full center of an algebra $A=(A, m_A, \iota_A)$ in a monoidal category $\CC$ is a commutative algebra in the monoidal center of that category. Instead of the algebra $A$ one can use its category of (right, say) modules $\CC_A$ as the input. Note that $\CC_A$ is an example of a (left) $\CC$-module. More generally, the full center can be defined for arbitrary $\CC$-modules via a universal property. 

To start with, let us recall the definition of the monoidal center (see \cite[Ex.\,2.3]{js} or \cite[Sect.\,VIII.4]{kassel}).

\begin{defn}
Let $\CC$ be a monoidal category. A {\em half-braiding} for an object $Z \in \CC$ is a natural isomorphism $z : Z \otimes (-) \to (-) \otimes Z$ such that $z_{1_\CC} = \id_Z$ and $z_{U \otimes V} = (\id_U \otimes z_V) \circ (z_U \otimes \id_V)$ for all $U,V \in \CC$ 
(recall Convention \ref{conv:f1-noassoc}\,(ii)).
The {\em monoidal center} $\CZ(\CC)$ of $\CC$ is the category where
\begin{itemize}
\item[-] {\bf objects} are pairs $(Z,z)$ where $Z \in \CC$ and $z$ is a half braiding for $Z$,
\item[-] {\bf morphisms} $f : (Y,y) \to (Z,z)$ are morphisms $f : Y \rightarrow Z$ in $\CC$ such that for all $U \in \CC$: $z_U \circ (f \otimes \id_U) = (\id_U \otimes f) \circ y_U$,
\item[-] {\bf composition and identities} are those of $\CC$.
\end{itemize}
\end{defn}

The monoidal center has two important properties. Firstly, there is a natural functor $\CZ(\CC) \to \CC$, namely the forgetful functor taking $(Z,z)$ to $Z$. Secondly, $\CZ(\CC)$ is monoidal and {\em braided} (see \cite[Thm.\,VIII.4.2]{kassel} for a proof). The tensor product and braiding isomorphisms of $\CZ(\CC)$ are
\be
(Y,y) \otimes (Z,z) = (Y \otimes Z , y|z)
\quad \text{where } (y|z)_U = (y_U \otimes \id_Z) \circ (\id_Y \otimes z_U)
\ee
and
\be
c_{(Y,y),(Z,z)} = y_Z : (Y,y) \otimes (Z,z) \to (Z,z) \otimes (Y,y) \ .
\ee

Given a left $\CC$-module $\CM$, {\it $\alpha$-induction} \cite{bek,ostrik} is a monoidal functor $\alpha$ from $\CZ(\CC)$ to the category $\CC_\CM^\ast$ of $\CC$-module endofunctors of $\CM$. The functor sends an object $(Z, z) \in \CZ$ to the functor 
\be
  \alpha(Z): \,\,  \CM \to \CM,\quad \alpha(Z)(M) = Z \ast M\ .
\ee  
with the $\CC$-module structure (cf.\ Definition~\ref{def:C-module-functor})
\be \label{diag:alpha-ind}
\raisebox{2em}{\xymatrix{
\alpha(Z)(X\ast M) \ar@{=}[d] \ar@{-->}[rrr]^{\alpha(Z)^{(2)}_{X,M}} &&& X\ast\alpha(Z)(M) \ar@{=}[d] \\
Z\ast(X\ast M) \ar[r] & (Z\otimes X)\ast M \ar[r]^{z_X1} & (X\otimes Z)\ast M \ar[r] & X\ast(Z\ast M) \ .
}}
\ee

Following \cite{eno05}, we denote the category $\Fun_\CC(\CM, \CM)$ of $\CC$-module functors from $\CM$ to 
$\CM$ by $\CC_\CM^\ast$. It is a monoidal category with tensor product giving by the composition of functors. We denote the monoidal category with the same underlining category and but with the opposite tensor product by $\CC_\CM^\vee$.

\begin{prop} \label{prop:aind-factor}
The $\alpha$-induction $\alpha: \CZ(\CC)\to\CC_\CM^\ast$ factors through $\CZ(\CC_\CM^\ast)$, i.e.
the following diagram of 
	monoidal
functors commutes 
$$\xymatrix{ & \CZ(\CC_\CM^\ast)\ar[dr]^{forget} \\ \CZ(\CC)\ar[rr]^\alpha \ar[ur]^{\tilde{\alpha}} && \CC_\CM^\ast}$$
\end{prop}
\pf
Define a half braiding $\delta_{Z,F}:\alpha(Z)\circ F\to F\circ\alpha(Z)$, for $F\in\CC_\CM^\ast$ by 
\begin{equation}\label{hb}
(\delta_{Z,F})_M = \Big(
\xymatrix{
\alpha(Z)\circ F(M) \ar@{=}[r] & Z*F(M) \ar[r]^{\hspace{0.3cm}(F_{Z,M}^{(2)})^{-1}} & F(Z*M) \ar@{=}[r] & F\circ \alpha(Z)(M)
}\Big)
\end{equation}
It is clear that $(\delta_{Z,F})_M$ is natural in $M$.
That $\delta_{Z,F}$ is even a natural transformation of $\CC$-module functors amounts to commutativity of the following diagram:
$$
\xymatrix{
\alpha(Z)(F(X*M)) \ar[dddddd] \ar[rrr]^{(\delta_{Z,F})_{X*M}} \ar@{=}[dr] &&& F(\alpha(Z)(X*M)) \ar[dddddd] \ar@{=}[dl] \\
& Z*F(X*M) \ar[d]_{1*F_{X,M}^{(2)}} & F(Z*(X*M)) \ar[l]^{F_{Z,X*M}^{(2)}} \ar[d]^{F(a_{Z,X,M})} \\
& Z*(X*F(M)) \ar[d]_{a_{Z,X,F(M)}} & F((Z\otimes X)*M) \ar[d]^{F(z_M*1)} \ar[ld]_{F_{ZX,M}^{(2)}}\\
& (Z\otimes X)*F(M) \ar[d]_{z_X*1} & F((X\otimes Z)*M) 
   \ar[d]^{F(a_{X,Z,M}^{-1})} 
\ar[ld]^{F_{XZ,M}^{(2)}}\\
& (X\otimes Z)*F(M) 
  \ar[d]_{a_{X,Z,M}^{-1}} 
& F(X*(Z*M)) \ar[d]_{F_{X,Z*M}^{(2)}}\\
& X*(Z*F(M)) & X*F(Z*M) \ar[l]_{1*F_{Z,M}^{(2)}}\\
X*\alpha(Z)(F(M)) \ar[rrr]_{X*(\delta_{Z,F})_M} \ar@{=}[ur] &&& X*F(\alpha(Z)(M)) \ar@{=}[ul]
}$$
Here the horizontal arrows of the outer square are half braiding isomorphisms for $\alpha(Z)$, while the vertical arrows are the $\CC$-structures of the compositions $\alpha(Z)\circ F$ and $F\circ\alpha(Z)$. 

The hexagon axiom for the half braiding follows from the definition of the composition of $\CC$-module functors:
$$\xymatrix@C=-20pt{
(\alpha(Z)\circ F\circ G)(M) \ar@{=}[dr] 
\ar@/_20pt/[rrddd]_{(\delta_{Z,F})_{G(M)}} \ar[rrrr]^{(\delta_{Z,F\circ G})_M} 
&&&& (F\circ G\circ\alpha(Z))(M) \ar@{=}[dl] \\
 & Z*F(G(M))\ar[dr]_{F_{Z,G(M)}} \ar[rr]^{(F\circ G)_{Z,M}} && F(G(Z*M)) \\
 && F(Z*G(M)) \ar[ur]_{F(G_{Z,M})} \ar@{=}[d]\\
 && (F\circ\alpha(Z)\circ G)(M) \ar@/_20pt/[rruuu]_{F((\delta_{Z,G})_M)}
}$$
The monoidal structure for the functor $\CZ(\CC)\to \CZ(\CC_\CM^\ast)$ is exhibited by the associativity of the action:
$$\xymatrix{
\alpha(Z\otimes U)(M) \ar@{=}[r] & (Z\otimes U)*M \ar[r]^{a^{-1}_{Z,U,M}} & Z*(U*M) \ar@{=}[r] & (\alpha(Z)\circ\alpha(U))(M).
}$$
The fact that this is an isomorphisms of objects in $\CZ(\CC_\CM^\ast)$ 
$$(\alpha(Z\otimes U),\delta_{ZU,-})\to (\alpha(Z)\circ\alpha(U),\delta_{Z,-}|\delta_{U,-}) = (\alpha(Z),\delta_{Z,-})\otimes(\alpha(U),\delta_{U,-})$$
amounts to commutativity of the outer square of the following diagram:
$$\xymatrix@C=-15pt{(\alpha(Z)\alpha(U)F)(M) \ar[rrrr]^{(\delta_{Z,-}|\delta_{U,-})_M} \ar@{=}[ddr] \ar[drr]_{1*\delta_{U,F}} \ar[ddddd] &&&& (F\alpha(Z)\alpha(U))(M) \ar@{=}[ddl] \ar[ddddd] \\ && (\alpha(Z)F\alpha(U))(M) \ar[rru]_{\delta_{Z,F}*1} \ar@{=}[dd] && \\& Z*(U*F(M)) \ar[dd]_{a_{Z,U,F(M)}} &  & F(Z*(U*M)) \ar[dd]^{F(a_{Z,U,M})} \ar[ld]^{F_{Z,U*M}^{(2)}}\\
&& Z*F(U*M) \ar[lu]^{1*F_{U,M}^{(2)}} \\
& (ZU)*F(M) \ar@{=}[dl] \ar[rr]_{(F_{ZU,M}^{(2)})^{-1}} && F((ZU)*M) \ar@{=}[dr]\\
(\alpha(ZU)F)(M) \ar[rrrr]_{(\delta_{ZU,F})_M} &&&& (F\alpha(ZU))(M)
}$$
which follows from the coherence properties of $F^{(2)}$. The morphism associated to units $\alpha(\one)(M) \xrightarrow{} \id_{\CM}(M)$ is given by $\one \ast M \xrightarrow{\simeq} M$. It is straightforward to check that all coherence conditions are satisfied. 
\epf

It is worthwhile to point out that the monoidal functor $\tilde{\alpha}$ does not respect the braiding. It maps the braiding of $\CZ(\CC)$ to the antibraiding of $\CZ(\CC_\CM^\ast)$. By slightly modifying the target category, we obtain a braided monoidal functor $\tilde{\alpha}: \CZ(\CC) \to \CZ(\CC_\CM^\vee)$ \cite{eo}. Let $\Fun_{\CC|\CC_\CM^\vee}(\CM, \CM)$ be the category of $\CC$-$\CC_\CM^\vee$-bimodule functors. 
It is clear that there are canonical functors \cite{eo}:
$$
\CZ(\CC)\xrightarrow{L_\CM} \Fun_{\CC|\CC_\CM^\vee}(\CM, \CM)  \xleftarrow{R_\CM} Z(\CC_\CM^\vee).
$$
Then it is an easy exercise to show that the following diagram: 
\be \label{diag:Z(M)}
\xymatrix{ 
& \Fun_{\CC|\CC_\CM^\vee}(\CM, \CM) & \\
\CZ(\CC) \ar[rr]^{\tilde{\alpha}} \ar[ur]^{L_\CM} & & Z(\CC_\CM^\vee) \ar[ul]_{R_\CM} \, .
}
\ee
is commutative. This commutative diagram and Proposition\,\ref{prop:aind-factor} will be used in the proof of Lemma\,\ref{thm:ENO}.

\begin{defn}[see {\cite[Sect.\,6]{da}}] \label{def:full-center}
Let $\CC$ be a monoidal category and let $\CM$ be a $\CC$-module. The {\em full center} $Z(\CM)$ of $\CM$ is the terminal object in the category of pairs $(Z,f)$ where $Z \in \CZ(\CC)$ and $f : \alpha(Z) \to \id_{\CM}$ is a $\CC$-module natural transformation.
Equivalently, $Z(\CM)$ is the terminal object in the comma-category $\alpha\,{\downarrow}\,id_\CM$.
\end{defn}

That $f$ is a $\CC$-module natural transformation means that it is a natural transformation of functors $\CM\to\CM$ such that $f_{U \ast M} = (\id_U \ast f_M) \circ (z_U \ast \id_M)$ for all $U \in \CC$ and $M \in \CM$, cf.\ Definition~\ref{def:C-mod-nat-xfer}. The full center of a module category may or may not exist. In fact, its existence is linked to that of an internal hom: the monoidal functor $\alpha : \CZ(\CC) \to \CC_\CM^\ast$ turns $\CC_\CM^\ast$ into a $\CZ(\CC)$-module, and for $\id_{\CM}$ the identity module functor on $\CM$ we have
\be \label{eq:Z(M)=[id,id]}
  Z(\CM) = [\id_{\CM},\id_{\CM}] \ .
\ee
Indeed, $[\id_{\CM},\id_{\CM}]$ is equally the terminal object in the comma-category $\alpha\,{\downarrow}\,\id_{\CM}$, cf.\ Def. \ref{def:action-internal-hom}. Accordingly, from \eqref{eq:def-int-hom-cl} we get a family of isomorphisms of hom spaces, natural in
$Z \in \CZ(\CC)$,
\be \label{eq:adjoint-of-alpha}
   \Hom_{\CZ(\CC)}(Z,Z(\CM)) \cong
\Hom_{\CC_\CM^\ast}(\alpha(Z),\id_\CM)
\ .
\ee

By Rem.\,\ref{rema:iso-int-hom}\,(i), $\CZ(\CM)$ is an algebra in $\CZ(\CC)$. In fact, Cor.\,\ref{cco} below will show that it is even a {\em commutative} algebra (see also \cite[Prop.\,6.1]{da}). The full center of an algebra $A \in \CC$ is defined in terms of the category $\CC_A$ of right $A$-modules as $Z(A) = Z(\CC_A)$. There is also a direct definition of $Z(A)$ in terms of the algebra $A$, see \cite[Sect.\,4\,\&\,Thm.\,6.2]{da}. In the case that $\CC$ is a modular tensor category, the present definition of $Z(A)$ coincides with the original one \cite[Def.\,4.9]{ffrs3}, cf.\ \cite[Sect.\,8]{da}.

\subsection{The centralizer of a module functor}\label{sec:cent-module-fun}

Let $\CM,\CN$ be left $\CC$-modules and $F: \CM\to\CN$ a $\CC$-module functor. The category $\Fun_\CC(\CM,\CN)$ of $\CC$-module functors from $\CM$ to $\CN$ is a left $\End_\CC(\CN)$-module and a right $\CC_\CM^\ast$-module. The $\alpha$-inductions $\alpha_\CN: \CZ(\CC)\to\End_\CC(\CN)$ and 
$\alpha_\CM: \CZ(\CC)\to\CC_\CM^\ast$ provide the category $\Fun_\CC(\CM,\CN)$ with a structure of a left and right $\CZ(\CC)$-module. The left and right actions are equivalent in the sense that for all $Z \in \CZ(\CC)$ and $F \in \Fun_\CC(\CM,\CN)$ we have an equivalence of $\CC$-module functors
\be
Z \ast F := \alpha_\CN(Z) \circ F ~\cong~   F \circ \alpha_\CM(Z) =: F \ast Z \ .
\ee
More precisely, we have the following

\begin{lemma} \label{lemma:2-actions}
Let $F : \CM \to \CN$ be a $\CC$-module functor. 
Then there is a natural isomorphism $\zeta(F)$ between the functors
\be \label{eq:act}
\xymatrix@=1em{ 
& \CZ(\CC) \ar[rd]_{\alpha_\CN}  \ar[ld]^{\alpha_\CM} & 
\\ 
\CC_\CM^\ast  \ar[rd]^{F\circ\da}
  & \overset{\zeta(F)}{\Longrightarrow} & \End_\CC(\CN) \ar[ld]_{\da\circ F} \ .
\\ 
  & \Fun_\CC(\CM,\CN) &}
\ee
\end{lemma}

\pf
To define $\zeta(F)$ we need to give, for each $Z \in \CZ(\CC)$, a morphism 
\be  \label{eq:def-zeta}
\zeta(F)_{Z}:F\circ\alpha_\CM(Z)\to\alpha_\CN(Z)\circ F
\ee 
of $\CC$-module functors, i.e.\ each $\zeta(F)_{Z}$ is a $\CC$-module natural transformation (cf.\ Definition~\ref{def:C-mod-nat-xfer}). We define its specialization 
$(\zeta(F)_{Z})_M$ on $M \in \CM$ to be the morphism
\be \label{eq:zeta-via-F2}
  \xymatrix{F(\alpha_\CM(Z)(M)) = F(Z\ast M) \ar[r]^{F_{Z,M}^{(2)}} & Z\ast F(M) = \alpha_\CN(Z)(F(M)) } \ ,
\ee
where $F_{Z,M}^{(2)}$ is the $\CC$-module structure of $F$. The fact that for each $Z \in \CZ$, $\zeta(F)_Z$ is a $\CC$-module natural transformation follows from the commutativity of the diagram
$$
{\small
\xygraph{ !{0;/r5.5pc/:;/u4.5pc/::}
*+{F(\alpha_\CM(Z)(X{\ast}M))}
(
:[r(5)] *+{\alpha_\CN(Z)(F(X{\ast}M))}="1" ^{(\zeta(F)_Z)_{X{\ast}M}}
:[d(2.5)] *+{\alpha_\CN(Z)(X{\ast}F(M))} ="2"^{\alpha_\CN(F^{(2)}_{X,M})}
:[d(1.5)] *+{X{\ast}\alpha_\CN(F(M))\ .}="3" ^{\alpha_\CN(Z)^{(2)}_{X,F(M)}}
:@{=}[u(.5)l] *+{X{\ast}(Z{\ast}F(M))}="4"
:[u(.5)l] *+{(XZ){\ast}F(M)} ="5"_(.4){a_{X,Z,F(M)}}
,
:@{=}[d(.5)r]*+{F(Z{\ast}(X{\ast}M))}
 (
 :[r(3)] *+{Z{\ast}F(X{\ast}M)} ^{F^{(2)}_{Z,X{\ast}M}}
  (
  :[d(1.5)] *+{Z{\ast}(X{\ast}F(M))} ^{1\,F^{(2)}_{X,M}}
   (
   :@{=}"2"
   ,
   :[u(.5)l] *+{(ZX){\ast}F(M)}="6" _(.45){a_{Z,X,F(M)}}
   :"5" ^{z_X 1}
   )
  ,
  :@{=}"1"
  )
 ,
 :[d(.5)r] *+{F((ZX){\ast}M)} _(.4){F(a_{Z,X,M})}
  (
  :"6" ^(.6){F^{(2)}_{ZX,M}}
  ,
  :[d(1.5)] *+{F((XZ){\ast}M)}="7" _{F(z_X 1)}
  :"5" ^(.6){F^{(2)}_{XZ,M}}
  )
 )
,
:[d(1.5)]*+{F(X{\ast}\alpha_\CM(Z)(M))} ^(.7){F(\alpha_\CM(Z)^{(2)}_{X,M})}
 (
 :@{=}[d(.5)r] *+{F(X{\ast}(Z{\ast}M))} 
  (
  :"7" ^(.6){F(a_{ZXM})}
  ,
  :[d(1.5)] *+{X{\ast}F(Z{\ast}M)}="8" _{F^{(2)}_{X,Z{\ast}M}}
  :"4" ^{1\,F^{(2)}_{Z,M}}
  )
 ,
 :[d(2.5)] *+{X{\ast}F(\alpha_\CM(Z)(M))} _{F^{(2)}_{X,\alpha(M)}}
  (
  :@{=}"8"
  ,
  :"3" _{1\,(\zeta(F)_Z)_M}
  )
 )
)
}}
$$
It remains to check that $\zeta(F)$ is a natural transformation, i.e.\ that for all $f : Y \to Z$ in $\CZ(\CC)$ the square
\be
\xymatrix{
F \circ \alpha_{\CM}(Y) \ar[rr]^{F \circ \alpha_{\CM}(f)} \ar[d]^{\zeta(F)_Y}
  && F \circ \alpha_{\CM}(Z) \ar[d]^{\zeta(F)_Z} 
  \\
\alpha_{\CN}(Y) \circ F \ar[rr]^{\alpha_{\CN}(f) \circ F} 
  && \alpha_{\CM}(Z) \circ F 
}
\ee
of $\CC$-natural transformations commutes. Evaluating on $M \in \CM$, this boils down to the identity 
$$
\Big(F(Y\ast M) \xrightarrow{F_{Y,M}^{(2)}} Y \ast F(M) \xrightarrow{f\,1} Z \ast F(M)\Big)
~=~
\Big(F(Y\ast M) \xrightarrow{F(f\,1)} F(Z \ast M) \xrightarrow{F_{Y,M}^{(2)}} Z \ast F(M) \Big) \ ,
$$
which holds because $F$ is a $\CC$-module functor.
\epf

By the above lemma we loose nothing if we restrict ourselves to the left (say) action of $\CZ(\CC)$ on $\Fun_\CC(\CM,\CN)$, so this will be the $\CZ(\CC)$-module structure on $\Fun_\CC(\CM,\CN)$ we use below.

\begin{defn}  
Let $F: \CM \rightarrow \CN$ be a $\CC$-module functor.
The {\em full centralizer} $Z(F)$ of the functor $F$ is defined to be the 
internal hom $[F,F]$ (valued in $\CZ(\CC)$) 
with respect to the $\CZ(\CC)$-action on $\Fun_\CC(\CM,\CN)$.
\end{defn}

As the full centralizer is defined in terms of an internal hom, it may or may not exist depending on the choice of $F$ (and $\CC$, $\CM$, $\CN$). By Rem.\,\ref{rema:iso-int-hom}\,(i), if it exists, $Z(F)$ is a (not necessarily commutative) algebra in $\CZ(\CC)$. The reason for the name `full centralizer' is the relation to the centralizer $Z(f)$ discussed in Section \ref{sec:intro-1} -- we will return to this point with more details in Section \ref{sec:lax-functor} below.

Fix a $\CC$-module functor $F : \CM \rightarrow \CN$, and fix a third $\CC$-module $\CP$. The left and right composition with $F$ defines functors
\be
  F \circ - : \Fun_\CC(\CP,\CM) \to \Fun_\CC(\CP,\CN)
  \quad , \quad
  - \circ F : \Fun_\CC(\CN,\CP) \to \Fun_\CC(\CM,\CP) \ .
\ee
The following lemma shows that these are in fact $\CZ(\CC)$-module functors.

\begin{lemma}
$F \circ -$ and $- \circ F$ with $\CZ(\CC)$-module structure
\be\label{eq:left-right-comp-C-mod-structure}
\begin{array}{rcrcrcl}
  (F \circ -)^{(2)}_{X,H} &\hspace{-.5em}=\hspace{-.5em}& \zeta(F)_X \circ H &\hspace{-.3em}:\hspace{-.3em}& F \circ \alpha_\CM(X) \circ H &\hspace{-.3em}\longrightarrow\hspace{-.3em}& \alpha_\CN(X) \circ F \circ H \ ,\\[.8em]
  (- \circ F)^{(2)}_{X,H} &\hspace{-.5em}=\hspace{-.5em}& \id_{\alpha_\CN(X) H F}  &\hspace{-.3em}:\hspace{-.3em}& \alpha_\CN(X) \circ H \circ F &\hspace{-.3em}\longrightarrow\hspace{-.3em}& \alpha_\CN(X) \circ H \circ F
\end{array}
\ee
are $\CZ(\CC)$-module functors.
\end{lemma}

\pf
We have to check compatibility with associator and unit isomorphisms. For $(- \circ F)^{(2)}$ this is trivial. For $(F \circ -)^{(2)}$, the associator condition reads
\be
\xymatrix@C=.5em{
& F \circ \alpha_\CM(X \otimes Y)  \circ  H \ar[dr]^{~~\zeta(F)_{XY} \circ H}
\\
F  \circ  \alpha_\CM(X) \circ \alpha_\CM(Y)  \circ  H \ar[ur]^{F \circ a^\CM_{X,Y,H-}~~} \ar[d]^{\zeta(F)_X \circ (\alpha_\CM(Y)H)}
 && \alpha_\CN(X \otimes Y) \circ  F \circ H 
 \\
  \alpha_\CN(X) \circ F \circ \alpha_\CM(Y) \circ H \ar[rr]^{\alpha_\CN(X) \circ \zeta(F)_Y \circ H} 
 && \alpha_\CN(X)  \circ \alpha_\CN(Y)  \circ F \circ H \ar[u]^{a^\CN_{X,Y,FH-}}
}
\ee
Applying this to some $P \in \CP$ and substituting the definition of $\zeta$ in \eqref{eq:zeta-via-F2}, this becomes precisely the associator condition for $F^{(2)}$. The unit compatibility follows along the same lines.
\epf

Let $G,G': \CN\to\mathcal{P}$, $H,H':\CL\to\CM$ and $F : \CM\to\CN$ be $\CC$-module functors. 
By the above lemma, $(-) \circ F$ and $F \circ (-)$ are $\CZ(\CC)$-module functors
so that the construction in \eqref{eq:def-F-MN} provides us with morphisms
\be\label{eq:C-mod-F-gives-ZC-morph}
  [-\circ F]_{[G,G']} : [G,G']\to [GF,G'F]
  \quad \text{and} \quad
  [F \circ -]_{[H,H']} : [H, H']\to [FH,FH'] 
\ee
in $\CZ(\CC)$, provided that the internal homs appearing as source and target exist.
These morphisms satisfy a `commutativity' relation stated in the following proposition, which will be instrumental in the further discussion. The commutative diagram below can be visualized as a sewing constraint when interpreted in the CFT setting briefly discussed in Section~\ref{sec:intro-3}, see Figure~\ref{fig:CFT-horizontal-defect-OPE}.

\begin{figure}[bt]
\centerline{\begin{tabular}{@{}c@{\quad}c@{\quad}c@{\quad}c@{\quad}c}
\raisebox{-50pt}{
  \begin{picture}(105,120)
   \put(0,8){\scalebox{0.6}{\includegraphics{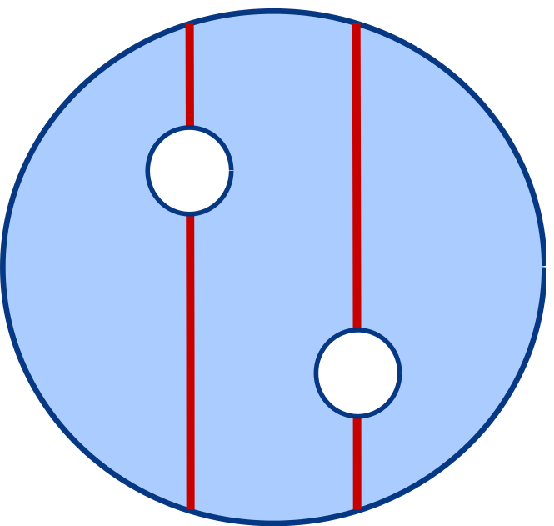}}}
   \put(0,8){
     \setlength{\unitlength}{.75pt}\put(-18,-19){
     \put(50, 32)      {\scriptsize $ G $}
     \put(47, 120)       {\scriptsize $ G' $}
     \put(50, 75)       {\scriptsize $ G $} 
     \put(88, 120)       {\scriptsize $ F' $}
     \put( 88, 75)       {\scriptsize $ F' $}
     \put( 90, 30)       {\scriptsize $ F $}
     \put(37, 55)        {\scriptsize $ \CL $}
     \put(71, 55)       {\scriptsize $ \CM $}
     \put(118, 55)    {\scriptsize $ \CN $} 
     }\setlength{\unitlength}{1pt}}
  \end{picture}}
& $\rightsquigarrow$ &
\raisebox{-50pt}{
  \begin{picture}(105,120)
   \put(0,8){\scalebox{0.6}{\includegraphics{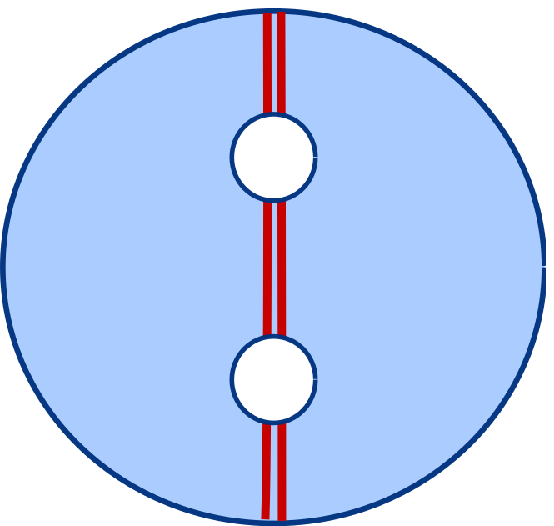}}}
   \put(0,8){
     \setlength{\unitlength}{.75pt}\put(-18,-19){
     \put(53, 120)       {\scriptsize $ G'F' $}
     \put(58, 75)       {\scriptsize $ GF' $} 
     \put(60, 32)      {\scriptsize $ GF $}
      \put(37, 55)        {\scriptsize $ \CL $}
     \put(118, 55)    {\scriptsize $ \CN $} 
     }\setlength{\unitlength}{1pt}}
  \end{picture}}
& $\rightsquigarrow$ &
\raisebox{-50pt}{
  \begin{picture}(105,120)
   \put(0,8){\scalebox{0.6}{\includegraphics{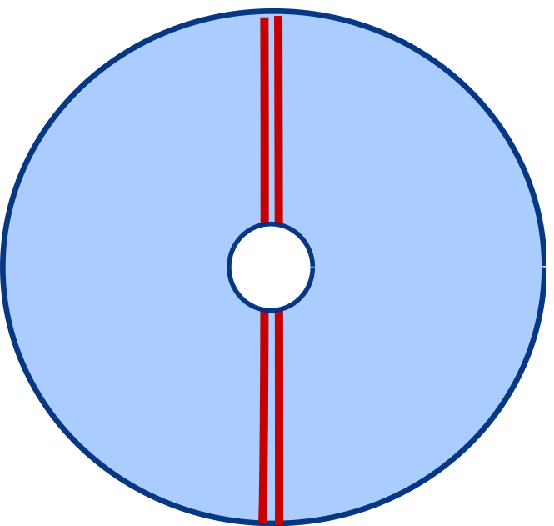}}}
   \put(0,8){
     \setlength{\unitlength}{.75pt}\put(-18,-19){
     \put(53, 115)       {\scriptsize $ G'F' $}
     \put(60, 37)      {\scriptsize $ GF $}
      \put(37, 55)        {\scriptsize $ \CL $}
     \put(118, 55)    {\scriptsize $ \CN $} 
     }\setlength{\unitlength}{1pt}}
  \end{picture}}
\\[5pt]
\raisebox{-50pt}{
  \begin{picture}(105,120)
   \put(0,8){\scalebox{0.6}{\includegraphics{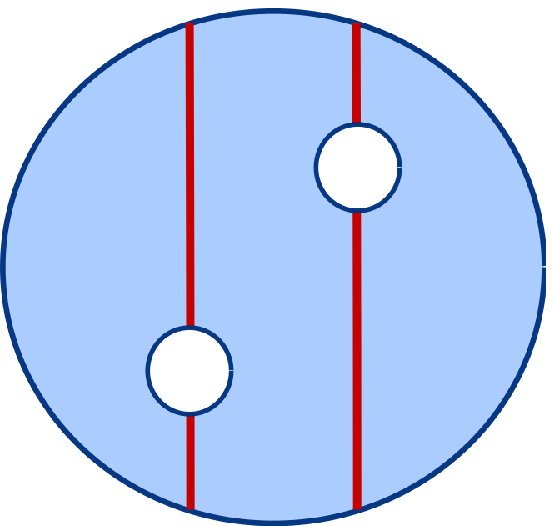}}}
   \put(0,8){
     \setlength{\unitlength}{.75pt}\put(-18,-19){
     \put(50, 32)      {\scriptsize $ G $}
     \put(47, 120)       {\scriptsize $ G' $}
     \put(47, 75)       {\scriptsize $ G' $} 
     \put(88, 120)       {\scriptsize $ F' $}
     \put( 90, 75)       {\scriptsize $ F $}
     \put( 90, 30)       {\scriptsize $ F $}
      \put(37, 55)        {\scriptsize $ \CL $}
     \put(78, 55)       {\scriptsize $ \CM $}
     \put(118, 55)    {\scriptsize $ \CN $} 
     }\setlength{\unitlength}{1pt}}
  \end{picture}}
& $\rightsquigarrow$ &
\raisebox{-50pt}{
  \begin{picture}(105,120)
   \put(0,8){\scalebox{0.6}{\includegraphics{comm-def-fields-a2.eps}}}
   \put(0,8){
     \setlength{\unitlength}{.75pt}\put(-18,-19){
     \put(53, 120)       {\scriptsize $ G'F' $}
     \put(57, 75)       {\scriptsize $ G'F $} 
     \put(60, 32)      {\scriptsize $ GF $}
      \put(37, 55)        {\scriptsize $ \CL $}
     \put(118, 55)    {\scriptsize $ \CN $} 
     }\setlength{\unitlength}{1pt}}
  \end{picture}}
& $\rightsquigarrow$ &
\raisebox{-50pt}{
  \begin{picture}(105,120)
   \put(0,8){\scalebox{0.6}{\includegraphics{comm-def-fields-a3.eps}}}
   \put(0,8){
     \setlength{\unitlength}{.75pt}\put(-18,-19){
     \put(53, 115)       {\scriptsize $ G'F' $}
     \put(60, 37)      {\scriptsize $ GF $}
      \put(37, 55)        {\scriptsize $ \CL $}
     \put(118, 55)    {\scriptsize $ \CN $} 
     }\setlength{\unitlength}{1pt}}
  \end{picture}}
\end{tabular}}
\caption{The CFT sewing constraint corresponding to \eqref{diag:comm} in the conventions used in Section~\ref{sec:intro-3}: We consider a CFT with three possible colors for the world sheet, given by the three $\CC$-modules $\CL$, $\CM$, $\CN$, and four different defect conditions labelled by the functors $F,F',G,G'$. The internal homs $[F,F']$, etc., give the space of defect (changing) fields. The composition of functors amounts to the `fusion' of topological defect lines. The sewing shown in figure a) corresponds to the upper path in \eqref{diag:comm}, and the sewing in figure b) to the lower path. In words it says that $[F,F']$ and $[G,G']$ can both be interpreted as subspaces of the space of defect changing fields from the fused defect $GF$ to $G'F'$, and that fields in these subspaces mutually commute.}
\label{fig:CFT-horizontal-defect-OPE}
\end{figure}

\begin{prop}  \label{prop:comm} 
Let $\CC$ be a monoidal category, let $\CL$, $\CM$, $\CN$ be $\CC$-modules, and let $F, F': \CL \to \CM$ and $G, G': \CM \to \CN$ be $\CC$-module functors such that all the internal homs in \eqref{diag:comm} exist. Then the diagram
\be \label{diag:comm}
\raisebox{2.8em}{\xymatrix@R=1em{
[G,G'] \otimes [F, F']  \ar[dd]_{c_{[G,G'],[F,F']}} \ar[r]^(0.4){a} & [GF', G'F'] \otimes [GF, GF'] \ar[rd]^(.6){\comp_{GF'}} & \\ 
& & [GF, G'F']  \\
[F, F'] \otimes [G, G']  \ar[r]^(0.4){b} & [G'F, G'F'] \otimes [GF, G'F] \ar[ur]_{\hspace{1cm}\comp_{G'F}} &
}}
\ee
commutes. Here $a = (- \circ F')_{[G,G']} \otimes (G \circ -)_{[F,F']}$ and $b = (G' \circ -)_{[F,F']} \otimes (- \circ F)_{[G,G']}$.
\end{prop}

\pf
By the universal property of internal homs, it is enough to verify the identity of the upper and lower path in the diagram after applying $(-) \ast GF$ to both sides and composing with $\ev_{GF}$. The resulting morphism from the upper path can be rewritten as shown in the following commutative diagram (we abbreviate $L = G \circ -$ and $R = - \circ F'$)
$$
\small{
\xymatrix{
[G,G'] \ast ([F,F'] \ast L(F)) \ar[r]^{[R]11} \ar[d]^{1 (L^{(2)}_{[F,F'],F})^{-1}}
 & [GF',G'F']\ast ([F,F'] \ast L(F)) \ar[d]^{1 (L^{(2)}_{[F,F'],F})^{-1}} \ar[dr]^{1 \, [L]1}
\\
[G,G'] \ast L([F,F'] \ast F)  \ar[d]^{1\,L(\ev_F)}
 & [GF',G'F'] \ast L([F,F'] \ast F)
   \ar[d]^{1 L(\ev_{F})}   
 & [GF',G'F']\ast ([GF,GF'] \ast GF)\ar[d]^{\comp_{GF'}} \ar[dl]^{1 \, \ev_{GF} }
\\
[G,G'] \ast GF' \ar[d]^{(R^{(2)}_{[G,G'],G})^{-1}} \ar[r]^{[R]1}
 & [GF',G'F'] \ast GF' \ar[d]^{\ev_{GF'}}
 & [GF,G'F'] \ast GF \ar[dl]^{\ev_{GF}}
\\
R([G,G'] \ast G) \ar[r]^{R(\ev_G)}
& G'F'
}}
$$
The diagram commutes by functoriality of $\ast$ in both arguments, and by the definitions \eqref{diag:def-compos}, \eqref{eq:def-F-MN} of $\comp$ and $[R]$, $[L]$. Inserting the definitions \eqref{eq:left-right-comp-C-mod-structure} of $L^{(2)}$ and $R^{(2)}$, we find that the lower path in the above diagram is equal to 
\be\label{eq:FG-comm-aux1}\begin{array}{l}
[G,G'] \ast ([F,F'] \ast (G F))
\xrightarrow{1\,\zeta(G)^{-1}_{[F,F']} \circ F}
[G,G'] \ast ( G \circ ([F,F'] \ast F)) \\[.5em]
\qquad 
\xrightarrow{1\,G(\ev_F)}
[G,G'] \ast ( G F')
\xrightarrow{ \ev_G \circ F'}
G' \circ F'
\end{array}\ee
An analogous calculation for the lower path in \eqref{diag:comm} yields the morphism (omitting associators)
\be\label{eq:FG-comm-aux2}\begin{array}{l}
[G,G'] \ast ([F,F'] \ast (G F))
\xrightarrow{c_{[G,G'] , [F,F']} 1}
[F,F'] \ast ([G,G'] \ast (G F)) \\[.5em]
\qquad \xrightarrow{1(\ev_G \circ F)}
[F,F'] \ast (G'F)
\xrightarrow{\zeta(G')^{-1}_{[F,F']} \circ F}
G'([F,F'] \ast F)
\xrightarrow{ G' \circ \ev_F}
G' \circ F'
\end{array}\ee
To establish commutativity of \eqref{diag:comm}, it hence remains to verify that \eqref{eq:FG-comm-aux1} and \eqref{eq:FG-comm-aux2} are equal, i.e.\ that the following diagram commutes,
\be  \label{diag:action-comm}
\xymatrix{
([G, G'] \otimes [F, F']) \ast (GF) \ar[d]_{c_{[G,G'],[F,F']}1} \ar[r]^\cong & [G,G']\ast [F,F'] \ast (GF) 
   \ar[d]^{1 \zeta(G)_{[F,F']}^{-1}} & \\
([F,F'] \otimes [G, G']) \ast (GF) \ar[d]_\cong & [G, G'] \ast (G([F, F']\ast F)) \ar[r]^{\hspace{0.5cm}1(1\ev_F)} \ar[d]^\cong &  [G,G']\ast GF' \ar[d]^\cong  \\
[F,F'] \ast (([G,G']\ast G) \circ F) \ar[r]^{x}  \ar[d]_{1\ev_G1}   & 
([G,G']\ast G)\circ ([F,F']\ast F) \ar[r]^{\hspace{0.5cm}1\ev_F} \ar[d]^{\ev_G1} & ([G,G']\ast G)\circ F' \ar[d]^{\ev_G1} \\
[F, F']\ast G'F \ar[r]^{\zeta(G')_{[F,F']}^{-1}} & G'\circ ([F,F']\ast F) \ar[r]^{\hspace{0.5cm}1\ev_F} &  G'\circ F' 
}
\ee
where $x=\zeta([G,G']\ast G)_{[F,F']}^{-1}$ 
and the map $\zeta([G,G']\ast G)_{[F,F']}$ is defined in (\ref{eq:zeta-via-F2}).
By the definition \eqref{diag:alpha-ind} of the $\CC$-module structure on $\alpha(Z)$, the natural transformation $\zeta(\alpha([G,G']))_{[F,F']}$ is given by the braiding $c_{[G,G'], [F,F']}$. Therefore upper-left subdiagram commutes. Substituting
$\zeta([G,G']\ast G)_{[F,F']} = \big([G,G']\ast G\big)^{(2)}_{[F,F'],-}$ and $\zeta(G')_{[F,F']} = G^{(2)}_{[F,F'],-}$, we see that
the commutativity of the lower-left subdiagram is due to the fact that $[G, G']\ast G \xrightarrow{\ev_G} G'$ is a $\CC$-module natural transformation (see condition \eqref{eq:c-mod-nat-xfer-condition}). Altogether, it follows that \eqref{diag:action-comm} commutes.
\epf

Recall from \eqref{eq:Z(M)=[id,id]} that the full center $Z(\CM)$ of a $\CC$-module $\CM$ is the internal hom $[\id_\CM,\id_\CM]$ in $\CZ(\CC)$. The following two special cases of Proposition\,\ref{prop:comm} will be useful.

\begin{cor}\label{cor:Z-[FF]-[FG]}
For two given $\CC$-module functors $F, G: \CM \rightarrow \CN$, the diagrams
\be  \label{diag:ZN-Cl-FG}
\raisebox{2.8em}{\xymatrix@R=1em{
Z(\CN) \otimes [F, G] \ar[dd]_{c_{Z(\CN), [F,G]}}  \ar[rr]^{[-\circ G]_{[\id,\id]}1} && Z(G) \otimes [F, G] \ar[rd]^{\comp_G}  &  \\
&&& [F, G]  \\
[F, G] \otimes Z(\CN) \ar[rr]^{1\,[-\circ F]_{[\id,\id]}} && [F, G] \otimes Z(F) \ar[ru]_{\comp_F} &  
}}
\ee
and
\be  \label{diag:ZM-Cr-FG}
\raisebox{2.8em}{\xymatrix@R=1em{
[F, G] \otimes Z(\CM)  \ar[dd]_{c_{[F,G],Z(\CM)}}  \ar[rr]^{1\,[F\circ -]_{[\id,\id]}} && [F, G] \otimes Z(F)  \ar[rd]^{\comp_F}  &  \\
&&& [F, G] \\
Z(\CM) \otimes [F, G]  \ar[rr]^{[G\circ -]_{[\id,\id]}1} && Z(G) \otimes [F, G] \ar[ru]_{\comp_G} &
}}
\ee
are commutative, provided all internal homs exist.
\end{cor}

Let $F: \CM \to \CN$ be a $\CC$-module functor. Recall that $Z(\CM)$, $Z(\CN)$ and $Z(F)$ are all algebras. 
By Lemma~\ref{lemma:F-C-functor}, the morphisms in the cospan
\be \label{eq:ZM-ZF-ZN-cospan}
\raisebox{2.8em}{\xymatrix@R=2em{
& Z(F) \\ Z(\CM)\ar[ru]^{\hspace{-.2cm}[F\circ -]_{[\id,\id]}} && Z(\CN) \ar[lu]_{\,\,[-\circ F]_{[\id,\id]}}
}}
\ee
are algebra homomorphisms in $\CZ(\CC)$. The images of these two morphisms are central in $Z(F)$ in a sense which we will now describe. 

The {\em left center} $C_l(B)$ (resp.\ {\em right center} $C_r(B)$) of an algebra $B$ in braided monoidal category $\CZ$ is the terminal object in the category of morphisms $f:Y\to B$ such that the diagram 
\begin{equation}\label{eq:lcp}
\raisebox{2.5em}{
\xymatrix@R=1em{ Y\otimes B \ar[r]^{f\,1} \ar[dd]_{c_{Y,B}} & B\otimes B \ar[rd]^\mu\\ & & B\\ B\otimes Y \ar[r]_{1\,f} & B\otimes B \ar[ur]_\mu }
}
\quad \Bigg(\text{resp.} ~~
\raisebox{2.5em}{
\xymatrix@R=1em{ B\otimes Y \ar[r]^{1\,f} \ar[dd]_{c_{B,Y}} & B\otimes B \ar[rd]^\mu\\ & & B\\ Y\otimes B \ar[r]_{f\,1} & B\otimes B \ar[ur]_\mu }
}
\Bigg)
\end{equation}
commutes,  where $\mu$ is the multiplication on $B$. The left and right centers -- if they exist -- have unique algebra structures in $\CZ$ such that the morphisms $C_l(B)\to B$ and $C_r(B)\to B$ are algebra homomorphisms in $\CZ$ \cite[Prop.\,5.1]{da}; these algebra structures on $C_l(B)$ and $C_r(B)$ are commutative.

\begin{cor}\label{cco} 
Suppose that the internal homs $Z(\CM)$, $Z(\CN)$ and $Z(F)$ exist, and that the left and right center of $Z(F)$ exist. Then the algebra homomorphism $Z(\CM)\to Z(F)$ (resp.\ $Z(\CN)\to Z(F)$) factors through the right center (resp.\ left center), i.e.\ the following diagram of algebra homomorphisms commutes,
$$\xymatrix{ C_r(Z(F)) \ar[r] & Z(F) & C_l(Z(F)) \ar[l] \\ Z(\CM)\ar[ru] \ar[u] && Z(\CN)\ . \ar[lu] \ar[u] }$$
In particular, for $F = \id_\CM$ we obtain that $Z(\CM)$ is commutative.
\end{cor}

\pf
By Cor.\,\ref{cor:Z-[FF]-[FG]}, the right (resp.\ left) morphism in the cospan \eqref{eq:ZM-ZF-ZN-cospan} satisfies the defining property \eqref{eq:lcp} in the category of arrows characterizing the left and right center. They therefore have a unique morphism into the terminal object, viz.\ the left/right center.
\epf

\section{Bicategories of commutative algebras} \label{sec:two-cat}

In this section we define two bicategories built from cospans of algebras which will serve as the target for the full center functor to be defined in Section~\ref{sec:center-func} below.  

The first bicategory is $\CAlg(\CZ)$, for $\CZ$ a suitable braided monoidal category. In $\CAlg(\CZ)$, objects are commutative algebras in $\CZ$, 1-morphisms are cospans and 2-morphisms are homomorphisms between cospans. This category will be the target for the full center if we take the source to be $\alg(\CC)$, the category of algebras and algebra homomorhisms in a monoidal category $\CC$ (Definition~\ref{def:1-cat-alg} below), and $\CZ = \CZ(\CC)$ is the monoidal center.

The second bicategory is $\CALGu(\CZ)$, where objects and 1-morphisms are as for $\CAlg(\CZ)$, but now 2-morphisms are defined in terms of certain isomorphism classes of cospans of bimodules. Conjecturally, by leaving bimodule maps as 3-cells, one obtains a tricategory $\CALG(\CZ)$. The bicategory $\CALGu(\CZ)$ will serve as target for the full center if we take the source to be an appropriate subbicategory of $\CC$-modules.

\medskip

\subsection{Coequalizers}

To set the stage, we start with some technical preliminaries. Throughout Section 4, we assume that $\CZ$ is a braided monoidal category which has certain coequalizers, and that these coequalizers are compatible with the tensor product. More precisely, we make

\begin{ass}\label{ass:coeq}
The category $\CZ$ is a braided monoidal category such that
\\[.5em]
(i) given two morphisms $L,R : U \to V$ in $\CZ$ such that there is a common right inverse $\tau : V \to U$, i.e.\ $L \circ \tau = \id_V = R \circ \tau$,  the following coequalizer $C$ exists in $\CZ$,
$$
\xymatrix{ U \ar@<+.7ex>[r]^L \ar@<-.7ex>[r]_R & V \ar[r]^\rho & C} \ .
$$
(ii) the tensor product preserves the coequalizers of (i) in the sense that for all $X \in \CZ$ also
$$
\xymatrix{ U\otimes X \ar@<+.7ex>[r]^{L1} \ar@<-.7ex>[r]_{R1} & V\otimes X \ar[r]^{\rho1} & C\otimes X}
\quad \text{and} \quad
\xymatrix{ X\otimes U \ar@<+.7ex>[r]^{1L} \ar@<-.7ex>[r]_{1R} & X\otimes V \ar[r]^{1\rho} & X\otimes C}
$$
are coequalizers in $\CZ$.
\end{ass}

For example, the assumption holds if the braided monoidal category $\CZ$ is abelian with right exact tensor product; in particular it holds for $\CZ=\Vect_k$. We have the following useful lemma, whose proof we include for the convenience of readers.

\begin{lemma} \label{lem:coeq-iterate}
Let $\xymatrix{ U \ar@<+.7ex>[r]^L \ar@<-.7ex>[r]_R & V \ar[r]^\rho & C}$ and $\xymatrix{ U' \ar@<+.7ex>[r]^L \ar@<-.7ex>[r]_R & V' \ar[r]^\rho & C'}$ be two coequalizers satisfying the conditions of Assumption~\ref{ass:coeq}\,(i). Then also 
$$
  \xymatrix{ U \otimes U' \ar@<+.7ex>[r]^{\hspace{0cm}LL'} \ar@<-.7ex>[r]_{\hspace{0cm}RR'} & V \otimes V'  \ar[r]^{\rho\rho'} & C \otimes C'}
$$ 
is a coequalizer satisfying the conditions of Assumption~\ref{ass:coeq}\,(i).
\end{lemma}

\begin{proof}
Clearly, $L \otimes L'$ and $R \otimes R'$ have the common right inverse $\tau \otimes \tau'$, so that we only have to verify the universal property of the coequalizer. The construction is summarized in the diagram
\be
  \xymatrix{ 
  U \otimes U' \ar@<+.7ex>[r]^{\hspace{0cm}LL'} \ar@<-.7ex>[r]_{\hspace{0cm}RR'} 
  & V \otimes V'  \ar[r]^{\rho\,1}  \ar[d]^x
  & C \otimes V'  \ar[r]^{1\,\rho'} \ar@{.>}[dl]_{\exists! \phi}
  & C \otimes C'  \ar@{.>}[dll]^{\exists! \psi}
  \\
  & X
  }
\ee
which we now describe step by step. Suppose that $x \circ (LL') = x \circ (RR')$. Composing from the right with $1\,\tau'$ we see that $x \circ (L\,1) = x \circ (R\,1)$. By Assumption~\ref{ass:coeq}\,(ii), $\rho\,1$ is a coequalizer, and its universal property gives the existence of a unique $\phi$. But, using that also $x \circ (1\,L')  = x \circ (1\,R')$,
\be
  \phi \circ (1\,L') \circ (\rho\,1) 
  =
  \phi \circ (\rho\,1) \circ (1\,L')  
  =
  x \circ (1\,L')  
  =
  x \circ (1\,R')  
  =
  \phi \circ (1\,R') \circ (\rho\,1) \ .
\ee
The universal property of $\rho\,1$ implies $\phi \circ (1\,L') = \phi \circ (1\,R')$. Once more by Assumption~\ref{ass:coeq}\,(ii), we have the coequalizer
\be
\xymatrix{ C\otimes U' \ar@<+.7ex>[r]^{1L'} \ar@<-.7ex>[r]_{1R'} & C\otimes V' \ar[r]^{1\rho'} & C\otimes C'} \ ,
\ee
and its universal property guarantees the existence of a unique $\psi$ such that $\phi = \psi \circ (1\,\rho')$. Altogether we have found a unique $\psi$ such that $x = \psi \circ (\rho\,\rho')$.
\end{proof}

\subsection{Cospans between commutative algebras}

In this subsection we define cospans of algebras and their composition. 

\begin{defn} \label{def:cospan}
A {\em cospan between commutative algebras} in $\CZ$, or {\em cospan} for short, is a triple of algebras $A,B,T$, where $A$ and $B$ are commutative, together with two algebra homomorphisms
$$
\xymatrix@R=1em{ & T \\ A\ar[ru]^a && B \ar[lu]_b}
$$
such that the diagrams
\be\label{eq:cospan-central}
\raisebox{2.5em}{
\xymatrix@R=1em{ T\otimes A \ar[r]^{1\,a} \ar[dd]_{c_{T,A}} & T\otimes T \ar[rd]^\mu\\ & & T\\ A\otimes T \ar[r]_{a\,1} & T\otimes T \ar[ur]_\mu }
}
\quad \text{and} \quad
\raisebox{2.5em}{
\xymatrix@R=1em{ B\otimes T \ar[r]^{b\,1} \ar[dd]_{c_{B,T}} & T\otimes T \ar[rd]^\mu\\ & & B\\ T\otimes B \ar[r]_{1\,b} & T\otimes T \ar[ur]_\mu }
}
\ee
commute.
\end{defn}

\begin{figure}[bt]
\begin{center}
\raisebox{8em}{a)}
\scalebox{0.6}{\includegraphics{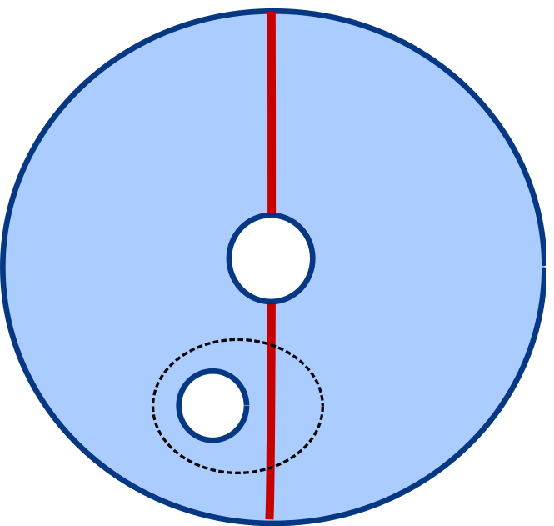}}
\hspace{3em}
\raisebox{8em}{b)}
\scalebox{0.6}{\includegraphics{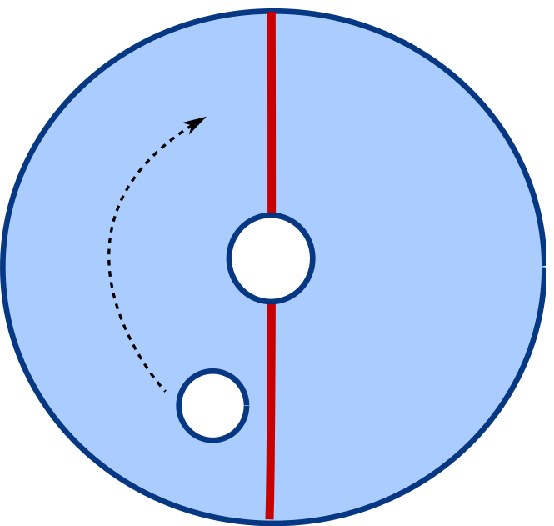}}
\hspace{3em}
\raisebox{8em}{c)}
\scalebox{0.6}{\includegraphics{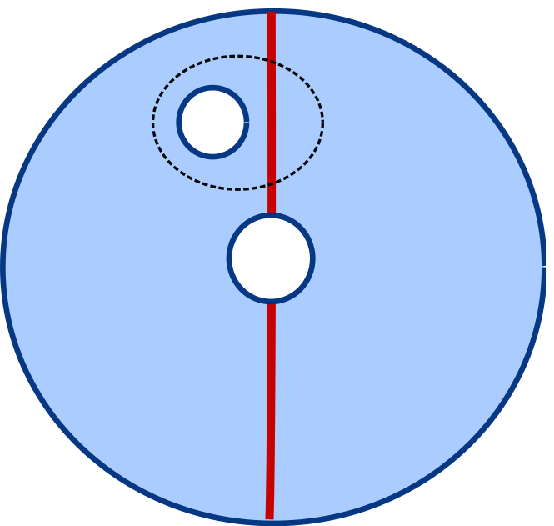}}
\end{center}
\caption{Here we use the same conventions as in the figures in Section~\ref{sec:intro-3}. Figures a), b) and c) above show how the sphere with two in-going and one out-going marked point as shown can be obtained at two different points in moduli space from sewing the building blocks in Figure~\ref{fig:defect-mult+inclusion}. The resulting identity of morphisms in the category $\CD = \mathrm{Rep}(V_L \otimes_\mathbb{C} V_R)$ is the first diagram in \eqref{eq:cospan-central} with $A = B_1$ and $T = D$.}
\label{fig:CFT-right-center-condition}
\end{figure}

\begin{rema}
(i) Comparing to conditions \eqref{eq:cospan-central} and \eqref{eq:lcp}, it follows that if the right and left centers of $T$ exist, the morphisms $a$ and $b$ in a cospan factor through them
\be
\raisebox{2em}{\xymatrix{ C_r(T) \ar[r] & T & C_l(T) \ar[l] \\ A\ar[ru]^a \ar[u] && B \ar[lu]_b \ar[u] } }
\quad .
\ee
(ii) Let us recall the reasoning behind the definition of cospans of commutative algebras from the introduction. There we saw that this structure naturally appears in three (related) examples: It describes how the centers of two algebras are related by the centralizer of an algebra map as in \eqref{cf}; it gives the relation between the full centers of module categories via the centralizer of a module functor as in \eqref{diag:Z-F}; it appears naturally in the context of conformal field theory in the presence of domain walls as in \eqref{eq:algcospan-via-CFT}. For example, in the CFT setting, the commutativity of the first diagram in \eqref{eq:cospan-central} is obtained from the sewing constraint illustrated in Figure~\ref{fig:CFT-right-center-condition}.
\\[.5em]
(iii) The reason that we do not allow {\em all} algebra maps $a$ and $b$ in the cospan $A \xrightarrow{a} T \xleftarrow{b} B$ is two-fold. 
Firstly, in the three examples mentioned in (i), the algebra maps that occur all satisfy the centrality condition (for the first two examples this will be proved in Section~\ref{sec:center-func}). Secondly, we want to define a composition of cospans (see \eqref{eq:cospan-comp} below), and to this end we need an algebra structure on a fibered tensor product $T \otimes_B S$ of two algebras. For this we need the action of $B$ to commute with the multiplication of $T$ and $S$, as we will see in more detail below.
\end{rema}

If $A$ is an algebra in $\CZ$, $M$ is a right $A$-module and $N$ a left $A$-module, Assumption~\ref{ass:coeq} guarantees the existence of the fibered tensor product $M \otimes_A N$ as the coequalizer 
\be
  \xymatrix{ M \otimes A \otimes N \ar@<+.7ex>[r]^{\hspace{0.3cm}L} \ar@<-.7ex>[r]_{\hspace{0.3cm}R} & M \otimes N  \ar[r]^\rho & M \otimes_A N} \ ,
\ee
where $L = \rho_M \otimes \id_N$, the right action of $A$ on $M$, and $R = \id_M \otimes \rho_N$. The common right inverse of $L$ and $R$ is $
\tau = \id_M \otimes \iota_A \otimes \id_N$. 

The aim of this subsection is to define a composition of cospans as in 
\be \label{eq:cospan-comp}
\raisebox{1.5em}{\xymatrix@R=1em{ 
& T && S \\ A\ar[ru]^a && B \ar[lu]_b \ar[ru]^{b'} && C \ar[lu]_c
}}
\quad \leadsto \quad
\raisebox{1.5em}{\xymatrix@R=1em{ 
& T\otimes_BS \\ A\ar[ru]^\alpha && C \ar[lu]_\gamma
}}
\ee
For notational convenience we have written the pair of cospans on the left hand side in the opposite order as indicated by $\cosp(B,C) \times \cosp(A,B)$.
The action of $B$ on $T$ and $S$ is defined via the algebra homomorphisms $b$ and $b'$, and
the two morphisms $\alpha$ and $\gamma$ are given by
\be \label{eq:A-T-C-S-TBS}
\alpha = \Big( A \xrightarrow{a} T \xrightarrow{1\,\iota_S} T\otimes S \xrightarrow{\rho} T\otimes_B S \Big)
~~,\quad   
\gamma = \Big(C \xrightarrow{c} S \xrightarrow{\iota_T\,1} T\otimes S \xrightarrow{\rho} T \otimes_B S\Big) \ .
\ee
That the composed cospan satisfies the properties of Definition~\ref{def:cospan} will be established in three steps.
First we define an algebra structure on $T \otimes_B S$ (Lemma~\ref{lem:TxBS-algebra}) -- that this can be done at all is ensured by the centrality conditions \eqref{eq:cospan-central} satisfied by $b$ and $b'$.
Next we show that
$\alpha$ and $\gamma$ are algebra homomorphisms (Lemma~\ref{lem:TxBS-al-be-hom}). 
Finally we verify that $\alpha$ and $\beta$ make the diagrams
\eqref{eq:cospan-central} commute (Lemma~\ref{lem:TxBS-center-condition}).

\medskip

Given two algebras $T$ and $S$ in a braided monoidal category, the tensor product $T \otimes S$ can be given an algebra structure with multiplication $\mu_{TS}$ and unit $\iota_{TS}$ given by
\be   \label{eq:TS-mu}
\mu_{TS} = \Big[(T\otimes S) \otimes (T\otimes S) \xrightarrow{1\,c_{S,T}1}  T\otimes T\otimes S\otimes S \xrightarrow{\mu_T \mu_S}  T\otimes S\Big]
\quad , \quad 
\iota_{TS} = \Big[1_\CZ \xrightarrow{\iota_T \iota_S}  T\otimes S\Big] \ .
\ee
Alternatively one could use the inverse braiding $c^{-1}_{T,S}$; this leads to an in general inequivalent algebra structure on $T\otimes S$. In the present paper we always choose the multiplication \eqref{eq:TS-mu} on $T \otimes S$.

\begin{lemma} \label{lem:TxBS-algebra}
With the notation of \eqref{eq:cospan-comp}, the fibered tensor product $T\otimes_B S$ has a unique algebra structure such that the coequalizer 
$T \otimes S \xrightarrow{\rho} T \otimes_B S$ becomes an algebra homomorphism.
\end{lemma}

\pf
Write
\bea \label{eq:TxBS-coeq-LR}
  L &=& \Big( T \otimes B \otimes S \xrightarrow{1\,b\,1} T\otimes T \otimes S \xrightarrow{\mu_T \,1} T \otimes S \Big) \ , \nonumber \\
  R &=& \Big( T \otimes B \otimes S \xrightarrow{1\,b'\,1} T\otimes S \otimes S \xrightarrow{1\,\mu_S} T \otimes S \Big) \ .
\eea
Note that $L$ and $R$ have the common right inverse $\tau = \id_T \otimes \iota_B \otimes \id_S$, so that we may apply
Lemma~\ref{lem:coeq-iterate} iteratively to obtain the three coequalizers (these are all we need)
\begin{subequations}
\bea 
&  \xymatrix{ T \otimes B \otimes S ~ \ar@<+.7ex>[r]^{\hspace{0.3cm}L} \ar@<-.7ex>[r]_{\hspace{0.3cm}R} & ~T \otimes S~  \ar[r]^\rho & ~T \otimes_B S} \ ,\label{eq:three-coequalizers-of-TS-1} \\
&  \xymatrix{ (T \otimes B \otimes S)^{\otimes 2} \ar@<+.7ex>[r]^{\hspace{0.3cm}LL} 
\ar@<-.7ex>[r]_{\hspace{0.3cm}RR} & 
    (T \otimes S)^{\otimes 2}  \ar[r]^{\rho\rho} & (T \otimes_B S)^{\otimes 2}} \ ,  \label{eq:three-coequalizers-of-TS-2} \\
&  \xymatrix{ (T \otimes B \otimes S)^{\otimes 3} \ar@<+.7ex>[r]^{\hspace{0.3cm}LLL} \ar@<-.7ex>[r]_{\hspace{0.3cm}RRR} & 
    (T \otimes S)^{\otimes 3}  \ar[r]^{\rho\rho\rho} & (T \otimes_B S)^{\otimes 3}} \ .
    \label{eq:three-coequalizers-of-TS-3}
\eea
\end{subequations}
The multiplication morphism $\mu : (T\otimes_BS)^{\otimes 2}\to T\otimes_BS$ will be constructed via the universal property of the coequalizer \eqref{eq:three-coequalizers-of-TS-2}. Namely, consider the following diagram,
\be \label{eq:TBS-alg}
\raisebox{2em}{\xymatrix{ 
(T\otimes B\otimes S)^{\otimes 2} \ar[d]_{\mu_{TBS}} \ar@<-.7ex>[r]_{\hspace{0.3cm}RR} \ar@<+.7ex>[r]^{\hspace{0.3cm}LL} & (T\otimes S)^{\otimes 2} \ar[r]^{\rho\rho} \ar[d]_{\mu_{TS}} & (T\otimes_B S)^{\otimes 2} \ar@{-->}[d]^{\exists ! \mu}  \\
T\otimes B\otimes S \ar@<-.7ex>[r]_{\hspace{0.3cm}R} \ar@<+.7ex>[r]^{\hspace{0.3cm}L}  & T\otimes S \ar[r]^{\rho} & T\otimes_B S
}}
\ee
We will show momentarily that the two overlaid squares on the left commute. 
But before that let us see how this 
completes the proof. Since the horizontal diagrams are coequalizer diagrams, and since the commutativity of the overlaid squares implies that $\rho \circ \mu_{TS} \circ (LL) = \rho \circ \mu_{TS} \circ (RR)$, by the universal property of the coequalizer $\rho\rho$ we get the existence and the uniqueness of $\mu$. Commutativity of the right square means that $\rho$ respects the multiplication. The unit of $T\otimes_B S$ is $\iota = \rho \circ (\iota_T \iota_S)$, so $\rho$ respects the unit by definition. 
To see that $\mu$ is associative, compose the associativity condition 
$\mu_{TS} \circ (\mu_{TS}\,1) = \mu_{TS} \circ (1\,\mu_{TS})$ of the algebra $T \otimes S$ with $\rho$ from the left. Using the commutativity of the right subdiagram of (\ref{eq:TBS-alg}), we can show $\mu \circ (\mu\,1) \circ (\rho\rho\rho) = \mu \circ (1\,\mu) \circ (\rho \rho \rho)$. By \eqref{eq:three-coequalizers-of-TS-3}, $\rho\rho\rho$ is a coequalizer, and from its universal property we conclude $\mu \circ (\mu\,1) = \mu \circ (1\,\mu)$, i.e.\ $\mu$ is associative. Similarly, the unit condition for $T \otimes S$ implies that of $T \otimes_B S$.

We now turn to commutativity of the two overlaid diagrams in \eqref{eq:TBS-alg}. We will show the equality 
$\mu_{TS} \circ (LL) = L \circ \mu_{TBS}$; 
the equality $\mu_{TS} \circ (RR) = R \otimes \mu_{TBS}$ 
follows analogously. One first convinces oneself (possibly by using the standard graphical notation) that it is enough to check that
\be
\xymatrix{ (T\otimes B)^{\otimes 2} \ar[rr]^{\hspace{0.4cm}(\mu_T \circ (1b))^{\otimes 2}} \ar[d]_{\mu_{TB}}  && 
T^{\otimes 2} \ar[d]^{\mu_T}  \\
T\otimes B \ar[rr]^{\mu_T \circ (1b)} && T
}
\ee
commutes. This in turn is implied by the commutativity of the diagram 
\be
\xygraph{ !{0;/r5.5pc/:;/u3.5pc/::}
*+{(TB)^2} 
(
:[dd] *+{T^2B^2} ^{1c1}
 (
 :[d] *+{TB} ^{\mu\mu}
 :[ll] *+{T^2}="1" _{1b}
 :[ll] *+{T}="2" _\mu
 ,
 :[ll] *+{T^4}="3" ^{11bb}
  (
  :"1" ^{\mu\mu}
  ,
  :[ul] *+{T^3}="4"  _{1\mu1}
  :"2" ^{\mu \circ (\mu \,1)}
  )
 ,
 :[l(1)u(.5)] *+{T^2BT}="5" _{111b}
 :"3" _{11b1}
 )
,
:[d(.5)l(1)] *+{TBT^2} ^{111b}
 (
 :"5" ^{1c1}
 ,
 :[l(1)u(.5)] *+{T^4}="6" ^{1b11}
 )
,
:"6" _{(1b)^2}
 (
 :[ll] *+{T^2} _{\mu\mu}
 :"2" _\mu
 ,
 :"4" _{1\mu 1}
 )
)
}
\ee
The inner pentagon commutes because by definition $b$ satisfies the second condition in \eqref{eq:cospan-central}. This completes the proof.
\epf

\begin{lemma} \label{lem:TxBS-al-be-hom}
With the notation of \eqref{eq:cospan-comp}, the maps $\alpha$ and $\gamma$ are algebra homomorphisms.
\end{lemma}

\pf
This is clear, as by Lemma~\ref{lem:TxBS-algebra}, the map $\rho$ is an algebra map, and so according to \eqref{eq:A-T-C-S-TBS}, $\alpha$ and $\gamma$ are compositions of algebra maps.
\epf

\begin{lemma} 
\label{lem:TxBS-center-condition}
With the notation of \eqref{eq:cospan-comp}, the maps $\alpha$ and $\gamma$ satisfy condition \eqref{eq:cospan-central}.
\end{lemma}

\begin{proof}
We will verify that the first of the two diagrams in \eqref{eq:cospan-central} commutes; the second one can be checked analogously. 
One quickly convinces oneself that the diagram
\be
\xygraph{  !{0;/r5.5pc/:;/u3.5pc/::}
*+{TSA}
(
:[d(4)] *+{ATS}="1" ^(.4){c_{TS,A}}
(
:[r(4)] *+{TSTS} _{a\,\iota_S 11}
  (
  :[r(0)u(2)] *+{TS}="2" _{\mu_{TS}}
  ,
  :[l(1)u(1)] *+{TTSS}="3"  ^{1\,c_{S,T}1}
  :"2" _(.4){\mu_T\,\mu_S}
  )
 ,
 :"3" _{a\,1\,\iota_S 1}
 ,
 :[u(1.5)r(2)] *+{TTS} ^{a\,11}
 :"2" ^{\mu_T 1}
 )
,
:[dr] *+{TAS} ^{1c_{S,A}} 
 (
 :"1" ^{c_{T,A}1}
 ,
 :[d(.5)r(1)] *+{TTS} _{1\,a\,1}
 :"2" ^(.4){\mu_T 1}
 ,
 :[r(2)] *+{TTSS}="4" ^{1\,a\,1\,\iota_S}
 :"2" ^{\mu_T\,\mu_S}
 )
,
:[r(4)] *+{TSTS} ^{11\,a\,\iota_S}
 (
 :"2" ^{\mu_{TS}}
 ,
 :"4" _{1\,c_{S,T} 1}
 )
)
}
\ee
commutes. In particular, the central pentagon commutes because by definition, $a$ makes the first diagram in \eqref{eq:cospan-central} commute. Composing the resulting equality of morphisms $(T \otimes S) \otimes A \to T \otimes S$ with $\rho$, using that $\rho$ is an algebra map and substituting the expression \eqref{eq:A-T-C-S-TBS} for $\alpha$, we obtain
\be
\xymatrix@R=1em{ 
(T\otimes S) \otimes A \ar[rr]^{\rho\,1} \ar[dd]_{\rho\,1} &&
(T\otimes_B S) \otimes A \ar[r]^{1\,\alpha}  & (T\otimes_B S)^{\otimes 2} \ar[rd]^\mu\\ && & &  T\otimes_B S\\ 
(T\otimes_B S) \otimes A \ar[rr]^{c_{ T\otimes_B S,A}}  &&
A \otimes (T\otimes_B S) \ar[r]^{\alpha\,1} & (T\otimes_B S)^{\otimes 2} \ar[ur]_\mu }
\ee
Since by Assumption~\ref{ass:coeq}, $\rho \otimes \id_A$ is a coequalizer, we obtain commutativity of the first diagram in \eqref{eq:cospan-central} (with $T \leadsto T\otimes_B S$ and $a \leadsto \alpha$).
\end{proof}

This completes the proof that the composition \eqref{eq:cospan-comp} is well-defined. We summarize this in the following proposition.

\begin{prop}\label{prop:cospan-composition}
Given two cospans $A \xrightarrow{a} T \xleftarrow{b} B$ and $B \xrightarrow{b'} S \xleftarrow{c} C$ between commutative algebras, then also $A \xrightarrow{\alpha} T \otimes_B S \xleftarrow{\gamma} C$, with $\alpha$ and $\gamma$ as in \eqref{eq:A-T-C-S-TBS}, is a cospan between commutative algebras.
\end{prop}

Finally, we state and prove the following lemma, which we will need later.

\begin{lemma} \label{lem:TBS-universal}
With the notation of \eqref{eq:cospan-comp}, the 
pair of algebra homomorphisms 
$T \xrightarrow{\rho \circ (1\iota)} T\otimes_B S \xleftarrow{\rho \circ (\iota1)} S$
is initial among all pairs of algebra homomorphisms $T \xrightarrow{w} C \xleftarrow{v} S$ which make 
the following two diagrams:
\be \label{eq:T-r-S-l}
\raisebox{2em}{
\xymatrix@R=1em{
& B \ar[dl]_{b} \ar[dr]^{b'} \\
T \ar[dr]_{w} && S \ar[dl]^{v} \\
& C}}
\quad \mbox{and} \quad
\raisebox{2em}{
\xymatrix@R=1em{
S \otimes T \ar[r]^{ v\otimes w}  \ar[dd]_{c_{S, T}} & C \otimes C \ar[rd]^{\mu} & \\
& & C \\
T \otimes S \ar[r]^{ w \otimes v} &   C \otimes C \ar[ru]_{\mu} &  
}}
\ee
commute. That is, for any such $T\xrightarrow{w} C \xleftarrow{v} S$, there is a unique algebra homomorphism $T\otimes_B S \xrightarrow{u} C$ such that the diagram
\be \label{eq:TBS-universal}
\raisebox{3.5em}{\xymatrix@R=1em{
& T\otimes_B S \ar@{-->}[dd]^{\exists !\, u} &  \\
T \ar[ur]^{\rho\circ (1\iota)} \ar[dr]_{w} & & S \ar[ul]_{\rho\circ (\iota1)} \ar[dl]^{v} \\
& C &
}}
\ee
commutes. 
\end{lemma}

\pf
We first check that the pair $T \xrightarrow{\rho \circ (1\iota)} T\otimes_B S \xleftarrow{\rho \circ (\iota1)} S$ itself satisfies the properties \eqref{eq:T-r-S-l}. The defining property of $\rho$ from \eqref{eq:TxBS-coeq-LR} and \eqref{eq:three-coequalizers-of-TS-1} is
\be
\Big( TBS \xrightarrow{1\,b\,1} TTS \xrightarrow{\mu_T 1} TS \xrightarrow{\rho} T \otimes_B S \Big)
~=~
\Big( TBS \xrightarrow{1\,b'\,1} TSS \xrightarrow{1\,\mu_S} TS \xrightarrow{\rho} T \otimes_B S \Big) \ .
\ee
Precomposing this with $\iota_T\,1\,\iota_S$ implies the first condition in \eqref{eq:T-r-S-l}. For the second condition in \eqref{eq:T-r-S-l}, one verifies that the two composed morphisms $S\otimes T \to C \equiv T \otimes_B S$ in the second diagram in \eqref{eq:T-r-S-l} are both equal to $\rho \circ c_{S,T}$. For example, for the upper path
\be
  \mu \circ ( (\rho \circ (1\,\iota_S)) \otimes (\rho \circ (\iota_T 1)))
  =  \rho \circ \mu_{TS} \circ (1\,\iota_S \iota_T 1)
  = \rho \circ c_{S,T} \ ,
\ee
using that $\rho$ is an algebra map (Lemma~\ref{lem:TxBS-algebra}) and the definition \eqref{eq:TS-mu} of $\mu_{TS}$.

Next we construct the map $u: T\otimes_B S \rightarrow C$ and prove that it has the required properties. Because $v$, $w$ are algebra homomorphisms and the first diagram in \eqref{eq:T-r-S-l} is commutative, we have $\mu \circ (w\otimes v) \circ L = \mu \circ (w\otimes v) \circ R$, where $L$ and $R$ are as in \eqref{eq:TxBS-coeq-LR}. Therefore the universal property of $\rho$ implies that there exists a unique $u$ which makes the square in
\be  \label{eq:rho-u-wv-mu}
\xymatrix{
T\otimes B \otimes S \ar@<+.7ex>[rr]^L \ar@<-.7ex>[rr]_R & & T\otimes S \ar[r]^\rho
\ar[d]_{w\, v} & T\otimes_B S \ar@{.>}[d]^{\exists ! \, u} \\
& & C\otimes C \ar[r]^\mu & C
}
\ee
commute. Precomposing the commuting square with $1\,\iota_S$ (resp.\ $\iota_T 1$) and using that $v$, $w$ are algebra maps gives commutativity of the left (resp.\ right) triangle in \eqref{eq:TBS-universal}. As the unit of $T \otimes_B S$ is $\rho \circ (\iota_T \iota_S)$, precomposing the above square with $\iota_T \iota_S$ shows that $u$ preserves the unit. To prove that it preserves the multiplication, we consider the following diagram:
\be \label{diag:u-alg-map-pf} 
\xymatrix{
(T\otimes S)^{\otimes 2} \ar[rr]^{\rho^{\otimes2}} \ar[ddd]_{\rho^{\otimes2}} \ar[ddr]_{(w\,v)^{\otimes 2}} \ar[drr]^{\mu_{TS}} 
  && (T\otimes_B S)^{\otimes 2} \ar[dr]^\mu 
  \\
& & TS \ar[r]^\rho \ar[d]^{w\,v}
& T \otimes_B S \ar[dd]^u
\\
& (C\otimes C)^{\otimes 2} \ar[dr]^{\mu^{\otimes2}} 
& C^{\otimes 2} \ar[dr]^\mu
\\
(T\otimes_B S)^{\otimes 2} \ar[rr]^{u^{\otimes2}}
&& C^{\otimes 2} \ar[r]^\mu
& C
}
\ee
in which the left and right lower square are commutative because of \eqref{eq:rho-u-wv-mu}, 
and the upper subdiagram is commutative because of \eqref{eq:TBS-alg}.  
Using the second commutative diagram in \eqref{eq:T-r-S-l}, it is easy to show that the middle pentagon commutes. 
Therefore, we obtain that the outer subdiagram in \eqref{diag:u-alg-map-pf} is commutative.
By the universal property of $\rho\rho$ (which is a coequalizer by Lemma~\ref{lem:coeq-iterate}), we obtain $u\circ \mu = \mu \circ (u \otimes u)$. Thus $u$ is an algebra homomorphism.

It remains to show that the solution $u$ to \eqref{eq:TBS-universal} is unique. By the uniqueness statement in \eqref{eq:rho-u-wv-mu}, it is enough to show that any algebra homomorphism $u$ making \eqref{eq:TBS-universal} commutative satisfies $u \circ \rho = \mu \circ (w\otimes v)$. This follows from the commuting diagram
\be
\xymatrix{
T\otimes S \ar[rr]_(.4){1\,\iota_S\iota_T 1} \ar@{=}@/_1em/[rrd] \ar@/^2em/[rrrr]^{w\,v} &&
(T\otimes S)^{\otimes2} \ar[r]_(.45){\rho\,\rho} \ar[d]^\mu & (T\otimes_B S)^{\otimes2} \ar[d]^\mu \ar[r]_(.55){u\,u} &  C \otimes C \ar[d]^\mu \\
&& T\otimes S \ar[r]^\rho & T\otimes_B S \ar[r]^u & C\ ,
}
\ee
where the upper cell commutes by \eqref{eq:TBS-universal} and the bottom squares commute since $\rho$ and $u$ are algebra homomorphisms.
\epf

\subsection{The bicategory $\CAlg(\CZ)$ of commutative algebras} \label{sec:CAlg}

In this subsection we introduce the bicategory $\CAlg(\CZ)$, whose objects are commutative algebras in $\CZ$. The morphism category between two such algebras $A$, $B$ is the category $\cosp(A,B)$, which we proceed to describe. 

\begin{defn} \label{def:CAlg(Z)}
Given two commutative algebras $A,B$ in a braided monoidal category $\CZ$ satisfying Assumption~\ref{ass:coeq}, $\cosp(A,B)$ is the following category:
\begin{itemize}
\item objects are cospans $A \rightarrow T \leftarrow B$ of commutative algebras as in Definition~\ref{def:cospan}.
\item a morphism from a cospan $A \rightarrow T \leftarrow B$ to a cospan $A \rightarrow T' \leftarrow B$ is an algebra map $f : T \to T'$ such that the following diagram commutes:
\be
\raisebox{2em}{\xymatrix@R=1em{
& T \ar[dd]^f \\
A \ar[ur]^a \ar[dr]_{a'} && B \ar[ul]_b \ar[dl]^{b'} \\
& T'  }}
\ee
\item the unit morphism for $A \rightarrow T \leftarrow B$ is the identity map $\id_T$, and the composition of morphisms is the composition of algebra maps.
\end{itemize}
\end{defn}

We remark here that in the next subsection we will define a {\em bi}category $\Cosp$ where objects are as above but morphisms will be categories of certain bimodules.

\medskip

The data and conditions defining a bicategory are listed in Definition~\ref{def:bicat}; we give the ingredients of $\CAlg(\CZ)$ in the same order as stated there. 
The objects of $\CAlg(\CZ)$ are commutative algebras in $\CZ$. 
The identity morphism $\one_A$, for $A \in \CZ$, has image in $\cosp(A,A)$ given by
\be \label{eq:CAlg-def-id}
\raisebox{2.5em}{\xymatrix@R=1em{
& A \ar[dd]^{\id} \\
A \ar[ur]^{\id} \ar[dr]_{\id} && A
\ar[ul]_{\id} \ar[dl]^{\id} \\
& A  }}
\qquad .
\ee
The composition functor $\circledcirc_{A,B,C} : \cosp(B,C) \times \cosp(A,B) \to \cosp(A,C)$ acts on objects and morphisms as
\be \label{eq:CAlg-def-comp}
\raisebox{2.5em}{\xymatrix@R=1em{
& T \ar[dd]^f & & S \ar[dd]^g \\
A \ar[ur]^a \ar[dr]_{a'} && B \ar[ul]_{b_1} \ar[dl]^{b_1'} \ar[ur]^{b_2} \ar[dr]_{b_2'} && C \ar[ul]_c \ar[dl]^{c'} \\
& T' && S' }}
\qquad \leadsto
\qquad
\raisebox{2.5em}{\xymatrix@R=1em{
& T \otimes_B S \ar[dd]^{f \otimes_B g} \\
A \ar[ur]^{\alpha} \ar[dr]_{\alpha'} && C \ar[ul]_\gamma \ar[dl]^{\gamma'} \\
& T'  \otimes_B S' }}
\quad .
\ee
Here we make use of the composition of cospans as stated in Proposition\,\ref{prop:cospan-composition}, where also the morphisms starting at $A$ and $C$ in the right diagram are given.
Recall from below \eqref{eq:cospan-comp} that our convention for the order in which we write the composition: $T \otimes_B S$ is the composition $\circledcirc_{A,B,C}(S,T)$.

The associativity natural isomorphism $\alpha_{A,B,C,D} : \circledcirc_{A,B,D} \circ (\circledcirc_{B,C,D} \times \id) \to \circledcirc_{A,C,D} \circ (\id \times \circledcirc_{A,B,C})$ between the two functors
\be
  \cosp(C,D) \times \cosp (B,C) \times \cosp (A,B) \longrightarrow \cosp(A,D)
\ee
acts on objects and morphisms as the canonical natural isomorphisms of the iterated fibered tensor product,
\be \label{eq:CAlg-def-ass}
(\alpha_{A,B,C,D})_{R,S,T} :
T \otimes_B (S \otimes_C R)  \longrightarrow
(T \otimes_B S) \otimes_C R \ .
\ee
The unit isomorphisms $l_{A,B} : \circledcirc_{A,B,B} \circ (\one_B \times \id) \to \id$ and $r_{A,B} : \circledcirc_{A,A,B} \circ (\id \times \one_A) \to \id$ in turn are the canonical isomorphisms
\be \label{eq:CAlg-def-unit}
(l_{A,B})_T : T \otimes_B B \to T
\qquad \text{and} \qquad
(r_{A,B})_T : A \otimes_A T  \to T \ .
\ee
It is then standard to check that these satisfy the coherence conditions of Definition~\ref{def:bicat}. Altogether, we obtain the following result. 
\begin{thm}
Let $\CZ$ be a braided monoidal category satisfying Assumption~\ref{ass:coeq}. Then
the following data defines a bicategory, which we call $\CAlg(\CZ)$:
\begin{itemize}
\item objects are commutative algebras $A,B,\dots$ in $\CZ$ and the category of morphisms from $A$ to $B$ is given by $\cosp(A,B)$,
\item identity, composition, associator and unit isomorphisms are as in \eqref{eq:CAlg-def-id}, \eqref{eq:CAlg-def-comp}, \eqref{eq:CAlg-def-ass}, and \eqref{eq:CAlg-def-unit}.
\end{itemize}
\end{thm}

\subsection{The bicategory $\Cosp(A,B)$ of cospans} \label{sec:Cosp(A,B)}

In this subsection we define a bicategory $\Cosp(A,B)$ obtained from the category $\cosp(A,B)$ by replacing the morphisms in Definition~\ref{def:CAlg(Z)} by the category of 2-diagrams, which we proceed to define.

\begin{defn} \label{def:2-diagram+3-cell}  
(i) A {\em 2-diagram} from a cospan $A \rightarrow S \leftarrow B$ to $A \rightarrow T \leftarrow B$ is 
a triple $S \xrightarrow{f} M \xleftarrow{g} T$, where $M$ is
a $T$-$S$-bimodule, $f$ is a right $S$-module map and $g$ is a left $T$-module map such that
the following three diagrams commute:
\be  \label{eq:2diagram-cond}
\raisebox{3.5em}{\xymatrix{
& S \ar[d]^f & \\ A \ar[ru]^{a_1} \ar[rd]_{a_2} & M & B \ar[ul]_{b_1} \ar[dl]^{b_2} \\
& T \ar[u]^g }}
\quad , \quad
\raisebox{3.5em}{\xymatrix{
A \otimes M \ar[dd]_{c_{M,A}^{-1}} \ar[r]^{a_2 1} & T \otimes M \ar[d]^{\text{act}}   \\
&  M \\
M\otimes A \ar[r]^{1\,a_1} & M \otimes S  \ar[u]_{\text{act}}  }}
\quad , \quad
\raisebox{3.5em}{\xymatrix{
B\otimes M \ar[dd]_{c_{B,M}} \ar[r]^{b_2 1} &  T\otimes M  \ar[d]^{\text{act}}   \\
 & M \\
M \otimes B \ar[r]^{1\,b_1} &  M\otimes S \ar[u]_{\text{act}}   }}
\quad .
\ee
(ii) A {\em 3-cell} between two 2-diagrams $S \xrightarrow{f} M \xleftarrow{g} T$ and $S \xrightarrow{f'} M' \xleftarrow{g'} T$ is a $T$-$S$-bimdule map $\delta : M \to M'$ such that the following diagram commutes:
\be
\raisebox{3em}{\xymatrix@R=1em{
& S \ar[ld]_f \ar[rd]^{f'} & \\ M \ar[rr]^\delta && M' \\
& T \ar[lu]^g \ar[ru]_{g'} }}
\ee
Invertible $3$-cells define equivalence classes of $2$-diagrams. We will call such an equivalence class a {\it $2$-cell}. 
\\[.5em]
(iii) The {\em category of 2-diagrams} $\tdiag_{AB}(S,T)$ has 2-diagrams as objects and 3-cells as morphisms. The identity 3-cell for the object $S \xrightarrow{f} M \xleftarrow{g} T$ is the identity map on $M$, the composition of 3-cells is given by composition of bimodule maps.
\end{defn}

Next we will define a composition functor $\circledcirc_{R,S,T} : \tdiag_{AB}(S,T) \times \tdiag_{AB}(R,S) \to \tdiag_{AB}(R,T)$ on objects and morphisms as
\be \label{eq:2diag-vertical-compos}
\raisebox{6.7em}{\xymatrix@C=1em{ 
&&& R \ar[dl]^e \ar[dr]^{e'}  \\
&& M \ar[rr]_{\alpha} && M' \\
A \ar@/^1.5em/[uurrr]^{a_1}\ar[rrr]^{a_2}\ar@/_1.5em/[ddrrr]_{a_3} &&& S \ar[ul]^f \ar[ur]^{f'}  \ar[dl]^g \ar[dr]^{g'}  &&& B
\ar@/_1.5em/[uulll]_{b_1}\ar[lll]_{b_2}\ar@/^1.5em/[ddlll]^{b_3} \\
&& N \ar[rr]^\beta && N' \\
&&& T \ar[ul]^h \ar[ur]^{h'} 
}}
\quad \longmapsto \quad
\raisebox{4.5em}{\xymatrix@C=.3em@R=3em{ 
&&& R \ar[dl]^u \ar[dr]^{u'}  \\
A \ar@/^1em/[urrr]^{a_1}\ar@/_1em/[drrr]_{a_3} && N \otimes_S M \ar[rr]^{\beta \otimes_S \alpha} && N' \otimes_S M' && B
\ar@/_1em/[ulll]_{b_1}\ar@/^1em/[dlll]^{b_3} \\
&&& T \ar[ul]^v \ar[ur]^{v'} 
}}
\ee
where
\begin{eqnarray}
u &=& \Big( R \xrightarrow{\iota_S \, e} S \otimes M \xrightarrow{g\,1} N \otimes M \xrightarrow{\rho} N \otimes_S M \Big) \ ,
\label{eq:2-diag-com-u-def} \\
v &=& \Big( T \xrightarrow{h\,\iota_S} N \otimes S \xrightarrow{1 \, f} N \otimes M \xrightarrow{\rho} N \otimes_S M \Big) \ ,
\label{eq:2-diag-com-v-def}
\end{eqnarray}
and analogous for $u'$ and $v'$. That this assignment is functorial (i.e.\ compatible with units and composition of 3-cells) amounts to the observation that $\id_N \otimes_S \id_M = \id_{N \otimes_S M}$ and $(\alpha' \circ \alpha) \otimes_S (\beta' \circ \beta) = (\alpha' \otimes_S \beta' ) \circ (\alpha \otimes_S \beta)$. It remains to check that the 2-diagrams and the 3-cell occurring on the right hand side of \eqref{eq:2diag-vertical-compos} obey the conditions of Definition~\ref{def:2-diagram+3-cell}. This is accomplished in the following two lemmas.

\begin{lemma}
In the notation of \eqref{eq:2diag-vertical-compos}, the following is a 2-diagram:
\be \label{eq:composed-2-diag-check}
\raisebox{3.5em}{\xymatrix@R=1em{
& R \ar[d]^u & \\ A \ar[ru]^{a_1} \ar[rd]_{a_3} & N \otimes_S M & B \ar[ul]_{b_1} \ar[dl]^{b_3} \\
& T \ar[u]^v }}
\ee
\end{lemma}

\pf
Since all maps in the definition of $u$ in \eqref{eq:2-diag-com-u-def} are right $R$-module maps, so is $u$. Analogously, $v$ is a left $T$-module map. The commutativity of the left subdiagram in \eqref{eq:composed-2-diag-check} follows from that of
\be
\raisebox{5em}{\xymatrix{ 
& R \ar[r]^e & 
  M \ar@/^2em/[rr]^{(g \circ \iota)\,1} &
  SS \ar[r]^{gf} &
  NM \ar[dr]^{\rho}
\\
A \ar[ur]^{a_1} \ar[rr]^{a_2} \ar[dr]^{a_3} &&
  S \ar[u]^f \ar[d]^g \ar[r]^{\iota\,1\,\iota} \ar[ur]^{\iota\,1} \ar[dr]_{1\,\iota} &
  SSS \ar[u]_{1\,\mu} \ar[d]^{\mu\,1} \ar[r]^{g\,1\,f} &
  NSM \ar[u]^{1\,\text{act}} \ar[d]_{\text{act}\,1} &
  N \otimes_S M
\\
& T \ar[r]^h & 
  N \ar@/_2em/[rr]_{1\,(f \circ \iota)}  &
  SS \ar[r]^{gf} &
  NM \ar[ur]_{\rho}
}}
\ee
The commutativity of the right subdiagram in \eqref{eq:composed-2-diag-check} can be checked similarly. To see that the second diagram in \eqref{eq:2diagram-cond} commutes first note that
\be
\raisebox{4em}{\xymatrix@R=.7em{ 
ANM \ar[r]^{a_3\,1\,1} \ar[dd]^{c_{N,A}^{-1}1} &
  TNM \ar[dr]^{\text{act}\,1}
\\  
&& NM \ar[dr]^{\rho} 
\\
NAM \ar[r]^{1\,a_2\,1} \ar[dd]^{1\,c_{M,A}^{-1}} &
  NSM \ar[ur]^{\text{act}\,1} \ar[dr]^{1\,\text{act}} &&
  N \otimes_S M 
\\
&& NM \ar[ur]^{\rho} 
\\
NMA \ar[r]^{1\,1\,a_1} &
  NMR \ar[ur]^{1\,\text{act}}
}}
\ee
commutes. Since $\rho$ is a bimodule map, this diagram implies the identity given by the second diagram in \eqref{eq:2diagram-cond}, but precomposed with $A\otimes (N \otimes M) \xrightarrow{1\,\rho} A \otimes (N \otimes_S M)$. The universal property of the coequalizer $1\rho$ now gives the desired identity. The commutativity of the third diagram in \eqref{eq:2diagram-cond} follows analogously. 
\epf

\begin{lemma}
In the notation of \eqref{eq:2diag-vertical-compos}, the following is a 3-cell:
\be
\raisebox{3em}{\xymatrix@R=1.5em{
& S \ar[ld]_u \ar[rd]^{u'} & \\ N\otimes_S M \ar[rr]^{\beta \otimes_S \alpha} && N' \otimes_S M' \\
& T \ar[lu]^v \ar[ru]_{v'} }}
\ee
\end{lemma}

\pf
Immediate from the definition of $u$ and $v$ in \eqref{eq:2-diag-com-u-def} and \eqref{eq:2-diag-com-v-def}, and from the fact that $\alpha$ and $\beta$ in \eqref{eq:2diag-vertical-compos} are 3-cells.
\epf

The image of the unit functor $\one_S : \one \to \tdiag_{AB}(S,S)$ 
is
\be \label{eq:2diag-identities}
\raisebox{3.5em}{\xymatrix@R=1em{
& S \ar@{=}[d] & \\
A \ar[ru]^{a} \ar[rd]_{a} & S & B \ar[ul]_{b} \ar[dl]^{b} \\
& S \ar@{=}[u] }}
\ee

\begin{prop}  \label{prop:Cosp-AB-bicat}
Let $\CZ$ be a braided monoidal category satisfying Assumption~\ref{ass:coeq} and let $A,B$ be commutative algebras in $\CZ$. Then
the following data defines a bicategory, which we call $\Cosp(A,B)$:
\begin{itemize}
\item objects are cospans $(A \rightarrow S \leftarrow B)$, $(A \rightarrow T \leftarrow B)$, \dots between the commutative algebras $A$ and $B$, and the morphism category $S \to T$ is given by $\tdiag(S,T)$,
\item identities and composition are given by \eqref{eq:2diag-identities} and \eqref{eq:2diag-vertical-compos},
\item the associativity isomorphism 
$\circledcirc_{Q,R,T} \circ (\circledcirc_{R,S,T} \times \id) \to \circledcirc_{Q,S,T} \circ (\id \times \circledcirc_{Q,R,S})$ is the canonical isomorphisms $(N \otimes_S M) \otimes_R L \to N \otimes_S (M \otimes_R L)$,
\item the left and right unit isomorphisms $\circledcirc_{S,T,T} \circ (\one_T \times \id) \to \id$ and $\circledcirc_{S,S,T} \circ (\id \times \one_S) \to \id$ are the canonical isomorphisms $T \otimes_T M \to M$ and $M \otimes_S S \to M$.
\end{itemize}
\end{prop}

To establish the proposition it remains to check the coherence conditions. As in the previous subsection, these are immediate implications from those for the fibered tensor product.

\subsection{Composition of cospans}\label{sec:compos-of-cospan}

In this subsection we will define a (non-lax) 2-functor
$$\Coco \equiv \Coco_{A,B,C} : \Cosp(B,C) \times \Cosp(A,B) \to \Cosp(A,C).$$ To this end we will give the data and verify the conditions as stated in Definition~\ref{def:lax-functor}. 

The objects of $\Cosp(B,C) \times \Cosp(A,B)$ are pairs of cospans. The first piece of data of the 2-functor is a map on objects; in this case it is just the composition of cospans defined in \eqref{eq:cospan-comp},
\be \label{eq:laxfun-data1}
\Coco ~:~
\raisebox{1.5em}{
\xymatrix@R=1em@C=1em{ & S && T \\ A\ar[ru]^(.4)a && B \ar[lu]_b \ar[ru]^{b'} && C \ar[lu]_(.4)c}
}
\quad \longmapsto \quad
\raisebox{1.5em}{
\xymatrix@R=1em@C=1em{ & S\otimes_B T \\ A\ar[ru]^(.35)\alpha && C \ar[lu]_(.35)\gamma}
}
\quad ,
\ee
where the ordering convention (putting the element of the second category in the product $\Cosp(B,C) \times \Cosp(A,B)$ in front of the first) is that already used in \eqref{eq:cospan-comp}; we will abbreviate such a pair as $(S,T)$. The right hand side is a cospan by Proposition\,\ref{prop:cospan-composition}.

Objects in the morphism category $\Mor\big( (S,T), (S',T') \big)$ are pairs of 2-diagrams, and morphisms are pairs of 3-cells. The second piece of data is a collection of functors $\Coco_{(S,T),(S',T')} : \Mor\big( (S,T), (S',T') \big) \to \Mor\big( \Coco(S,T) , \Coco(S',T') \big)$. They are given on objects and morphisms as follows
\be \label{eq:laxfun-data2}
\raisebox{3.6em}{\xymatrix@C=.6em{
& & S \ar[ld]^f \ar[rd]_{f'} && && T \ar[ld]^h \ar[rd]_{h'} \\
A \ar@/^.5em/[rru]^a \ar@/_.5em/[rrd]_{a'} & M \ar[rr]^{\phi} & & M' & 
  B \ar@/^.5em/[rru]^d \ar@/_.5em/[llu]_b \ar@/_.5em/[rrd]_{d'} \ar@/^.5em/[lld]^{b'} & N \ar[rr]^\psi & & N' & 
  C \ar@/_.5em/[llu]_c \ar@/^.5em/[lld]^{c'} \\
& & S' \ar[lu]_g \ar[ru]^{g'} && &&T' \ar[lu]_k \ar[ru]^{k'}
}}
~~\longmapsto~~
\raisebox{3.8em}{\xymatrix@C=0em{
&& & S{\otimes_B}T \ar[ld]_(.65){f \otimes_B h} \ar[rd]^(.65){f' \otimes_B h'}  \\ 
A \ar@/^1.5em/[rrru]^{\alpha} \ar@/_1.5em/[rrrd]_{\alpha'}  && M{\otimes_B}N \ar[rr]^{\hspace*{-1em}\phi \otimes_B \psi\hspace*{-1em}} && 
  M'{\otimes_B}N' && C \ar@/_1.5em/[lllu]_\gamma \ar@/^1.5em/[llld]^{\gamma'}\\
&& & S'{\otimes_B}T' \ar[lu]^(.65){g \otimes_B k} \ar[ru]_(.65){g' \otimes_B k'}  
}}
\ee
It is clear that the 3-cell on the right hand side is well-defined (i.e.\ the two triangles involving $\phi \otimes_B \psi$ commute), but for the two 2-diagrams the verification is more involved and is the content of Lemma~\ref{lem:laxfun-2diag-welldef} below. 

In any case, it is clear that $\Coco_{(S,T),(S',T')}$ defines a functor: the identity (i.e.\ the pair consisting of two identity 3-cells) gets mapped to the identity 3-cell, and compatibility with composition amounts to the identity $(\phi' \circ \phi) \otimes_B (\psi' \circ \psi) = (\phi' \otimes_B \psi') \circ (\phi \otimes_B \psi)$, which is a property of $\otimes_B$.

The third piece of data are the unit transformation, a collection of 3-cells $i_{(S,T)} : \one_{\Coco(S,T)} \to \Coco_{(S,T),(S,T)} \circ \one_{(S,T)}$. One quickly checks that both sides are equal to the 2-diagram
\be
\raisebox{3em}{\xymatrix@R=1.5em@C=1em{
& S\otimes_B T \ar[d]^(.5){1} & \\ A \ar@/^.6em/[ru] \ar@/_.6em/[rd] & S\otimes_B T & C \ar@/_.6em/[ul] \ar@/^.6em/[dl]\\
& S\otimes_B T \ar[u]_(.5){1}
}}
\ee
and we will choose
\be \label{eq:laxfun-data3}
  i_{(S,T)} = \id_{S \otimes_B T} \ .
\ee

The fourth and last piece of data are the multiplication transformations, a collection of natural transformations $m$ from
\be
\Coco_{(S_2T_2),(S_3T_3)} 
\Bigg(
\raisebox{2.3em}{\footnotesize \xymatrix@R=1em@C=1em{
& S_2 \ar[d] && T_2 \ar[d] \\
A \ar[ru] \ar[rd] & M' & B \ar[ru] \ar[lu] \ar[rd] \ar[ld] & N' & C \ar[lu] \ar[ld] \\
& S_3 \ar[u] && T_3 \ar[u]}}
\Bigg)
\circledcirc
\Coco_{(S_1T_1),(S_2T_2)}
\Bigg(
\raisebox{2.3em}{\footnotesize \xymatrix@R=1em@C=1em{
& S_1 \ar[d] && T_1 \ar[d] \\
A \ar[ru] \ar[rd] & M & B \ar[ru] \ar[lu] \ar[rd] \ar[ld] & N & C \ar[lu] \ar[ld] \\
& S_2 \ar[u] && T_2 \ar[u]}}
\Bigg)
\ee
to
\be
\Coco_{(S_1T_1),(S_3T_3)} 
\Bigg(
\raisebox{2.3em}{\footnotesize \xymatrix@R=1em@C=1em{
& S_1 \ar[d] && T_1 \ar[d] \\
A \ar[ru] \ar[rd] & M'{\otimes_{S_2}}M & B \ar[ru] \ar[lu] \ar[rd] \ar[ld] & N'{\otimes_{T_2}}N & C \ar[lu] \ar[ld] \\
& S_3 \ar[u] && T_3 \ar[u]}}
\Bigg)
\ee
That is, we have to provide a collection of 3-cells
\be \label{eq:Cosp-comp-3cells}
\xymatrix@C=1em{
& S_1 \otimes_B T_1 \ar[ld] \ar[rd] &  \\
(M'\otimes_B N') \otimes_{S_2 \otimes_B T_2} (M\otimes_B N) \ar[rr]^{\beta_{(M'N'),(MN)}}
&&
(M' \otimes_{S_2} M) \otimes_B (N'\otimes_{T_2} N) 
\\
& S_3 \otimes_B T_3 \ar[lu] \ar[ru] &
}
\ee
natural in $M$, $N$, $M'$, $N'$. The morphisms $\beta_{(M'N'),(MN)}$ are defined via coequalizers in the diagram
\be\label{eq:Cosp-comp-3cells-def}
\raisebox{2.3em}{\xymatrix{
M' N' M N \ar[r]^(.35){\rho_B\rho_B} \ar[d]^{1\,c_{N',M}1} &  (M'\otimes_B N')\otimes (M\otimes_B N) \ar[r]^(.46){\rho_{S_2\otimes_B T_2}}&
(M'\otimes_B N') \otimes_{S_2 \otimes_B T_2} (M\otimes_B N) \ar@{-->}[d]^{\exists! \, \beta_{(M'N'),(MN)}}
\\
M' M  N' N \ar[r]^(.35){\rho_{S_2}\rho_{T_2}}  & (M'\otimes_{S_2} M) \otimes (N'\otimes_{T_2} N) 
\ar[r]^{\rho_B} &
(M' \otimes_{S_2} M) \otimes_B (N'\otimes_{T_2} N) \ .
}}
\ee
In Lemma~\ref{lem:beta}, we will prove existence of $\beta_{(M'N'),(MN)}$ and show that it is an isomorphism and a 3-cell.

\medskip

For the rest of this subsection we will go through the announced lemmas in turn, leading to the proof that $\Coco$ is a 2-functor between bicategories (Proposition\,\ref{prop:C_ABC-lax-functor}).

\begin{lemma} \label{lem:laxfun-2diag-welldef}
In the notation of \eqref{eq:laxfun-data2}, the diagram
\be \label{eq:laxfun-2diag-welldef-aux}
\xymatrix{
& S\otimes_B T \ar[d]^(.55){f \otimes_B h} & \\ 
A \ar@/^.6em/[ru]^\alpha \ar@/_.6em/[rd]_{\alpha'} & M\otimes_B N & 
  C \ar@/_.6em/[ul]_\gamma \ar@/^.6em/[dl]^{\gamma'} \\
& S' \otimes_B T' \ar[u]^(.55){g \otimes_B k}
}
\ee
is a 2-diagram in the sense of Definition~\ref{def:2-diagram+3-cell}.
\end{lemma}

\begin{proof}
The requirements for a 2-diagram are as follows.
\\[.5em]
\nxt $M\otimes_B N$ is a $S' \otimes_B T'$-$S \otimes_B T$-bimodule:\\
The right action of $S \otimes_B T$ on $M\otimes_B N$ is defined via coequalizers by the following diagram:
\be \label{eq:SxBT-action-def}
\raisebox{2em}{\xymatrix{ 
M{\otimes}B{\otimes}N{\otimes}S{\otimes}B{\otimes}T 
  \ar[d]_{R_{SBT}} \ar@<-.7ex>[rr]_{RR} \ar@<+.7ex>[rr]^{LL}
&& M{\otimes}N{\otimes}S{\otimes}T  \ar[r]^(.4){\rho\rho} \ar[d]_{R_{ST}} & (M{\otimes_B}N) \otimes (S{\otimes_B}T) \ar@{-->}[d]^{\exists ! R_{S\otimes_BT}}  \\
M{\otimes}B{\otimes}N \ar@<-.7ex>[rr]_{R} \ar@<+.7ex>[rr]^{L}  && M\otimes N \ar[r]^(.4){\rho} & M\otimes_B N
}}
\ee
where $R_{SBT}$ is the right action of $S\otimes B\otimes T$ on $M \otimes B \otimes N$. It is defined by using $c_{B \otimes N,S}$ to braid $S$ past $B \otimes N$ and then acting on $M$, and analogously for $B$ (acting on $B$) and $T$ (acting on $N$). The right action of
  $S \otimes T$ 
on $M\otimes N$ is denoted by  $R_{ST}$. The proof that the two overlaid squares (one involving $LL$ and $L$, and the other $RR$ and $R$) commute and that $\sigma$ indeed defines an action is parallel to that of Lemma~\ref{lem:TxBS-algebra} and we omit the details; we just mention that it also uses the fact that commutativity of the $L$-part of the left square uses the left-center property \eqref{eq:cospan-central} of the right $B$-action, and the commutativity of the $R$-square requires \eqref{eq:2diagram-cond} (the middle diagram with $A$ replaced by $B$ and $M$ by $N$) -- the easiest way to see this is to write out the conditions in the standard graphical notation.

The left $S' \otimes_B T'$-action is found analogously, and one checks that the two actions commute.
\\[.5em]
\nxt $f \otimes_B h$ is a right $S \otimes_B T$-module map, and $g \otimes_B k$ is a left $S' \otimes_B T'$-module map:\\
With $R_{ST}$ as above and $\mu_{ST}$ as in \eqref{eq:TS-mu}, one easily checks that the two morphisms
$(f \otimes g) \circ \mu_{ST}$ and $R_{ST} \circ (f \otimes g \otimes \id_{S \otimes T})$ are equal as morphisms $S\otimes T\otimes S\otimes T \to M\otimes N$. Using the universal property of $(-)\otimes_B (-)$ (as defined via coequalizers), one finds $(f \otimes_B g) \circ \mu = R_{S\otimes_ST} \circ ( (f \otimes_B g) \otimes \id_{S \otimes_B T})$ (with $\mu$ from \eqref{eq:TBS-alg} and $R_{S\otimes_ST}$
from \eqref{eq:SxBT-action-def}), as required. The corresponding property for $g \otimes_B k$ follows along the same lines.
\\[.5em]
\nxt The first diagram in \eqref{eq:2diagram-cond} commutes:\\
Commutativity of the left sub-diagram in \eqref{eq:2diagram-cond}, which in the present case is the left sub-diagram in \eqref{eq:laxfun-2diag-welldef-aux}, follows from commutativity of
\be
\raisebox{3.7em}{
\xymatrix{ 
& S \ar[r]^(.35){1\,\iota_T} & 
     S \otimes T  \ar[d]^{f h} \ar[r]^\rho & 
     S \otimes_B T \ar[d]^{f \otimes_B h} \\
A \ar[ur]^a \ar[dr]_{a'} &  &   
  M \otimes N \ar[r]^\rho & M \otimes_B N \\
& S' \ar[r]^(.35){1\,\iota_{T'}} & 
  S' \otimes T' \ar[u]_{g k} \ar[r]^\rho & 
  S' \otimes_B T' \ar[u]_{g \otimes_B k}  \ .
}}
\ee
To see that the hexagon in the above diagram commutes, one uses that $S \xrightarrow{f} M \xleftarrow{g} S'$ is a 2-diagram by assumption (see \eqref{eq:laxfun-data2}), and that 
   $h \circ \iota_T = k \circ \iota_{T'}$. 
The latter identity is obtained by composing the first diagram in \eqref{eq:2diagram-cond} with the unit of $A$ (say), together with the fact that $a_1$ and $a_2$ are algebra maps.

Commutativity of the right sub-diagram in \eqref{eq:2diagram-cond} can be verified similarly.
\\[.5em]
\nxt The second and third diagram in \eqref{eq:2diagram-cond} commute:\\
Consider the diagram
\be
\raisebox{3.5em}{\xymatrix{
A \otimes M \otimes N \ar[dd]_{c_{MN,A}^{-1}} \ar[rr]^{a' \, \iota_{T'}\, 1\,1} && S' \otimes T' \otimes M \otimes N \ar[d]^{L_{S'T'}}   \\
& & M \otimes N \\
M \otimes N \otimes A \ar[rr]^{1\,1 \,a\,\iota_T} && M \otimes N \otimes S \otimes T  \ar[u]_{R_{ST}}  }}
\ee
Using that $S \rightarrow M \leftarrow S'$ is a 2-diagram, one quickly checks that the above diagram commutes (again the standard graphical notation is helpful). Via the universal property of coequalizers one then shows the corresponding statement with $M \otimes_B N$, $S \otimes_B T$ and $S' \otimes_B T'$.
This gives commutativity of the second diagram in \eqref{eq:2diagram-cond}. The arguments for the third diagram are analogous.
\end{proof}

\begin{lemma} \label{lem:beta}
(i) There exists a unique morphism $\beta_{(M'N'),(MN)}$, 
such that \eqref{eq:Cosp-comp-3cells-def} commutes. \\
(ii) $\beta_{(M'N'),(MN)}$ is an isomorphism for all $M,M',N,N'$. \\
(iii) The diagram \eqref{eq:Cosp-comp-3cells} defines a 3-cell.  \\
(iv) $\beta_{(M'N'),(MN)}$
is natural in $(M,N)$ and $(M',N')$.
\end{lemma}

The proof of above lemma is long (but straightforward) and irrelevant to the rest of this paper. So we moved it to Appendix\,\ref{app:proof-lem:beta}.

\begin{prop} \label{prop:C_ABC-lax-functor}
$\Coco$ as given above is a (non-lax) $2$-functor. 
\end{prop}

The proof of this Proposition is moved to Appendix\,\ref{app:proof-prop:C_ABC-lax-functor}. As a consequence (recall Remark \ref{rema:truncated-functor}), the $2$-functor $\Coco$ automatically defines a $1$-functor 
$$
\underline{\Coco} \equiv \underline{\Coco_{A,B,C}} : \underline{\Cosp}(B,C) \times \underline{\Cosp}(A,B) \to \underline{\Cosp}(A,C).
$$

\subsection{The bicategory \underline{CALG}($\CZ$)} \label{sec:under-CALG-Z}

Consider two commutative algebras $A$, $B$ in $\CZ$ and let $A \rightarrow S \leftarrow B$ and $A \rightarrow T \leftarrow B$ be two cospans between commutative algebras. Two 2-diagrams from $S$ to $T$ are isomorphic iff there is an invertible 3-cell between them.
Recall from Definition~\ref{def:2-diagram+3-cell} the definition of the category $\tdiag_{AB}(S,T)$ of 2-diagrams and 3-cells.
We will denote the set of isomorphism classes of 2-diagrams, i.e. the set of 2-cells, by $\tdiagu_{AB}(S,T)$.

Similarly, we denote by $\Cospu(A,B)$ the category obtained from the bicategory $\Cosp(A,B)$ by taking the same objects but replacing the category of morphisms $\tdiag_{AB}(S,T)$ by the set $\tdiagu_{AB}(S,T)$. With these ingredients, we can now list the data of the bicategory $\CALGu(\CZ)$:
\begin{itemize}
\item The objects are commutative algebras $A,B,\dots \in \CZ$ and the category of morphisms from $A$ to $B$ is given by $\Cospu(A,B)$.
\item The identity morphism $\one_A : \one \to \Cospu(A,A)$ has image 
\be
\raisebox{3.5em}{\xymatrix@R=1em{
& A \ar[d]^{1} & \\ A \ar[ru]^{1} \ar[rd]_{1} & A & A \ar[ul]_{1} \ar[dl]^{1} \\
& A \ar[u]^{1} }}
\quad .
\ee
\item 
The composition functor is induced by the functor $\Coco_{A,B,C}$ from Proposition\,\ref{prop:C_ABC-lax-functor}. We use the same symbol to denote the resulting functor $\Coco : \Cospu(B,C) \times \Cospu(A,B) \to \Cospu(A,C)$. Note that $\Coco$ is well-defined on isomorphism classes of 2-diagrams because, in the notation of \eqref{eq:laxfun-data2}, if $\phi$ and $\psi$ are isomorphism, so is $\phi \otimes_B \psi$.
\item
Denoting an object of $\Cospu(C,D) \times \Cospu(B,C) \times \Cospu(A,B)$ by
\be
\raisebox{1em}{\xymatrix@R=1em@C=1em{
& T  && S && R\\
A \ar[ru] && B \ar[ru] \ar[lu] && C \ar[ru] \ar[lu] && D \ar[lu]
}} \quad ,
\ee
the associativity morphism from the image under $\Coco \circ (\Coco \times \id)$ to the image under $\Coco \circ (\id \times \Coco)$ is a $2$-cell given by the associator of the fibered product as  (cf.\ \eqref{eq:CAlg-def-ass})
\be
\raisebox{3.5em}{\xymatrix{
& T\otimes_B (S \otimes_C R) \ar[d]^{\alpha^{-1}} & \\ 
A \ar[ru] \ar[rd] & (T \otimes_B S)\otimes_C R & D \ar[lu] \ar[ld] \\
& (T \otimes_B S)\otimes_C R \ar@{=}[u] 
}}
\quad ,
\ee
\item
The unit isomorphisms $l_T : T \otimes_B B \to T$ and $r_T : A \otimes_A T  \to T$ are two $2$-cells given by the unit isomorphisms of the fibered product (cf.\ \eqref{eq:CAlg-def-unit})
\be    \label{eq:unit-iso-l-r-CALG}
\raisebox{2.3em}{\xymatrix@C=1em@R=1em{
& T\otimes_B B \ar[d]^{\cong} & \\ 
A \ar[ru] \ar[rd] & T & B \ar[lu] \ar[ld] \\
& T \ar@{=}[u] 
}}
\quad , \quad
\raisebox{2.3em}{\xymatrix@C=1em@R=1em{
& A\otimes_A T \ar[d]^{\cong} & \\ 
A \ar[ru] \ar[rd] & T & B \ar[lu] \ar[ld] \\
& T \ar@{=}[u] 
}}
\quad .
\ee
\end{itemize}
The associativity and unit coherence conditions are again implied by those of the fibered product. Altogether, we have

\begin{thm}
Let $\CZ$ be a braided monoidal category satisfying Assumption~\ref{ass:coeq}. Then the data listed above defines a bicategory, which we denote by $\CALGu(\CZ)$.
\end{thm}

The reason for the notation $\CALGu(\CZ)$ is the following conjecture, which we aim to return to in future works.

\begin{conj} \label{conj:tricat-CALG}  
If we take objects to be commutative algebras in $\CZ$ and the bicategories $\Cosp(-,-)$ as morphisms with composition 2-functor
defined as in Proposition\,\ref{prop:C_ABC-lax-functor}, for a suitable choice of coherence morphisms we obtain a tricategory $\CALG(\CZ)$ in the sense of \cite{gps, gurski}.
\end{conj}

\begin{rema} 
Although it seems natural to consider all 4-layer of structures in $\CALG(\CZ)$ as a conjectured tricategory mathematically, a 3-cell does not seem to have any natural physical meaning. For applications in 2-dimension RCFT, it is enough to focus our attention to the bicategory $\CALGu(\CZ)$. 
\end{rema}

\medskip

Let us compare the bicategory $\CALGu(\CZ)$ defined in this section to the bicategory $\CAlg(\CZ)$ defined in Section~\ref{sec:CAlg}. Passing from $\CAlg(\CZ)$ to $\CALGu(\CZ)$ amounts to an enlargement of the morphism spaces in the sense that there is a locally faithful, but in general not locally full, 2-functor (cf.\ App.\,\ref{app:bicategories} for definitions)
\be \label{eq:functor-E-CAlg-CALG}
  E : \CAlg(\CZ) \to \CALGu(\CZ) \ .
\ee   
Its data -- in the order of Definition~\ref{def:lax-functor} -- is given by:
\begin{itemize}
\item On objects $E$ is the identity map, $E(A) = A$.
\item The functor $E_{(A,B)} : \cosp(A,B) \to \Cospu(A,B)$ is given by
\be
E_{(A,B)} ~:~ 
\raisebox{2.4em}{\xymatrix@R=1em{
& T \ar[dd]^f \\
A \ar[ur]^a \ar[dr]_{a'} && B \ar[ul]_b \ar[dl]^{b'} \\
& T'  }}
\quad \longmapsto \quad
\raisebox{2.4em}{\xymatrix@R=1em{
& T \ar[d]^f \\
A \ar[ur]^a \ar[dr]_{a'} & T'& B \ar[ul]_b \ar[dl]^{b'} \\
& T' \ar@{=}[u] }}
\ee
\item The unit transformation between the two functors $\one \to \Cospu(A,A)$ and the multiplication transformation between the two functors 
$\cosp (B,C) \times \cosp (A,B) \longrightarrow \Cospu(A,C)$ are both given by identity morphisms.
\end{itemize}
The coherence conditions are then trivially satisfied. Since $i$ and $m$ are isomorphisms, $E$ is indeed a $2$-functor, not just a lax $2$-functor.

\section{The full center as a lax functor} \label{sec:center-func}

In this section, we will present the complete full center construction. In other words, we will put together all the ingredients that is needed for the functoriality of the full center construction $\BZ$. In Section~\ref{sec:module-to-cospan}, we construct the restriction of $\BZ$ on Hom categories and prove its functoriality. In Section~\ref{sec:composition-comp}, we construct the multiplication transformation of $\BZ$. In Section~\ref{sec:lax-functor}, we state the functoriality of the full center construction $\BZ$ in a simple situation when $\CC=\Vect_k$ as a warm-up to the general cases. In Section~\ref{sec:lax-functor-II}, we give a general statement of the functoriality of the full center. 

\medskip
Throughout this section, we assume that $\CC$ is a monoidal category such that $\CZ(\CC)$ satisfies Assumption~\ref{ass:coeq}.

\subsection{From module functors to cospans} \label{sec:module-to-cospan}

Let $\CC$ be a monoidal category and let $\CM,\CN$ be $\CC$-modules. Recall from Section \ref{sec:cent-module-fun} that the category of $\CC$-module functors $\Fun_\CC(\CM,\CN)$ from $\CM$ to $\CN$ is a $\CZ(\CC)$-module. Let
\be
  \CF_{\CM,\CN} \subset \Fun_\CC(\CM,\CN)
\ee
be a subcategory of $\Fun_\CC(\CM,\CN)$ which is $\CZ(\CC)$-closed (recall Definition~\ref{def:C-closed}); in particular, $\CF_{\CM,\CN}$ need not be full or a $\CZ(\CC)$-submodule). We consider $\CF_{\CM,\CN}$ as a bicategory by adding identity 2-morphisms (the associativity and unit isomorphisms are hence also identities).

The aim of this section is to construct a lax 2-functor
$\BZ \equiv \BZ_{\CM,\CN}$ from $\CF_{\CM,\CN}$ to the bicategory $\Cosp(Z(\CM),Z(\CN))$ defined in Section~\ref{sec:Cosp(A,B)}. To a module functor $F \in \CF_{\CM,\CN}$ it assigns the cospan \eqref{eq:ZM-ZF-ZN-cospan}:
\be\label{eq:Z(F)-def}
\BZ(F) := \raisebox{2em}{\xymatrix{
& Z(F) \\ Z(\CM) \ar[ru]^{[F \circ -]_{[\id,\id]}}  &  & Z(\CN)\ar[ul]_{[- \circ F]_{[\id,\id]}}
}} \quad .
\ee
That this is indeed a cospan between commutative algebras, i.e.\ that it obeys the conditions in Definition~\ref{def:cospan}, follows form Lemma~\ref{lemma:F-C-functor} (showing that the two arrows are algebra maps) and Cor.\,\ref{cor:Z-[FF]-[FG]} in the case $F=G$ (showing that the two diagrams in \eqref{eq:cospan-central} commute). 

For each pair $F,G \in  \Fun_\CC(\CM,\CN)$ denote by $\Nat_\CC(F,G)$ the set of $\CC$-module natural transformations from $F$ to $G$. We will understand $\Nat_\CC(F,G)$ as a category with only identity morphisms. If $F,G \in \CF_{\CM,\CN}$, we can define a functor $\BZ_{(F,G)} : \Nat_\CC(F,G) \to \tdiag(\BZ(F),\BZ(G))$ via the following

\begin{lemma} 
Given $\phi \in \Nat_\CC(F,G)$,
the diagram 
\be\label{eq:Z_FG(phi)-def}
\BZ_{(F,G)}(\phi) ~:=~ \raisebox{3.8em}{\xymatrix{
& Z(F) \ar[d]^{[F,\phi]} & \\ Z(\CM) \ar[ru] \ar[rd] & [F,G] & Z(\CN)\ar[ul] \ar[dl]\\
& Z(G) \ar[u]^{[\phi, G]}
}}
\ee
is a 2-diagram (recall Definition~\ref{def:2-diagram+3-cell}). 
\end{lemma}

\pf
By Rem.\,\ref{rema:iso-int-hom}\,(i), $[F,G]$ is naturally a $Z(G)$-$Z(F)$-bimodule. By Lemma~\ref{lem:composition-properties}\,(iii), the map $[F,\phi]$ is a right $[F,F]$-module map (set $K=L=F$, $M=G$ and $f = \phi$), while the map $[\phi,G]$ is a left $[G,G]$-module map. The commutativity of \eqref{eq:Z_FG(phi)-def} follows from Lemma~\ref{lem:GF-nat-commute}, and the last two properties in \eqref{eq:2diagram-cond} hold by Cor.\,\ref{cor:Z-[FF]-[FG]}.
\epf

In order to define the multiplication transformation, we need the following construction. By the universal property of the coequalizer and the associativity of the composition (Lemma~\ref{lem:composition-properties}\,(ii)), there exists a unique morphism $\overline{\comp}$ which makes the diagram 
\be \label{eq:underline-comp-def}
\raisebox{2.3em}{
  \xymatrix{ [G,H] \otimes [G,G] \otimes [F,G] \ar@/^5pt/[rr]^{L} \ar@/_5pt/[rr]_{R} && 
  [G,H] \otimes [F,G]  \ar[r]^(.45)\rho \ar[d]^{\comp} & 
  [G,H] \otimes_{[G,G]} [F,G] \ar@{.>}[dl]^{\exists ! \,\overline{\comp}}
  \\
  && [F,H] 
}}
\ee
commute. We can use the coequalizer property of $[H,H] \otimes ([G,H] \otimes_{[G,G]} [F,G])$ and $([G,H] \otimes_{[G,G]} [F,G]) \otimes [F,F]$ (recall Assumption~\ref{ass:coeq}) to define a $[H,H]$-$[F,F]$-bimodule structure on $[G,H] \otimes_{[G,G]} [F,G]$. Along the same lines, one verifies that
\be \label{eq:mult-map-Phi}
  \overline{\comp} : [G,H] \otimes_{[G,G]} [F,G] \longrightarrow [F,H] 
\ee
as defined by \eqref{eq:underline-comp-def}, is a morphism of $[H,H]$-$[F,F]$-bimodules. Let now $F \xrightarrow{\phi} G \xrightarrow{\psi} H$ be two $\CC$-module natural transformations and consider the diagram 
\be \label{eq:Phi-as-3cell}
\raisebox{3.9em}{
\xymatrix{
& & Z(F) \ar[ld]_u \ar[rd]^{u'} & & \\
Z(\CM) \ar@/^1em/[rru] \ar@/_1em/[rrd] & [G,H]\otimes_{[G,G]} [F,G]  \ar[rr]^(0.6){\overline{\comp}} &  & [F, H] & Z(\CN) \ar@/_1em/[llu] \ar@/^1em/[lld] \\
& & Z(H) \ar[lu]^v \ar[ru]_{v'} & & 
}}
\quad ,
\ee
where $u' = [F, \psi \circ \phi]$ and $v' = [\psi \circ \phi, H]$. The maps $u$ and $v$ are obtained from the composition of 2-diagrams $\BZ_{(F,G)}(\phi)$ and $\BZ_{(G,H)}(\psi)$ as in \eqref{eq:2-diag-com-u-def} and \eqref{eq:2-diag-com-v-def}. 

\begin{lemma}
The diagram \eqref{eq:Phi-as-3cell} commutes.
\end{lemma}

\begin{proof}
It only remains to show that the two triangles involving $\overline{\comp}$ commute. Substituting the definition of $u$ in \eqref{eq:2-diag-com-u-def}, the upper triangle amounts to commutativity of the diagram
\be
\xymatrix@C=3em{
[F,F] \ar[r]^{[F,\phi]} \ar@/_4em/[drrr]^{[F,\psi \circ \phi]} &
  [F,G] \ar[r]^{\iota_{[G,G]}1} \ar@{=}[dr] &
  [G,G][F,G] \ar[d]^{\comp} \ar[r]^{[G,\psi]1} &
  [G,H][F,G] \ar[d]^{\comp} \ar[r]^\rho  &
[G,H] \otimes_{[G,G]} [F,G] \ar[dl]^{\overline{\comp}}
\\
 &  & [F,G] \ar[r]^{[F,\psi]} & [F,H] 
}
\ee
Here, the left triangle is the unit property \eqref{eq:unit-int-hom} of $\comp$, the middle square is Lemma~\ref{lem:composition-properties}\,(iii), and the right triangle is just the definition of $\overline{\comp}$. Commutativity of the lower triangle in \eqref{eq:Phi-as-3cell} is checked similarly. 
\end{proof}

\begin{prop} \label{prop:comp-inv-Z_mn-lax}
The map $\BZ_{\CM,\CN}$ on objects in \eqref{eq:Z(F)-def} and the collection of functors $\BZ_{(F,G)}$ on morphisms in \eqref{eq:Z_FG(phi)-def} with identities as unit transformations, and with $\overline{\comp}$ as multiplication transformation (recall \eqref{eq:mult-map-Phi} and \eqref{eq:Phi-as-3cell}) defines a lax 2-functor
of bicategories
\be
  \BZ_{\CM,\CN} ~:~ \CF_{\CM,\CN} \longrightarrow \Cosp(Z(\CM),Z(\CN)) \ .
\ee
\end{prop}

\begin{proof}
The unit conditions in Definition~\ref{def:lax-functor} follows from the definition of the unit isomorphisms of $\Cosp(Z(\CM),Z(\CN))$ given in Proposition\,\ref{prop:Cosp-AB-bicat}. It remains to verify the associativity condition in Definition~\ref{def:lax-functor}. This amounts to the statement that the two morphisms $\overline{\comp} \circ (\id \otimes \overline{\comp})$ and $\overline{\comp} \circ (\id \otimes \overline{\comp})$ from $[H,K] \otimes_{[H,H]} [G,H] \otimes_{[G,G]} [F,G]$ to $[F,K]$ (omitting associators) are equal. As $\overline{\comp}$ is defined in terms of $\comp$, this in turn is a consequence of the associativity of $\comp$.
\end{proof}

Recall Remark \ref{rema:truncated-functor}. We have the following result.
\begin{cor} \label{cor:Phi-iso-Z-functor}
If the morphisms $\overline{\comp}$ in \eqref{eq:mult-map-Phi} are isomorphisms for all $F,G,H \in  \CF_{\CM,\CN}$, then $\BZ_{\CM,\CN}$ gives a (non-lax) 2-functor
Moreover, we obtain a 1-functor $\underline{\BZ_{\CM,\CN}}$ between two $1$-categories:
\be
  \underline{\BZ_{\CM,\CN}} ~:~ \CF_{\CM,\CN} \longrightarrow \Cospu(Z(\CM),Z(\CN)) \ .
\ee
\end{cor}

\subsection{Compatibility with composition} \label{sec:composition-comp}

For this subsection, we fix three $\CC$-modules $\CL,\CM,\CN$ and three $\CZ(\CC)$-closed subcategories
\be
\CF_{\CL,\CM} \subset \Fun_\CC(\CL,\CM) ~,~~
\CF_{\CM,\CN} \subset \Fun_\CC(\CM,\CN) ~,~~
\CF_{\CL,\CN} \subset \Fun_\CC(\CL,\CN)\ ,
\ee
of which we demand in addition that the image of the composition $\CF_{\CM,\CN} \times \CF_{\CL,\CM} \to \Fun_\CC(\CL,\CN)$ lies in $\CF_{\CL,\CN}$. We want to establish a compatibility condition of the lax 2-functor
$\BZ_{\CM,\CN}$ defined in Proposition~\ref{prop:comp-inv-Z_mn-lax} and the composition of functors and of cospans (see Section~\ref{sec:compos-of-cospan}). Namely, we will give a natural transformation $\Bm$ between the (lax) 2-functors
\be \label{eq:composition-nat-xfer-m-diag}
\raisebox{2em}{\xymatrix{
\CF_{\CM,\CN} \times \CF_{\CL,\CM} \ar[d]_{(-)\circ(-)} \ar[rrr]^(0.34){\BZ_{\CM,\CN} \times \BZ_{\CL,\CM}} &&&
  \Cosp(Z(\CM),Z(\CN)) \times \Cosp(Z(\CL),Z(\CM)) \ar[d]^{ \Coco_{Z(\CL),Z(\CM),Z(\CN)} } \ar@{}[dlll]|{\Bm \swarrow} \\ 
\CF_{\CL,\CN} \ar[rrr]_{\BZ_{\CL,\CN} } &&& \Cosp(Z(\CL),Z(\CN))
}}
\quad .
\ee
According to Definition~\ref{def:nat-trans-fun-bicat}, such a transformation assigns a 1-morphism $\Bm_{(F,G)}$ to each pair $\CL \xrightarrow{F} \CM \xrightarrow{G} \CN$. We will take this 1-morphism to be
\be \label{eq:m_(F,G)-def}
\Bm_{(F,G)} ~:=~ \raisebox{3.8em}{\xymatrix{
& Z(F) \otimes_{Z(\CM)} Z(G) \ar[d]^{m_{F,G}} & \\ 
Z(\CL) \ar[ru]^{a_1} \ar[rd]_{a_2} & Z(G \circ F) & Z(\CN)\ar[ul]_{b_1} \ar[dl]^{b_2} \\
& Z(G \circ F) \ar[u]^{1}
}}
\quad ,
\ee
where $m_{F,G}$ will be defined in \eqref{eq:m_FG-def} below, and we will verify the properties of a 2-diagram in Lemma~\ref{lem:m(F,G)-is-2diag}. Note that $a_1, a_2, b_1, b_2$ are fixed by \eqref{eq:Z(F)-def} and the composition rule in (\ref{eq:cospan-comp}) -- explicit expressions will be given in \eqref{eq:a1-b2_for_m(F,G)} below.

The second piece of data, which corresponds to the data $\sigma_{AB}$ (or $\{ \sigma_f \}$) in \eqref{diag:lax-nat-1cell}, is a collection of 2-morphisms (i.e.\ 3-cells, cf.\ Definition~\ref{def:2-diagram+3-cell}) between two compositions of 2-diagrams. In the present case this boils down to the following: 
Let $F,F' : \CL \to \CM$ and $G,G': \CM \to \CN$. 
For a given pair $\big( F \xrightarrow{\phi} F'\,,\,G \xrightarrow{\psi} G'\big)$ in $\Nat_\CC(F,F') \times \Nat_\CC(G,G')$ we need a 3-cell 
$\Bm_{(\phi, \psi)}$ between the compositions
\be\label{eq:mFGFG-def-precomp}
\raisebox{7.5em}{\xymatrix@C=2em{
& Z(F) \otimes_{Z(\CM)} Z(G) \ar[d]^{[F,\phi] \otimes_{Z(\CM)} [G,\psi]} \\
& [F,F'] \otimes_{Z(\CM)} [G,G'] \\
Z(\CL) \ar@/^3.5em/[uur] \ar@/_2em/[ddr] \ar[r] & 
  Z(F') \otimes_{Z(\CM)} Z(G')  \ar[u]^{[\phi,F'] \otimes_{Z(\CM)} [\psi,G']} 
   \ar[d]^{m_{F',G'}} &
  Z(\CN) \ar[l] \ar@/_3.5em/[uul] \ar@/^2em/[ddl] \\
& Z(G'  F') \\
& Z(G'  F') \ar[u]^1
}}
~~\xrightarrow{\Bm_{(\phi, \psi)}}~~
\raisebox{7.5em}{\xymatrix@C=0.5em{
& Z(F) \otimes_{Z(\CM)} Z(G) \ar[d]^{m_{F,G}} \\
& Z(G  F) \\
Z(\CL) \ar@/^2em/[uur] \ar@/_2em/[ddr] \ar[r] & 
  Z(G  F) \ar[u]^{1} \ar[d]^{[GF,\psi\phi]} &
  Z(\CN) \ar[l] \ar@/_2em/[uul] \ar@/^2em/[ddl] \\
& [GF,G'  F'] \\
& Z(G'  F') \ar[u]^{[\psi\phi,G'F']}
}}
\ee
where the left diagram corresponds to $\sigma_B \circ \mathbf{F}(f)$ in \eqref{diag:lax-nat-2cell} and
the right diagram corresponds to $\mathbf{G}(f) \circ \sigma_A$ 
in \eqref{diag:lax-nat-2cell}. Also, we have not yet carried out the composition of 2-diagrams; this, together with the construction of $\Bm_{(\phi, \psi)}$ will be done below.

The data $\Bm_{(F,G)}$ and $\Bm_{(\phi, \psi)}$ have to satisfy the conditions stated in Definition~\ref{def:nat-trans-fun-bicat}. This is the content of Proposition\,\ref{prop:nat-tran-m} below.

\medskip

To construct $\Bm_{(F,G)}$ and $\Bm_{(\phi, \psi)}$ we first introduce the auxiliary map $n_{F,F',G,G'}$ which is defined by the following coequalizer diagram,
\be \label{eq:nFFGG-def}
\raisebox{2.3em}{
  \xymatrix{ [F,F'] \otimes Z(\CM) \otimes [G,G'] \ar@<+.7ex>[r]^{L} \ar@<-.7ex>[r]_{R} & 
  [F,F'] \otimes [G,G']  \ar[r]^(.45)\rho \ar[d]^{[G'\circ-]_{[F,F']}\, \otimes \,[- \circ F]_{[G,G']}} & 
  [F,F'] \otimes_{Z(\CM)} [G,G'] \ar@{.>}[d]^{\exists ! n_{F,F',G,G'}}
  \\
  & [G'F,G'F']\otimes[GF,G'F] \ar[r]^{\comp} & [GF,G'F'] 
}}
\ee
where the composite morphism at the bottom satisfies the coequalizer property: 
$$
\comp \circ ([G'\circ-]_{[F,F']}\, \otimes \,[- \circ F]_{[G,G']}) \circ L = \comp \circ ([G'\circ-]_{[F,F']}\, \otimes \,[- \circ F]_{[G,G']}) \circ R \ .
$$ 
This follows from the commutative diagram \eqref{eq:C-functor} and the associativity of $\comp$ (Lemma~\ref{lem:composition-properties}).

\begin{lemma} \label{lem:nFGFG-assoc}
Given $\CC$-module functors $F, F' \in \CF_{\CL,\CM}$, $G, G' \in \CF_{\CM,\CN}$ and $H, H'\in 
\CF_{\CN,\CP}$, the morphism $n$ satisfies the following associativity condition:
\be
\raisebox{2.3em}{\xymatrix{
([F,F'] \otimes_{Z(\CM)} [G,G']) \otimes_{Z(\CN)} [H,H']
\ar[rr]^{~~~~~~~n_{F,F',G,G'}\,1} \ar[d]^{1\,n_{G,G',H,H'}}  && 
[GF,G'F']  \otimes_{Z(\CN)} [H,H']
\ar[d]_{n_{GF,G'F',H,H'}}
\\
[F,F'] \otimes_{Z(\CM)} [HG,H'G']
 \ar[rr]^{n_{F,F',HG,H'G'}} && [HGF,H'G'F']
}}
\ee
\end{lemma}

\begin{proof}
It follows easily from the commutative diagram (\ref{eq:C-functor}), the associativity (\ref{eq:asso-int-hom}) and the universal properties of tensor products $\otimes_{Z(\CM)}$ and $\otimes_{Z(\CN)}$. 
\end{proof}

We define the map $m_{F,G}$ in \eqref{eq:m_(F,G)-def} to be
\be\label{eq:m_FG-def}
  m_{F,G} = n_{F,F,G,G} ~:~ Z(F) \otimes_{Z(\CM)} Z(G) \longrightarrow Z(GF) \ .
\ee
For the remainder of this section, it is convenient to abbreviate
\be
  Y := Z(F) \otimes_{Z(\CM)} Z(G) 
  \quad \text{and} \quad
  Y' := Z(F') \otimes_{Z(\CM)} Z(G') \ .
\ee

\begin{lemma} \label{lem:m(F,G)-is-2diag}
(i) $m_{F,G}$ is an algebra homomorphism.\\
(ii) Eqn.\,\eqref{eq:m_(F,G)-def} defines a 2-diagram.
\end{lemma}

\begin{proof}
(i) Consider the diagram
\be
\raisebox{3.9em}{
\xymatrix{
& [\id_\CM,\id_\CM] \ar[dl]_{[- \circ F]_{[\id,\id]}} \ar[dr]^{[G \circ -]_{[\id,\id]}} \ar[dd]^x \\
[F,F] \ar[dr]_{[G \circ -]_{[F,F]}} && [G,G] \ar[dl]^{[- \circ F]_{[G,G]}} \\
& [GF,GF]}}
\quad ,
\ee
where $x = [G \circ (-) \circ F]_{[\id,\id]}$. 
According to Lemma~\ref{lemma:F-C-functor}, all five maps in the diagram are algebra maps. 
The two triangles involving $x$ commute by Lemma~\ref{lemma:G-F=GF}. This shows that the first diagram in \eqref{eq:T-r-S-l} commutes. The second diagram in \eqref{eq:T-r-S-l} is commutative because of the commutative diagram \eqref{diag:comm} in Proposition\,\ref{prop:comm} (with $F=F'$ and $G=G'$). Thus we can apply Lemma~\ref{lem:TBS-universal}, which provides us with a unique algebra map $u : Y \to Z(GF)$. From the construction of $u$ in \eqref{eq:rho-u-wv-mu} we see that in fact $u = m_{F,G}$.

\medskip\noindent
(ii) The object $Z(GF)$ in $\CZ(\CC)$ becomes a $Z(GF)$-$Y$-bimodule by using the algebra map $m_{F,G}$ to define the $Y$-action. It is then clear that $m_{F,G}$ is a right $Y$-module map (and in any case $1$ is a left $Z(GF)$-module map). Consider the diagram
\be \label{eq:m(F,G)-is-2diag-aux1}
\raisebox{3.8em}{\xymatrix{
&&& Y \ar[d]^{m_{F,G}}  \\ 
Z(\CL) \ar@/^1.2em/[rrru]^{a_1} \ar@/_1.4em/[rrrd]_{a_2} \ar[rr]^{[F \circ -]_{[\id,\id]}} && 
  Z(F) \ar[ru]^{\rho \circ (1\,\iota)} \ar[rd]_{[G \circ -]_{[F,F]}} &
  Z(G  F) & 
  Z(G) \ar[ul]_{\rho \circ (\iota\,1)} \ar[dl]^{[- \circ F]_{[G,G]}} &&
  Z(\CN)\ar@/_1.2em/[ulll]_{b_1} \ar@/^1.4em/[dlll]^{b_2} \ar[ll]_{[- \circ G]_{[\id,\id]}} \\
&&& Z(G  F) \ar[u]^{1}
}}
\quad ,
\ee
where the maps $a_1,a_2,b_1,b_2$ are those in \eqref{eq:m_(F,G)-def}; explicitly they are given by
\begin{eqnarray} 
a_1 &=& \big(~Z(\CL) \xrightarrow{[F \circ -]_{[\id,\id]}} Z(F) \xrightarrow{1\,\iota} Z(F) \otimes Z(G) \xrightarrow{\rho} Y ~\big) \ , \nonumber \\
a_2 &=& \big(~ Z(\CL) \xrightarrow{[(GF) \circ -]_{[\id,\id]}} Z(GF) ~\big) \ , \nonumber \\
b_1 &=& \big(~ Z(\CN) \xrightarrow{[- \circ G]_{[\id,\id]}} Z(G) \xrightarrow{\iota\,1} Z(F) \otimes Z(G) \xrightarrow{\rho} Y ~\big) \ , \nonumber \\
b_2 &=& \big(~ Z(\CN) \xrightarrow{[- \circ (GF)]_{[\id,\id]}} Z(GF) ~\big) \ .
\label{eq:a1-b2_for_m(F,G)}
\end{eqnarray}
The triangles in \eqref{eq:m(F,G)-is-2diag-aux1} involving $a_1, b_1$ commute by definition. Those involving $a_2, b_2$ commute by Lemma~\ref{lemma:G-F=GF}. The remaining cells commute by the observation $u=m_{F,G}$ in part (i), as they just amount to \eqref{eq:TBS-universal}. This establishes the first of the three diagrams in \eqref{eq:2diagram-cond}. The second condition in \eqref{eq:2diagram-cond} follows form the commutativity of the diagram
\be
\raisebox{2.8em}{\xymatrix@C=2em@R=1.5em{
Z(\CL) \, Z(GF) \ar[dd]_{c^{-1}}  \ar[rr]^{a_2\,1} && Z(GF)\,Z(GF) \ar[rd]^{\comp}  &  \\
&&Z(GF)Z(GF) \ar[r]^{\comp} & Z(GF) \\
Z(GF)\,Z(\CL) \ar@/^1em/[urr]^{1\,a_2} \ar[rr]^{1\,a_1} && Z(GF) \,Y \ar[u]^{1\,m_{F,G}} \ar[ru]_{\text{act}} &
}}
\ee
The commutativity of the upper square is the statement of Cor.\,\ref{cor:Z-[FF]-[FG]}, the commutativity of the bottom left triangle is equivalent to that of the left square in \eqref{eq:m_(F,G)-def}, which we already proved. The bottom right triangle is just the definition of the $Y$-action on $Z(GF)$. The third condition in \eqref{eq:2diagram-cond} can be proved similarly. 
\end{proof}

By Lemma~\ref{lem:laxfun-2diag-welldef}, $[F, F']\otimes_{Z(\CM)} [G, G']$ is a $Y'$-$Y$-bimodule. On the other hand, $[GF,G'F']$ is a $Z(G'F')$-$Z(GF)$-bimodule. Since by Lemma~\ref{lem:m(F,G)-is-2diag}, the maps $m_{F,G} : Y \to Z(GF)$ and $m_{F',G'} : Y' \to Z(G'F')$ are algebra maps, we also obtain a $Y'$-$Y$-bimodule structure on $[GF,G'F']$.

\begin{lemma} \label{lemma:m-FF'-GG'}
The map $n_{F,F'; G,G'}$ in \eqref{eq:nFFGG-def} is a $Y'$-$Y$-bimodule morphism. 
\end{lemma}

\pf
The proof is exactly parallel to that of $u$ in \eqref{diag:u-alg-map-pf} being an algebra map in the proof of Lem\,\ref{lem:TBS-universal} but with $u$ replaced by $n_{F,F'; G,G'}$. We only sketch the idea. Consider a diagram similar to (\ref{diag:u-alg-map-pf}).
	The two maps $\mu$ leading to the bottom right corner are replaced by the left module action $[G'F',G'F']\otimes [GF, G'F']\to [GF, G'F']$, and the commutativity of the central pentagon follows form \eqref{diag:comm}.
Then we obtain the commutativity of the outer subdiagram. By the universality of $\rho^{\otimes 2}$, we are done. 
\epf

Having completed the definition of $\Bm_{(F,G)}$ in \eqref{eq:m_(F,G)-def}, we now turn to $\Bm_{(\phi, \psi)}$ in \eqref{eq:mFGFG-def-precomp}. To start with, we work out the compositions of 2-diagrams implicit in \eqref{eq:mFGFG-def-precomp}. From this we see that $\Bm_{(\phi, \psi)}$ needs to provide a 3-cell for
\be\label{eq:mFGFG-def-withcomp}
\raisebox{3.7em}{\xymatrix@C=1em{
& Y \ar[ld]_u \ar[rd]^{u'} &  \\
Z(G'F') \otimes_{Y'} \big( [F,F'] \otimes_{Z(\CM)} [G,G'] \big) \ar[rr]^{\hspace{0.5cm}\Bm_{(\phi, \psi)}}
&&
[GF,G'F'] \otimes_{Z(GF)} Z(GF)
\\
& Z(G'F') \ar[lu]^v \ar[ru]_{v'} &
}}
\ee
The explicit form of the maps $u,u',v,v'$ follows from \eqref{eq:2-diag-com-u-def} and (\ref{eq:2-diag-com-v-def}) depending on the pair $(\phi, \psi)$. To construct the morphism $\Bm_{(\phi, \psi)}$, consider the coequalizer diagram (we abbreviate $Q  = [F,F'] \otimes_{Z(\CM)} [G,G']$)
\be  \label{diag:def-m'}
\raisebox{2.3em}{\xymatrix{
Z(G'F') \otimes Y' \otimes Q
  \ar@<+.7ex>[r]^{L} \ar@<-.7ex>[r]_{R} & 
Z(G'F')  \otimes Q
  \ar[r]^(.45)\rho \ar[d]^{1\,n_{F,F',G,G'}}  &
Z(G'F') \otimes_{Y'} Q
   \ar@{.>}[d]^{\exists ! m'_{F,F',G,G'}}
\\
& Z(G'F')  \otimes [GF,G'F'] \ar[r]^{\hspace{1cm}\comp} & [GF,G'F']
}}
\ee
For $A$ an algebra and $M$ a left $A$-module, denote by $r$ the isomorphism $A \otimes_A M \to M$, so that
$r^{-1} = M \xrightarrow{\iota_A\,1} A \otimes M \xrightarrow{~\rho~} A \otimes_A M$. We define
\be  \label{eq:mFFGG-def}
m_{F,F',G,G'} :=  r^{-1} \circ m'_{F,F',G,G'}.
\ee
We finally define $\Bm_{(\phi, \psi)}$ in \eqref{eq:mFGFG-def-withcomp} by
\be \label{eq:m-phipsi-def}
  \Bm_{(\phi, \psi)} := m_{F,F',G,G'} = r^{-1} \circ m'_{F,F',G,G'} \ .
\ee

\begin{rema}
Notice that the definition of $\Bm_{(\phi, \psi)}$ in \eqref{eq:m-phipsi-def} is independent of the pair $(\phi, \psi)$. This is due to the speciality of our definition of $\BZ_{(F,G)}(\phi)$ in \eqref{eq:Z_FG(phi)-def}. For this reason, we will also use the notation $m_{F,F',G,G'}$ for $\Bm_{(\phi, \psi)}$ in a few places when we want to emphasize this independence of the pair $(\phi, \psi)$. 
\end{rema}

Following a parallel argument as in the proof of Lem\,\ref{lemma:m-FF'-GG'}, it is easy to see that $m_{F,F',G,G'}$ is a $Z(G'F')$-$Y$-bimodule map.

\begin{lemma}  \label{lem:mFGFG-3cell} 
Eqn.\,(\ref{eq:mFGFG-def-withcomp}) defines a 3-cell.
\end{lemma}

\pf
It remains to show that two triangles in \eqref{eq:mFGFG-def-withcomp} are commutative. For the upper triangle, using \eqref{diag:def-m'}, we obtain that 
\be  \label{eq:mu-n-phi-psi}
m_{F,F',G,G'} \circ u = r^{-1} \circ n_{F,F',G,G'} \circ ([F,\phi]\otimes_{Z(\CM)} [G,\psi]).
\ee
On the other hand, it is easy to show that $u'= r^{-1} \circ [GF, \psi\phi] \circ m_{F,G}$. Therefore, it is enough to show that 
$$
n_{F,F',G,G'} \circ ([F,\phi]\otimes_{Z(\CM)} [G,\psi]) = [GF, \psi\phi] \circ m_{F,G}.
$$
Composing both sides by $\rho: Z(F) \otimes Z(G) \to Y$ and using \eqref{eq:nFFGG-def}, we obtain 
\bea
n_{F,F',G,G'} \circ ([F,\phi]\otimes_{Z(\CM)} [G,\psi]) \circ \rho &=& n_{F,F',G,G'} \circ \rho \circ 
([F,\phi]\otimes [G,\psi])  \nn 
	&=& \comp \circ ([G' \circ -]\otimes [- \circ F]) \circ ([F,\phi]\otimes [G,\psi]), \nn
\left[ GF, \psi \phi\right] \circ m_{F,G} \circ \rho &=& [GF,\psi\phi] \circ \comp \circ ([G \circ -]\otimes [- \circ F]).
\eea
Therefore, it is enough to prove 
\be \label{eq:m-3-cell-pf-eq-1}
\comp \circ ([G' \circ -]\otimes [- \circ F]) \circ ([F,\phi]\otimes [G,\psi])
= [GF,\psi\phi] \circ \comp \circ ([G \circ -]\otimes [- \circ F]).
\ee

Consider the following commutative diagram: 
\be  \label{eq:m-3-cell-pf-eq-2}
\xymatrix{
& [F,F] \ast G'F \ar[rd]^{[G' \circ -] 1} &  \\
[F,F]\ast GF \ar[r]^{[G' \circ -]1} \ar[ur]^{1\ast \psi F} \ar[d]_{[G \circ -] 1} & [G'F, G'F]\ast GF \ar[r]^{1\ast \psi F} 
	\ar[d]_{[\psi F, G'F]1} &
[G'F,G'F]\ast G'F \ar[d]^{\ev} \\
[GF, GF]\ast GF \ar[r]^{[GF,\psi F]1} \ar[rd]_{\ev} & [GF, G'F] \ast GF \ar[r]^\ev & G'F \\
& GF \ar[ur]_{\psi F} & 
} 
\ee
where the left square is just \eqref{diag:GF-nat-commute}. The right and bottom square commute by definition, see \eqref{eq:def-fast-gast}. 
Using the first commutative diagram in (\ref{diag:[F]-fun-property}), the first commutative diagram in \eqref{eq:def-fast-gast} and the commutativity of outer subdiagram 
in (\ref{eq:m-3-cell-pf-eq-2}), one can show the commutativity of the following diagram:
\be
\xymatrix{
([F, F] \otimes [G, G]) \ast GF \ar[r] \ar[d]_{([F,\phi]\otimes [G,\psi])1} & GF \ar[d]^{\psi\phi} \\
([F, F'] \otimes [G, G']) \ast GF \ar[r] & G'F'
}
\ee
which implies (\ref{eq:m-3-cell-pf-eq-1}) immediately. Thus we have finished the proof of the commutativity of upper triangle in (\ref{eq:mFGFG-def-withcomp}).

\medskip
For the lower triangle in (\ref{eq:mFGFG-def-withcomp}), expanding the definition of $v, v', m_{F,F',G,G}$, it is enough to prove the commutativity of the following diagram: 
\be  \label{diag:iota-phipsi-m}
\xymatrix{
Z(G'F') \ar[r]^{\hspace{-2cm}1\iota} \ar[dd]_{[\psi\phi, G'F']} & Z(G'F') \otimes (Z(F')\otimes_{Z(\CM)} Z(G'))  \ar[d]^{1([\phi,F']\otimes_{Z(\CM)} [\psi,G'])} \\
& Z(G'F') \otimes ([F,F']\otimes_{Z(\CM)} [G,G'])  \ar[d]^{1n_{F,F';G,G'}} \\
[GF,G'F'] & Z(G'F') \otimes [GF,G'F']. \ar[l]_{\comp}
}
\ee
Using the first commutative diagram in (\ref{diag:[F]-fun-property}), the first commutative diagram in \eqref{eq:def-fast-gast} and the commutativity of outer subdiagram in (\ref{eq:m-3-cell-pf-eq-2}), we obtain the following commutative diagram:
$$
\xymatrix{
([F',F'] \otimes [G',G']) \ast GF \ar[r]^{1(\psi\phi)} \ar[d]_{([\phi,F'][\psi,G'])1}  & ([F',F'] \otimes [G',G']) \ast G'F' \ar[d]  \\
([F,F'] \otimes [G,G']) \ast GF \ar[r] & G'F'
}
$$
which, together with the following commutative diagram:
$$
\xymatrix{
(Z(F')\otimes_{Z(\CM)} Z(G')) \ast GF \ar[r]^{1(\psi\phi)}  \ar[d]_{m_{F',G'}1} & (Z(F')\otimes_{Z(\CM)} Z(G')) \ast G'F' \ar[d]_{m_{F',G'}1} \\
Z(G'F') \ast GF \ar[r]^{1(\psi\phi)}  \ar[d]_{[\psi\phi, G'F']1} & Z(G'F') \ast G'F' 
    \ar[d]^{\ev_{G'F'}} \\
[GF, G'F'] \ast GF \ar[r]^{\ev_{GF}} 
 \ar[r] & G'F'. 
}
$$ 
implies the following identity: 
\be  \label{eq:m-psiphi-psiphi-m}
n_{F,F';G,G'} \circ ([\phi,F'][\psi,G']) = [\psi\phi, G'F'] \circ m_{F',G'}. 
\ee
Applying (\ref{eq:m-psiphi-psiphi-m}) to diagram (\ref{diag:iota-phipsi-m}) and using the fact that $[\psi\phi, G'F']$ is a left $Z(G'F')$-module map and $m_{(F',G')}$ preserves the unit, we obtain the commutativity of (\ref{diag:iota-phipsi-m}) immediately. 

Therefore, we have proved that $\Bm_{(\phi, \psi)} = m_{F,F',G,G}$ is a well-defined 3-cell. 
\epf

We have now gathered all the ingredients to state the main result of this section.

\begin{prop}  \label{prop:nat-tran-m} 
Let $\CC$ be a monoidal category and let $\CL,\CM,\CN$ be $\CC$-modules. 
Let 
$$
\CF_{\CL,\CM} \subset \Fun_\CC(\CL,\CM) ~,~~
\CF_{\CM,\CN} \subset \Fun_\CC(\CM,\CN) ~,~~
\CF_{\CL,\CN} \subset \Fun_\CC(\CL,\CN)
$$
be $\CZ(\CC)$-closed subcategories such that the composition $\CF_{\CM,\CN} \times \CF_{\CL,\CM} \to \Fun_\CC(\CL,\CN)$ lies in $\CF_{\CL,\CN}$. 
The data \eqref{eq:m_(F,G)-def}, \eqref{eq:m_FG-def} and \eqref{eq:mFGFG-def-withcomp}, \eqref{eq:m-phipsi-def} defines a natural transformation $\Bm$ between the lax 2-functors stated in 
\eqref{eq:composition-nat-xfer-m-diag}.
\end{prop}

\pf
It remains to show that properties (\ref{diag:lax-nat-axiom-1}) and (\ref{diag:lax-nat-axiom-2}) hold. 
Because two $i$ maps in (\ref{diag:lax-nat-axiom-2}) are identity maps in this case, the condition (\ref{diag:lax-nat-axiom-2}) is reduced to the condition $m_{F,F, G,G}= r^{-1}\circ l$ for $F=G=\id_{\CM}$,
which holds obviously. For the property (\ref{diag:lax-nat-axiom-1}), we consider $\CC$-module functors $F,F',F'': \CL \to \CM$ and $G,G',G'': \CM \to \CN$, and natural transformations $F\xrightarrow{\phi} F'\xrightarrow{\phi'} F''$ and $G\xrightarrow{\psi} G' \xrightarrow{\psi'} G''$. Let $Y'':=Z(F'')\otimes_{Z(\CM)} Z(G'')$. Then the property (\ref{diag:lax-nat-axiom-1}) amounts to showing that the following composition of maps: 
\bea
&&\hspace{-0.5cm}Z(G''F'') \otimes_{Y''} ([F', F'']\otimes_{Z(\CM)} [G', G'']) 
\otimes_{Y'} ([F, F']\otimes_{Z(\CM)} [G, G'])   \nn
&& \xrightarrow{1\beta} Z(G''F'') \otimes_{Y''} ([F', F''] \otimes_{Z(F')} [F, F']) \otimes_{Z(\CM)} ([G',G'']\otimes_{Z(G')} [G,G']) \nn
&& \xrightarrow{1\otimes_{Y''}\overline{\comp}\otimes_{Z(\CM)}\overline{\comp}} Z(G''F'') \otimes_{Y''} [F,F'']\otimes_{Z(\CM)} [G, G''] \nn
&& \xrightarrow{m_{F,F'', G,G''}} [GF, G''F''] \otimes_{Z(GF)} Z(FG) \ , 
\eea
where $\beta:=\beta_{([F',F''],[G,G'']), ([F,F'], [G,G'])}$ (recall (\ref{eq:Cosp-comp-3cells})) is constructed explicitly in \ref{app:proof-lem:beta}, 
equals to the following composition of maps:
\bea
&&\hspace{-0.5cm}Z(G''F'') \otimes_{Y''} ([F', F'']\otimes_{Z(\CM)} [G, G']) \otimes_{Y'} ([F, F']\otimes_{Z(\CM)} [G, G'])   \nn
&& \xrightarrow{m_{F',F'',G',G''}1} 
[G'F', G''F'']\otimes_{Z(G'F')} Z(G'F') \otimes_{Y'} ([F,F']\otimes_{Z(\CM)} [G,G']) \nn
&& \xrightarrow{1m_{F,F',G,G'}} [G'F', G''F'']\otimes_{Z(G'F')} [GF, G'F'] \otimes_{Z(GF)} Z(GF) \nn
&& \xrightarrow{\overline{\comp}1} [GF, G''F''] \otimes_{Z(GF)} Z(FG)
\eea
Using the commutativity (\ref{diag:action-comm}), it is easy to check that both maps are induced from the same natural map: $(Z(G''F'')\otimes [F',F'']\otimes [G',G''] \otimes [F,F'] \otimes [G,G']) \ast GF \to G''F''$.
\epf

Recall Remark \ref{rema:truncated-nat-transf}. We have the following result. 
\begin{cor} \label{cor:m-iso-m-natxfer}
If the categories $\CF_{\CL,\CM}$, $\CF_{\CM,\CN}$, $\CF_{\CL,\CN}$ all meet the conditions of Cor.\,\ref{cor:Phi-iso-Z-functor}, and if the $\Bm_{(\phi, \psi)}$ in \eqref{eq:m-phipsi-def} are all isomorphisms, then $\Bm$ in Proposition\,\ref{prop:nat-tran-m} defines a natural transformation $\underline{\Bm}$ of functors between (1-)categories as in
(note the underlines)
\be \label{eq:composition-nat-xfer-m-diag-1cat-level}
\raisebox{2em}{\xymatrix{
\CF_{\CM,\CN} \times \CF_{\CL,\CM} \ar[d]_{(-)\circ(-)} \ar[rrr]^(0.34){\underline{\BZ_{\CM,\CN}} \times \underline{\BZ_{\CL,\CM}}} &&&
  \Cospu(Z(\CM),Z(\CN)) \times \Cospu(Z(\CL),Z(\CM)) \ar[d]^{ \underline{\Coco_{Z(\CL),Z(\CM),Z(\CN)}} } \ar@{}[dlll]|{\underline{\Bm}\,\, \swarrow} \\ 
\CF_{\CL,\CN} \ar[rrr]_{\underline{\BZ_{\CL,\CN}} } &&& \Cospu(Z(\CL),Z(\CN))~.
}}
\ee
Here $\underline{\Bm}$ is just a collection of equivalence classes of $2$-diagrams as shown in \eqref{eq:m_(F,G)-def} for all $G\in \CF_{\CM,\CN}$ and $F\in \CF_{\CL,\CM}$. 
\end{cor}

\subsection{The full center as a lax 2-functor, I}\label{sec:lax-functor}

\begin{defn} \label{def:1-cat-alg}
For a monoidal category $\CC$, let $\alg(\CC)$ be the category whose objects are unital associative algebras in $\CC$ and whose morphisms are unital algebra homomorphisms. If $\CC$ is braided, we define $\calg(\CC)$ to be the full subcategory of $\alg(\CC)$ consisting of only commutative $\CC$-algebras. 
\end{defn}

In this subsection, we will construct a lax 2-functor 
$\BZ$ from $\alg(\CC)$ to $\CAlg(\CZ(\CC))$. For concreteness, we start by taking the monoidal category $\CC$ to be $\text{Vect}_k$. In this case, the monoidal center $\CZ(\CC)=\text{Vect}_k$ satisfies Assumption~\ref{ass:coeq} and all the internal homs appearing below exist and can be defined set-theoretically. A more extensive discussion of $\BZ$ in the vector space case can be found in \cite[Sect.\,4.3]{dkr2}.

We will consider $\alg(\text{Vect}_k)$ as a 2-category whose only 2-morphisms are identities. We proceed to give the data of a lax 2-functor 
$$\BZ : \alg(\text{Vect}_k) \to \CAlg(\text{Vect}_k);$$ the statement that this indeed provides a lax 2-functor 
is the content of Theorem~\ref{thm:center-lax-functor-1} below. 
\begin{enumerate}
\item
On objects $A \in \alg(\CC)$ we set
\be \label{eq:center-lax-functor-1-obj}
  \BZ(A) := Z(\mbox{$A$-mod}) ~ \in  \, \text{Vect}_k
\ee
where $\mbox{$A$-mod}$ is the category of right $A$-modules, understood as a left $\text{Vect}_k$-module, and $Z(\mbox{$A$-mod}) \in \text{Vect}_k$ is its full center as in Definition~\ref{def:full-center}. In this case, the full center $Z(\mbox{$A$-mod})$ is nothing but the ordinary center of an algebra $A$, i.e. 
\be  \label{eq:center-vector-sp}
Z(\mbox{$A$-mod}) = Z(A) = \{ z\in A \, | \, za=az, \forall a\in A \}.
\ee  
To see this, note that for any $z\in Z(A)$, assigning to a right $A$-module $M$ the $A$-module endomorphism $m \mapsto m.z$ gives a natural transformation of the identity functor. Conversely, given such a natural transformation, evaluating it on $A$ produces an element in the center of the algebra.
\item 
For a linear map $f : A \to B$ consider the functor $(-) \otimes_{A} {}_fB  : \mbox{$A$-mod} \to \mbox{$B$-mod}$. Here ${}_f B$ is $B$ equipped with the structure of an $A$-$B$-bimodule where $B$ acts by multiplication and $A$ acts by composing with $f$ and then multiplying. We set
\be \label{eq:z-lax1-morph}
  \BZ(f) := \BZ((-) \otimes_{A} {}_fB) ~ \in \cosp(\BZ(A),\BZ(B)) \ ,
\ee
with $\BZ((-) \otimes_{A} {}_fB)$ as defined in \eqref{eq:Z(F)-def} (and for $\cosp$ see Definition~\ref{def:CAlg(Z)}). In this case, $\BZ(f)$ is the cospan $(Z(A) \xrightarrow{f} Z(f) \hookleftarrow Z(B))$, where $Z(f)$ is the centralizer defined as follows:
\be  \label{eq:centralizer-vector-sp}
Z(f) = \{ z \in B | zf(a) = f(a) z, \forall a\in A \} \ ,
\ee 
which is nothing but $Z({}_f B):= \Hom_{A\otimes B^{\op}} ({}_f B, {}_f B)$. 
\item
The unit transformation is just the canonical isomorphism $\BZ(\id_{\text{$A$-mod}}) \simeq \BZ((-) \otimes_A A)$, induced by the natural isomorphism $\id_{\text{$A$-mod}} \xrightarrow{\simeq} (-) \otimes_A A$ between the two endofunctors $\mbox{$A$-mod} \to \mbox{$A$-mod}$. 
\item The multiplication transformation of $\BZ$: 
given two morphism $A \xrightarrow{f} B \xrightarrow{g} C$, let $F = (-) \otimes_{A} {}_fB$ and $G=(-) \otimes_B {}_gC$, there is a natural isomorphism from $G \circ F = (-) \otimes_{A} {}_fB \otimes_B {}_gC$ to $(-) \otimes_C {}_{g \circ f}C$. This induces an isomorphism $Z(G\circ F) \xrightarrow{\sim} Z(g\circ f)$. The multiplication transformation of $\BZ$ is defined by composing this isomorphism with $m_{F,G}$ in \eqref{eq:m_(F,G)-def}. Namely, it is a collection of linear maps (as $2$-morphisms in $\CAlg(\text{Vect}_k)$) defined by the composed morphisms:
$$
Z(f)\otimes_{Z(B)} Z(g) = Z(F) \otimes_{Z(B)} Z(G) \xrightarrow{m_{F,G}} Z(G\circ F) \xrightarrow{\sim} Z(g\circ f)
$$ 
for $A,B \in \alg(\text{Vect}_k)$ and algebra homomorphisms $A \xrightarrow{f} B \xrightarrow{g} C$. In this case, above morphism is nothing but a linear map defined by
\be \label{eq:vect-sp-m}
z \otimes_{Z(B)} z' \mapsto g(z)\cdot z'
\ee
for $z\in Z(f)$ and $z'\in Z(g)$,	see \cite[Lem.\,4.11]{dkr2}.
\end{enumerate}
The associativity property follows from Lemma~\ref{lem:nFGFG-assoc}, and the unit property is straightforward. One can also use the set-theoretical definition given in \eqref{eq:center-vector-sp}, \eqref{eq:centralizer-vector-sp} and \eqref{eq:vect-sp-m} to check these properties directly. Thus altogether we have the following result:
\begin{thm} \label{thm:center-lax-functor-1} 
The assignment $\BZ : \alg(\text{Vect}_k) \to  \CAlg(\text{Vect}_k)$ defined above is a lax 2-functor.
\end{thm}

For general $\CC$ we define $\BZ$ in the same way, that is, via \eqref{eq:center-lax-functor-1-obj} on objects and via \eqref{eq:z-lax1-morph} on morphisms. We then have:

\begin{thm}  \label{thm:center-lax-functor-C}
For a monoidal category $\CC$ such that $\CZ(\CC)$ satisfies Assumption~\ref{ass:coeq}, if all above internal homs exist in $\CZ(\CC)$, then the assignment $\BZ: \alg(\CC) \to \CAlg(\CZ(\CC))$ is a lax 2-functor. 
\end{thm}

\subsection{The full center as a lax 2-functor, II} \label{sec:lax-functor-II}

We present the main result of this paper. It simply says that the full center construction gives a lax 2-functor 
 under a few natural conditions. 

\begin{thm} \label{thm:center-lax-functor-2} 
Let $\CC$ be a monoidal category such that $\CZ(\CC)$ satisfies Assumption~\ref{ass:coeq}. Let $\BM(\CC)$ be a locally full sub-bicategory (see App.\,\ref{app:bicategories}) of $\Mod(\CC)$ such that
\begin{enumerate}
\item for all $\CM,\CN \in \BM(\CC)$, the functor sub-category $\CHom_{\BM(\CC)}(\CM,\CN)$ is $\CZ(\CC)$-closed.
\item for all $F,G,H \in \CHom_{\BM(\CC)}(\CM,\CN)$, the morphism $\overline{\comp}$ in \eqref{eq:mult-map-Phi} is an isomorphism.
\item for all $F,F' \in \CHom_{\BM(\CC)}(\CM,\CN)$ and all $G,G' \in \CHom_{\BM(\CC)}(\CM,\CN)$, and for all $F\xrightarrow{\phi} F'$ and $G\xrightarrow{\psi} G'$, the morphism $m_{F,F',G,G'}$ in \eqref{eq:m-phipsi-def} is an isomorphism.
\end{enumerate}
Then the full center construction: 
\begin{itemize}
\item[-] on objects: $\BZ(\CM) = Z(\CM)$ as defined in Definition~\ref{def:full-center}.
\item[-] on morphism categories: $\underline{\BZ_{\CM,\CN}} : \CHom_{\BM(\CC)}(\CM,\CN) \to \Cospu(Z(\CM),Z(\CN))$ between morphism categories where $\underline{\BZ_{\CM,\CN}}$ is given in Cor.\,\ref{cor:Phi-iso-Z-functor}, 
\end{itemize}
	i.e. 
		\be
\BZ: \,\, \xymatrix{ \CM \ar@/^12pt/[rr]^F \ar@/_12pt/[rr]_G & \Downarrow \phi & \CN}
\quad \longmapsto \quad
\raisebox{3em}{\xymatrix@R=1.3em{  
  & Z(F) \ar[d]^{[F, \phi]} & \\ 
  Z(\CM) \ar@/^8pt/[ur]^(0.4){[F \circ -]_{[\id,\id]}} \ar@/_8pt/[rd]_(0.4){[G \circ -]_{[\id,\id]}} & [F,G] & Z(\CN) \ar@/_8pt/[ul]_(0.4){[- \circ F]_{[\id,\id]}} \ar@/^8pt/[dl]^(0.4){[- \circ G]_{[\id,\id]}}  \\
  & Z(G) \ar[u]^{[\phi,G]} 
  }} \ ,
		\ee
together with the multiplication transformation given by $\underline{\Bm}$ defined in Cor.\,\ref{cor:m-iso-m-natxfer} and the unit transformation given by the identity,
defines a lax 2-functor:
\be \label{eq:C-mod-to-CALG}
  \BZ : \BM(\CC) \to \CALGu(\CZ(\CC)) 
\ee
between two bicategories. If in addition
$\underline{\Bm}$ is invertible, $\BZ$ is a (non-lax) $2$-functor. 
\end{thm}

\begin{proof}
We need to show that the associativity \eqref{diag:asso-lax} and unit properties \eqref{diag:unit-lax-l} and \eqref{diag:unit-lax-r} hold. Recall that the multiplication transformation $\underline{\Bm}$ is determined by a collection of morphisms $\{ \Bm_{(F,G)} \}$ (or equivalently $\{ m_{F,G} \}$) defined in (\ref{eq:m_(F,G)-def}), and the unit isomorphisms $l$ and $r$ in $\CALGu(\CZ(\CC))$ are defined in (\ref{eq:unit-iso-l-r-CALG}). The unit properties (\ref{diag:unit-lax-l}) and (\ref{diag:unit-lax-r}) amount to prove that 
$$
\Bm_{(\id_\CM, F)} = l_{Z(F)}\quad \mbox{and} \quad
\Bm_{(F,\id_\CM)} = r_{Z(F)}
$$ 
which are obvious.

The property (\ref{diag:asso-lax}) amounts to the commutativity of the following diagram:
\be
\xymatrix{
(Z(F)\otimes_{Z(\CM)} Z(G)) \otimes_{Z(\CN)} Z(H) \ar[r]^{\hspace{1cm}m_{F,G}1} \ar[d]
& Z(GF) \otimes_{Z(\CN)} Z(H) 
\ar[r]^{\hspace{1cm}m_{GF, H}} & Z(HGF) \ar@{=}[d] \\
Z(F)\otimes_{Z(\CM)} (Z(G) \otimes_{Z(\CN)} Z(H)) \ar[r]^{\hspace{1cm}1m_{G,H}} & Z(F) \otimes_{Z(\CM)} Z(HG) \ar[r]^{\hspace{1cm}m_{F, GH}} & Z(HGF) 
}
\ee
which follows from Lemma~\ref{lem:nFGFG-assoc}.
\end{proof}

\begin{rema}
1) We use the same notation $\BZ$ for the two distinct lax 2-functors 
defined in Theorems \ref{thm:center-lax-functor-C} and \ref{thm:center-lax-functor-2}. This is justified by the following commutative diagram:
\be
\raisebox{2.3em}{\xymatrix{
\alg(\CC) \ar[rr]^{\BZ} \ar[d]^{\mathrm{Mod}} && \CAlg(\CZ(\CC)) \ar[d]_E \\
\BM(\CC) \ar[rr]^{\BZ} && \CALGu(\CZ(\CC)), 
}}
\ee
assuming both $\BZ$ are well-defined. Here the functor $\mathrm{Mod}$ acts on objects and morphisms as $A \mapsto \mbox{$A$-mod}$ and $A \xrightarrow{f} B$ to $(-) \otimes_A {}_fB$; we assume that the locally full subcategory $\BM(\CC)$ of $\Mod(\CC)$ is large enough to contain the image of the functor $\mathrm{Mod}$. The 2-functor 
 $E$ is given in \eqref{eq:functor-E-CAlg-CALG}. Note that the horizontal arrows are lax, while the vertical arrows are non-lax. In any case, both, the upper and lower path in the diagram give
\be
  \big( \, A \xrightarrow{f} B \, \big)
  ~\longmapsto~
  \big( \, Z(\mbox{$A$-mod}) \xrightarrow{\BZ((-) \otimes_A {}_fB)} Z(\mbox{$B$-mod}) \, \big) \ ,
\ee
where $\BZ$ is the cospan in \eqref{eq:Z(F)-def}.
\end{rema}

In Sections \ref{sec:auto} and \ref{sec:fusion-cat}, we will give two cases, in which the conditions in Theorem\,\ref{thm:center-lax-functor-2} are satisfied, so that we obtain $\BZ$ as a (non-lax) 2-functor. 
In general, however, these conditions, in particular the condition 2\ and 3\ in Theorem~\ref{thm:center-lax-functor-2}, are not satisfied. Therefore, we must work with $3$-cells and tricategory in general. We propose the following conjecture. 
\begin{conj}  \label{conj:tricat-laxfun-Z}
Without assuming the condition 2.\ and 3.\ in Theorem~\ref{thm:center-lax-functor-2} for $\BM(\CC)$, 
the assignment $\BZ$: $\CM\to Z(\CM)$, (\ref{eq:Z(F)-def}) and (\ref{eq:Z_FG(phi)-def}), 
together with the lax $2$-functor $\Bm$ given in Proposition\,\ref{prop:nat-tran-m}, can be extended (by adding other coherence morphisms properly) to a lax 3-functor between $\BM(\CC)$ and the conjectural tricategory $\CALG(\CZ(\CC))$ (recall Conjecture\,\ref{conj:tricat-CALG}).
\end{conj}

\section{Automorphisms of full centers} \label{sec:auto}

Given a category $\CD$, we denote the set of isomorphisms in $\Hom_\CD(U,V)$ by $\CHom_\CD^\times(U,V)$, and the maximal groupoid inside $\CD$ by $\CD^\times$. Analogously, if $\BD$ is a bicategory, we denote by $\CHom_\BD^\times(U,V)$ the 1-groupoid consisting of invertible 1-morphisms and 2-isomorphisms, and by $\BD^\times$ the maximum 2-groupoid in $\BD$. 
In Section \ref{sec:inv-cosp-2cell}, we will give a simple characterization of $\CALGu(\CZ)^\times$. 
Then in Section \ref{sec:fun-groupoid}, we will show that the second and third conditions in Theorem \ref{thm:center-lax-functor-2} are automatically satisfied for $\BM(\CC)^\times$. Therefore, assuming only the first condition in Theorem~\ref{thm:center-lax-functor-2} for $\BM(\CC)^\times$, we obtain a (non-lax) 2-functor $\BZ: \BM(\CC)^\times \to \CALGu(\CZ(\CC))^\times$ between 2-groupoids. 

\subsection{Invertible cospans and 2-cells} \label{sec:inv-cosp-2cell}
 
Let $\CZ$ be a braided monoidal category satisfying Assumption~\ref{ass:coeq}.

\begin{prop}   \label{prop:inv-1-2cell}  
In the bicategory $\CALGu(\CZ)$,
\\[.5em]
(i) a 1-morphism (i.e.\ a cospan)
\be \label{eq:inv-cospan}
\raisebox{1.5em}{\xymatrix@R=1em@C=1em{ & T \\ A\ar[ru]^(.35)f && B \ar[lu]_(.35)g}}
\ee 
is invertible iff $f$, $g$ are isomorphisms in $\CZ$.
\\[.5em]
(ii) a 2-cell (i.e.\ an isomorphism class of 2-diagrams)  
\be \label{eq:inv-2cell}
\raisebox{3em}{\xymatrix@R=1em{
& S \ar[d]^{v} & \\ A \ar[ru]^{} \ar[rd]_{} & M & B \ar[ul]_{} \ar[dl]^{} \\
& T \ar[u]^{w} }}
\ee 
is invertible iff $v$, $w$ are isomorphisms in $\CZ$.
\end{prop}

\begin{proof}
We will first prove the second statement, then the first.
\\[.5em]
(ii) If $v$ and $w$ are invertible, we will only show that the 2-cell given in (\ref{eq:inv-2cell}) has a right inverse. The proof of the existence of a left inverse is similar. Since $w$ is invertible, then $M$ is also naturally a right $T$-module via
$$ 
M\otimes T \xrightarrow{w^{-1}1} T\otimes T \xrightarrow{\mu} T \xrightarrow{w} M. 
$$
Moreover, this right $T$-action commutes with the left $T$-action on $M$. In this way, $M$ becomes a $T$-$T$-bimodule which is isomorphic to $T$. Similarly, $M$ is also a $S$-$S$-bimodule and $v: S \rightarrow M$ is an isomorphism of $S$-$S$-bimodules. Consider the following 2-cells: 
\be \label{eq:right-inv-2cell}
\raisebox{3em}{\xymatrix@R=1em{
& T \ar[d]^{w} & \\ A \ar[ru]^{} \ar[rd]_{} & M & B \ar[ul]_{} \ar[dl]^{} \\
& S \ar[u]^{v} }}
\ee
The vertical composition $(\ref{eq:right-inv-2cell}) \circ (\ref{eq:inv-2cell})$ gives the target of the following 3-cell: 
$$
\raisebox{3em}{\xymatrix@R=1em{
    & S \ar[ld]_{\rho\circ (w1)\circ (\iota_T v)}
 \ar@{=}[rd] & \\  M\otimes_T M & M\otimes_T T \ar[l]_{\hspace{0.3cm}1\otimes_T w} & 
S \ar[l]_{\hspace{0.3cm} v}  \\
    & S \ar[lu]^{\rho\circ (1w)\circ (v \iota_T)} \ar@{=}[ru]  }}
$$
where we identified $M \otimes_T T = M = T \otimes_T M$.
It is an easy exercise to prove that 
   $\rho \circ (w\otimes 1)\circ (\iota_T \otimes v) = (1\otimes_T w) \circ v = \rho \circ (1\otimes w) \circ (v \otimes \iota_T)$. 
Since $v$ and $w$ are invertible, above $3$-cell is invertible. Thus (\ref{eq:right-inv-2cell}) is the right inverse of (\ref{eq:inv-2cell}). 

Conversely, given an invertible 2-cell (\ref{eq:inv-2cell}) we want to show that both $v$ and $w$ are invertible. By the invertibility assumption, there exists a 2-cell, say given by
$$
\raisebox{3em}{\xymatrix@R=1em{
& T \ar[d]^{x} & \\ A \ar[ru] \ar[rd] & N & B \ar[ul] \ar[dl] \\
& S \ar[u]^{y} }}
$$ 
and isomorphisms $a\se N\otimes_TM\to S$, $b\se M\otimes_SN\to T$, of $S$- and $T$-bimodules respectively, such that $a$ is the inverse to both of the two compositions:
$$
S \xrightarrow{v} M=T\otimes_T M \xrightarrow{x\otimes_T 1} N\otimes_T M, \quad\quad
S \xrightarrow{y} N=N\otimes_T T \xrightarrow{1\otimes_T w} N\otimes_T M
$$
and $b$ is the inverse to both of the two compositions:
$$
T \xrightarrow{w} M=M\otimes_S S \xrightarrow{1\otimes_S y} M\otimes_SN, \quad\quad
T \xrightarrow{x} N =S\otimes_S N \xrightarrow{v\otimes_S 1} M\otimes_S N. 
$$
Therefore, we have $a \circ (x\otimes_T1) \circ v=\id_S$. To prove that $v$ is invertible, it is enough to show $v\circ (a \circ (x\otimes_T1))$ is invertible. 
By the commutativity of the following diagram:
$$
\xymatrix{
T\otimes_T M \ar[r]^{x1} & N\otimes_TM \ar@{=}[d] \ar[r]^a & S \ar[r]^v & M \\
& S\otimes_S N\otimes_T M \ar[r]^{v11} & M\otimes_S N\otimes_T M \ar[r]^{\hspace{0.6cm}1a}  & M\otimes_S S, \ar@{=}[u]
}
$$
we have $v \circ (a \circ (x1)) =(1a) \circ (v11 \circ x1) = 1a \circ (b^{-1}1)$ 
which is invertible.

The proof of the invertibility of $w$ is similar.  
\\[.5em]
(i) If $f,g$ are isomorphisms, then $T$ is also a commutative algebra and $(B \xrightarrow{g} T \xleftarrow{f} A)$
is a 1-morphism in $\CALGu(\CZ)$. Its composition with original cospan (\ref{eq:inv-cospan}) gives the cospan 
$$
\raisebox{1.5em}{\xymatrix@R=1em@C=1em{ & T\otimes_B T \\ A\ar[ru]^(.35)i && B \ar[lu]_(.35)j}}
$$ 
It is easy to show that $i=j$ and $i$ is an isomorphism. We want to show that there is an invertible 2-cell between $(A \xrightarrow{i} T\otimes_B T \xleftarrow{i}A)$ and the identity cospan $(A\xrightarrow{=}A\xleftarrow{=}A)$. It is given as follows:
$$
\raisebox{3em}{\xymatrix@R=1em{
& T\otimes_B T \ar[d]^{i^{-1}} & \\ 
A \ar[ur]^i \ar@{=}[dr]  & A \ar@{=}[d]  & A \ar[ul]_i \ar@{=}[dl]  \\
& A & }}
$$
By the result of (ii) which has been proved, above 2-cell is invertible. Hence, we obtain that (\ref{eq:inv-cospan}) is invertible.

Conversely, if the cospan (\ref{eq:inv-cospan}) is invertible,
then there is a cospan $(B \xrightarrow{g'} G \xleftarrow{f'} A)$ 
such that there exist bimodules $M,N$ and
four isomorphisms: $T\otimes_B G \xrightarrow{a_1} M \xleftarrow{a_2} A$ and $G\otimes_A T \xrightarrow{b_1} N\xleftarrow{b_2} B$ forming the invertible 2-cells:
$$
\raisebox{3em}{\xymatrix@R=1em{
& T\otimes_B G \ar[d]^{a_1} & \\ 
A \ar[ur]^{f_1} \ar@{=}[dr]  & M  & A   \ar[ul]_{f_2} \ar@{=}[dl] \\
& A \ar[u]^{a_2}  & }}
\quad\quad
\raisebox{3em}{\xymatrix@R=1em{
& G\otimes_A T \ar[d]^{b_1} & \\ 
B \ar[ur]^{g_1} \ar@{=}[dr]  & N  & A   \ar[ul]_{g_2} \ar@{=}[dl] \\
& B \ar[u]^{b_2}  & }} \quad ,
$$
where the morphisms $f_1, f_2, g_1, g_2$ are composed morphisms given as follows:
$$
f_1:  A \xrightarrow{f}  T = T\otimes_B B \xrightarrow{1g'}  T\otimes_B G, \quad\quad
f_2:  A \xrightarrow{f'} G = B\otimes_B G \xrightarrow{g1} T\otimes_B G,
$$
$$
g_1: B \xrightarrow{g'}G = G\otimes_A A \xrightarrow{1f}  G\otimes_A T, \quad\quad
g_2: B \xrightarrow{g} T=  A\otimes_A T \xrightarrow{f'1}  G\otimes_A T.
$$
We want to show that both $f$ and $g$ are invertible.  

First, since $a_1$ and $a_2$ are invertible
we must have $f_1=f_2$ and $g_1=g_2$. Moreover, $f_1$ and $g_1$ are both isomorphisms with the inverses given by $a_2^{-1}\circ a_1$ and $b_2^{-1}\circ b_1$, respectively. Secondly, we have $(f_1^{-1}\circ (1g')) \circ f=\id_A$ and $(g_1^{-1} \circ (f'1)) \circ g=\id_B$. It remains to show that $f \circ (f_1^{-1} \circ (1g'))$ 
and $g \circ (g_1^{-1} \circ (f'1))$ are invertible. It is easy to show that the following diagram: 
$$
\xymatrix{
T \cong T\otimes_B B \ar[r]^{1g'} & T\otimes_B G \ar[r]^{f_1^{-1}} \ar[d]^= & A \ar[r]^f & T \\
& T\otimes_B G \otimes_A A \ar[r]^{11f} & T\otimes_B G\otimes_A T \ar[r]^{\hspace{0.5cm}f_1^{-1}1} 
& A\otimes_A T \ar[u]_= \\
}
$$
is commutative. Therefore, we obtain $f \circ (f_1^{-1} \circ (1g')) =  (f_1^{-1}1) \circ ((11f) \circ (1g')) =  (f_1^{-1}1) \circ (1 g_1)$
 which is clearly invertible. Therefore, $f$ is invertible. 

The proof of the invertibility of $g$ is similar. 
\end{proof}

\medskip
In the remaining part of this subsection, we give a simple characterization of $\CALGu(\CZ)^\times$.
Let $A$ and $B$ be two commutative $\CZ$-algebras.  We define the map
\be\label{eq:A-1grpd-def}
  \mathcal{A} ~:~ (A \xrightarrow{f} B) \quad \longmapsto \quad
\raisebox{1.5em}{
\xymatrix@R=1em@C=1em{ & B \\ A\ar[ru]^(.35)f && B \ar@{=}[lu]}
}
\ee
If we view $\Hom_{\alg(\CZ)}(A,B)$ as a category with only identity morphisms, $\mathcal{A}$ automatically defines a functor $\Hom_{\alg(\CZ)}(A,B) \to \CHom_{\CALGu(\CZ)}(A,B)$. Moreover, we have the following result. 

\begin{lemma} 
The functor $\mathcal{A}: \Hom_{\alg(\CZ)}^\times(A,B) \to \CHom_{\CALGu(\CZ)}^\times(A,B)$ is an equivalence. 
\end{lemma}
\pf
The functor $\mathcal{A}$ is essentially surjective on objects since, for any invertible cospan 
$(A \xrightarrow{f} T \xleftarrow{g} B)$, there is an invertible 2-cell
\be  \label{eq:A-surject}
\raisebox{3em}{\xymatrix@R=1em{
& T  \ar@{=}[d] & \\ 
A \ar[ur]^{f} \ar[dr]_{g^{-1}f}  & T  & B   \ar[ul]_{g} \ar@{=}[dl] \\
& B \ar[u]^{g}  & }} \qquad .
\ee 

It remains to show that the functor $\mathcal{A}$ is fully faithful. 
The functor $\mathcal{A}$ is trivially faithful, since $\Hom_{\alg(\CZ)}(A,B)$ has only identity morphisms. To see that it is full, let
\begin{equation}\label{2cf}
\raisebox{3em}{\xymatrix@R=1em{
& B  \ar[d]^{\phi} & \\ 
A \ar[ur]^{f} \ar[dr]_{g}  & M  & B   \ar@{=}[ul] \ar@{=}[dl] \\
& B \ar[u]^{\psi}  & }}
\end{equation}
be an arbitrary invertible 2-cell between two cospans in the image of $\mathcal{A}$. It is clear that we must have $\phi=\psi$ which in turn implies $f=g$. Moreover, since $M$ is naturally a $B$-$B$-bimodule and $\phi=\psi$ is bimodule isomorphism, we obtain that above 2-cell is isomorphic (thus identical) to the identity 2-cell. Hence it lies in the image of $\mathcal{A}$.
\epf

The above lemma immediately implies the following

\begin{thm}
Let $\CZ$ be a braided monoidal category satisfying Assumption~\ref{ass:coeq}. Think of the 1-groupoid $\calg(\CZ)^\times$ as a 2-groupoid with only identity 2-morphisms. Then the functor $\mathcal{A}: \calg(\CZ)^\times \rightarrow \CALGu(\CZ)^\times$ in \eqref{eq:A-1grpd-def} is an equivalence of 2-groupoids.
\end{thm}

\subsection{The full center is non-lax on equivalences}  \label{sec:fun-groupoid}

Let $\BM(\CC)$ in Theorem~\ref{thm:center-lax-functor-2} satisfy only the condition 1. For $F,G,H \in \CHom_{\BM(\CC)}^\times(\CM, \CN)$, if $F\cong G$ (or $G\cong H$), it is easy to show that $\overline{\comp}$ defined in (\ref{eq:underline-comp-def}) is an isomorphism. In other words, the second condition in Theorem~\ref{thm:center-lax-functor-2} is automatically satisfied for $\BM(\CC)^\times$. We want to show the the third condition is also satisfied for $\BM(\CC)^\times$. We need a lemma. 

\begin{lemma} \label{lem:F-equiv-gives-ZC-iso} 
For $\CC$-module functors $G,G': \CHom_{\BM(\CC)}(\CN,\mathcal{P})$, $H,H' \in \CHom_{\BM(\CC)}(\CL,\CM)$ and $F \in \CHom_{\BM(\CC)}^\times(\CM,\CN)$, i.e. $F$ is a $\CC$-module equivalence between $\CM$ and $\CN$, the morphisms
\be
  [-\circ F]_{[G,G']} : [G,G']\to [GF,G'F]
  \quad \text{and} \quad
  [F \circ -]_{[H,H']} : [H, H']\to [FH,FH']~, 
\ee
defined in \eqref{eq:C-mod-F-gives-ZC-morph} and (\ref{eq:def-F-MN}), are isomorphisms.
\end{lemma}

\begin{proof}
Let $\bar{F}$ be a quasi-inverse of $F$. By definition, there are two natural isomorphisms: $F\bar{F} \xrightarrow{\phi}  \id_\CN$ and $\bar{F}F \xrightarrow{\psi} \id_\CM$. 
It is enough to show that the following compositions of morphisms: 
\bea  \label{eq:equ-iso-alg}
&& [G,G'] \xrightarrow{[- \circ F]_{[G,G']}} [GF,G'F] \xrightarrow{[- \circ \bar{F}]_{[GF,G'F]}} [GF\bar{F},G'F\bar{F}] \xrightarrow{[1\phi^{-1}, G']\circ [GF\bar{F},1\phi]} [G,G']  \\
&& [GF,G'F] \xrightarrow{[- \circ \bar{F}]} [GF\bar{F},G'F\bar{F}] \xrightarrow{[- \circ F]} 
[GF\bar{F}F,G'F\bar{F}F]  
\xrightarrow{[11\psi^{-1},G'F] \circ [GF\bar{F}F, 11\psi]} [GF,G'F] 
\nonumber
\eea
equal to identity maps. We will only prove that (\ref{eq:equ-iso-alg}) gives the identity map. The proof for the other one is similar. Recall Remark \ref{rema:iso-int-hom}\,(ii) which says that the morphism $[1\phi^{-1}, G]\circ [GF\bar{F},1\phi]$ is an isomorphism. Hence it is enough to check that 
\be \label{eq:equ-iso-alg-pf-1}
[-\circ \bar{F}]_{[GF,G'F]} \circ [-\circ F]_{[G,G']}= [GF\bar{F}, 1\phi^{-1}] \circ [1\phi, G'].
\ee
This follows from the commutativity of the following diagram: 
$$
\xymatrix{
[G, G'] \ast GF\bar{F} \ar[rr]^{\hspace{-.3cm}[1\phi,G']1} \ar[d]_{11\phi} & & [GF\bar{F}, G'] \ast GF\bar{F} \ar[rr]^{[GF\bar{F},1\phi^{-1}]} \ar[d]^{\ev_{GF\bar{F}}}  & & [GF\bar{F}, G'F\bar{F}] \ast GF\bar{F} \ar[d]^{\ev_{GF\bar{F}}}  \\
[G,G']\ast G \ar[rr]^{\ev_G} &  &  G' \ar[rr]^{1\phi^{-1}} &  & G'F\bar{F} 
}
$$
which is just a special case of the two commutative diagrams in (\ref{eq:def-fast-gast}). 
Notice also that the map $(1\phi^{-1}) \circ \ev_G \circ (11\phi)$ equals to $[G,G']\ast GF\bar{F} \cong ([G,G']\ast G)F\bar{F} \xrightarrow{\ev_G11} G'F\bar{F} $. By (\ref{eq:def-F-MN}) and Lemma \ref{lemma:G-F=GF}, we obtain the identity (\ref{eq:equ-iso-alg-pf-1}) immediately. 
\end{proof}

The following lemma says that the third condition in Theorem~\ref{thm:center-lax-functor-2} is satisfied automatically for $\BM(\CC)^\times$. 
\begin{lemma} \label{lem:FFGG-equiv-nFFGG-iso} 
For $F,F' \in \CHom_{\BM(\CC)}^\times(\CL, \CM)$ and $G,G'\in \CHom_{\BM(\CC)}^\times(\CM,\CN)$, if there are isomorphisms $F\xrightarrow{\phi} F'$ and $G\xrightarrow{\psi} G'$, then the morphism $n_{F,F',G,G'} : [F,F'] \otimes_{Z(\CM)} [G,G'] \to [GF,G'F'] $ defined in \eqref{eq:nFFGG-def} and the morphism $\Bm_{(\phi, \psi)}:=m_{F,F',G,G'}$ defined in (\ref{eq:mFFGG-def}) and (\ref{eq:m-phipsi-def}) are both isomorphisms.
\end{lemma}
\begin{proof}
Let $F\xrightarrow{\phi} F'$ and $G\xrightarrow{\psi} G'$ be two natural isomorphisms.
Using Lemma \ref{lem:F-equiv-gives-ZC-iso} and the commutative diagram (\ref{eq:nFFGG-def}), it is easy to obtain the following commutative diagram
\be  \label{diag:h-iso}
\xymatrix{
[F,F']\otimes [G,G'] \ar[r]^\rho \ar[d]_{[G'\circ -]\otimes [-\circ F]}^\cong & [F,F']\otimes_{Z(\CM)} [G,G']  \ar[d]^f \ar[rd]^{n_{F,F';G,G'}} &  \\
[G'F, G'F'] \otimes [GF, G'F] \ar[r]^{\hspace{-0.5cm}\rho} & [G'F, G'F'] \otimes_{Z(G'F)} [GF, G'F] \ar[r]^{\hspace{1.5cm}\overline{\comp}}
 \ar[d]_{[G'F,G'\phi^{-1}]1} & [GF, G'F'] \ar[d]^{[GF,G'\phi^{-1}]} \\
& [G'F, G'F] \otimes_{Z(G'F)} [GF, G'F] \ar[r]^{\hspace{1.5cm}\cong} & [GF, G'F] 
}
\ee
where $f$ is given by the universal property of coequalizer. It is easy to show that $f$ is an isomorphism. Then $n_{F,F';G,G'}=g\circ f$,
where $g$ is the composition of the three isomorphisms at the bottom of diagram \eqref{diag:h-iso},
 is an isomorphism. 

Recall the identity (\ref{eq:mu-n-phi-psi}). Notice that the morphism $u$ in (\ref{eq:mu-n-phi-psi}) (defined by (\ref{eq:2-diag-com-u-def})) is an isomorphism in our case. Therefore, $m_{F,F',G,G'}$ is also an isomorphism.
\end{proof}

Altogether, we obtain the following result. 
\begin{thm}  \label{thm:funct-eq} 
$\BZ : \BM(\CC)^\times \to \CALGu(\CZ(\CC))^\times$ is a (non-lax) 2-functor between two bi-groupoids.
\end{thm}

\begin{rema}  {\rm
We can reformulate the main result in \cite{morita} in terms of the 2-functor
$\BZ$ in Theorem~\ref{thm:funct-eq}. Assume that $\CC$ is a modular tensor category. We define a bicategory $\mathbf{ssFA}(\CC)$ of special symmetric Frobenius algebras in $\CC$: 
\bnu 
\item $1$-morphisms between objects $A$ and $B$ are $A$-$B$-bimodules in $\CC$.
\item $2$-morphisms between two $A$-$B$-modules $M$ and $N$ are bimodule maps. 
\enu
Then there is a fully faithful embedding $\mathbf{ssFA}(\CC) \hookrightarrow \Mod(\CC)$. 
Let $\mathbf{ssFA}(\CC)^\times$ be 2-groupoid consisting of all isomorphisms in $\mathbf{ssFA}(\CC)$. Composing this embedding with $\BZ$, we obtain a 2-functor
 $\mathbf{ssFA}(\CC)^\times \to \CALGu(\CZ(\CC))^\times$. Let $\underline{\mathbf{ssFA}(\CC)^\times}$ and $\underline{\CALGu(\CZ(\CC))^\times}$ be the truncated $1$-groupoids (see Definition~\ref{def:truncated-cat}). The main result of \cite{morita} is that the Picard group of a special symmetric Frobenius algebra $A$ in $\CC$ is isomorphic to the  group of automorphisms of its full center of $Z(A)$. In terms of $\BZ$, this result can be restated as follows: the functor $\underline{\BZ}: \underline{\mathbf{ssFA}(\CC)}^\times \to \underline{\CALGu(\CZ(\CC))^\times}$ is an equivalence between two groupoids. A categorification of this result was proven in \cite{ENO2009} (see also \cite{kitaev-kong} for its physical meaning).
} 
\end{rema}

\section{Fusion categories}  \label{sec:fusion-cat}

In this section, we assume that $\CC$ is a fusion category over a ground field $k$ \cite{eno05}. 
Namely, $\CC$ is a $k$-linear semisimple rigid monoidal category with finitely many isomorphism classes of simple objects and finite dimensional spaces of morphisms such that the unit object $\one$ is simple and $\text{End}(\one) \cong k$.

Let $\Mod^o(\CC)$ be the subcategory of $\Mod(\CC)$ consisting of indecomposable semisimple $\CC$-modules and $k$-linear $\CC$-module functors. One easily checks that indecomposable semisimple $\CC$-modules are automatically finite, i.e.\ have a finite number of simple objects and finite dimensional morphism spaces.
All $\CC$-modules in $\Mod^o(\CC)$ are $\CC$-closed, and all $\CC$-module functors between such $\CC$-modules are automatically exact. The main goal of this section is to show that all three conditions in Theorem~\ref{thm:center-lax-functor-2} are satisfied for $\BM(\CC)=\Mod^o(\CC)$. We will also show that the multiplication transformation $\underline{\Bm}$ is an isomorphism, so that $\BZ$ is a non-lax 2-functor.

\medskip
Throughout this section, we assume $\CM, \CN\in \Mod^o(\CC)$. Let $M$ be an non-zero object in $\CM$. The internal hom $[M,M]$ is a $\CC$-algebra. The following result is due to Ostrik \cite{ostrik}.
\begin{thm}
\label{thm:ostrik}
The functor $[M, -]: \CM \to \CC_{[M,M]}$ defined by $N \mapsto [M, N]$ is an equivalence.   
\end{thm}
We have the following lemma: 
\begin{lemma}  \label{lemma:ev-M-iso}  
Let $M, N\in \CM$ and $M\neq 0$. The following map
\be \label{eq:MNM=N}
\overline{\ev_M}: [M, N] \otimes_{[M, M]} M \rightarrow N
\ee
which is induced from the evaluation map $\ev_M: [M, N] \otimes M \rightarrow N$, is an isomorphism. 
\end{lemma}
\pf
Because $\CM$ is semisimple, it is enough to prove $\overline{\ev_M}$ in (\ref{eq:MNM=N}) is an isomorphism when $N$ is simple. Let thus $N$ be simple. Notice first that $[M, N]\neq 0$ by Theorem~\ref{thm:ostrik}. Then the map $[M,N]\otimes M \xrightarrow{\ev_M} N$ is non-zero because it is the image of $\id_{[M,N]}$ under the isomorphism (\ref{eq:def-int-hom-cl}) for $X=[M,N]$.
So $\overline{\ev_M}$ is also nonzero. Then it is enough to show that 
$[M, N] \otimes_{[M,M]} M \cong N$ as objects. 
Since the functor $[M,-] : \CM \to \CC_{[M,M]}$ is an equivalence thus preserves coequalizers, and by Lemma\,\ref{lemma:beta-iso} and Lemma\,\ref{lem:composition-properties}, we have $[M, [M, N]\otimes_{[M, M]} N]  \cong [M, N]\otimes_{[M, M]} [M,M]$ and the latter is naturally isomorphic to $[M,N]$.
\epf

\begin{prop} \label{prop:ev-M-iso}
Let $M, N\in \CM$ and $M\neq 0$. The following map
\be \label{eq:MNLM=LN}
\overline{\comp_M}: [M, N] \otimes_{[M,M]} [L, M] \rightarrow [L, N],
\ee
which is induced from $\comp_M: [M, N] \otimes [L, M] \to [L, N]$, is an $[N, N]$-$[L, L]$-bimodule isomorphism. 
\end{prop}
\pf
As before we identify $\CM=\CC_{[M,M]}$.
Recall the map (\ref{diag:def-beta}). By the universal property of $-\otimes_{[M,M]}-$, we obtain a map 
$\overline{\gamma_{[M, N]}}$ making the following diagram:  
\be
\xymatrix{
[M, N] \otimes [L, M] \ar[r]^{\gamma_{[M,N]}} \ar[d]_{\rho} & [L, [M,N]\otimes M] \ar[d]^{[L, \rho]} \\
[M, N] \otimes_{[M,M]} [L, M] \ar[r]^{\overline{\gamma_{[M,N]}}} & 
[L, [M,N] \otimes_{[M,M]} M]\, .
}
\ee
commutative. By Lemma \ref{lemma:beta-iso} and the rigidity of $\CC$, $\gamma_{[M,N]}$ is an isomorphism. Since $[L,-]$ is an equivalence (again by Theorem \ref{thm:ostrik}), also $[L,\rho]$ is a co-equalizer. This allows one to define $\overline{\gamma_{[M,N]}^{-1}}$ in $\CC_{[L,L]}$, so in particular in $\CC$, showing that $\overline{\gamma_{[M,N]}}$ is an isomorphism. 
The commutativity of the diagram (\ref{diag:beta-Nev=ev}) implies that the following diagram: 
\be
\xymatrix{
& [L, [M, N]\otimes_{[M, M]} M] \ar[rd]^{[L,\overline{\ev_M}]} & \\
[M,N]\otimes_{[M,M]} [L,M] \ar[ur]^{\overline{\gamma_{[M,N]}}}  \ar[rr]^{\overline{\comp_M}} & & [L, N]
}
\ee
is also commutative. By Theorem \ref{thm:ostrik} and Lemma \ref{lemma:ev-M-iso}, $[L,\overline{\ev_M}]$ is an isomorphism. 
Thus also $\overline{\comp_M}$ is an isomorphism, and by Lemma \ref{lem:composition-properties} it is a $[N, N]$-$[L, L]$-bimodule map. 
\epf

When $\CC$ is a fusion category, by \cite{eno05}, both categories $\CZ(\CC)$ and $\CC_\CM^\vee$ are fusion categories. (Recall from Section \ref{sec:full-center-module} that $\CC_\CM^\vee$ is the same category as $\CC_\CM^\ast$ but with the opposite tensor product.) As a consequence, the category $\CZ(\CC)$ satisfies Assumption~\ref{ass:coeq}. 

\smallskip
It is well-known that $\CZ(\CC) \cong \CZ(\CC_\CM^\vee)$ (see e.g.\ \cite{eo}). 
Let $\CM$ and $\CN$ be semisimple indecomposable $\CC$-modules. The category $\Fun_\CC(\CM, \CN)$ is naturally a $\CC_\CN^\ast$-$\CC_\CM^\ast$-bimodule. 
By \cite[Thm.\,3.31]{eo}, $\Fun_\CC(\CM, \CN)$ is a semisimple indecomposable module over $\CC_\CM^\ast$ (or $\CC_\CN^\ast$), and hence in particular finite. 
Via $\alpha$-induction, the category $\Fun_\CC(\CM, \CN)$ is also a finite semisimple $\CZ(\CC)$-module. As a consequence, $\Fun_\CC(\CM, \CN)$ is also $\CZ(\CC)$-closed. Therefore, the first condition in Theorem~\ref{thm:center-lax-functor-2} is satisfied for $\BM(\CC)=\Mod^o(\CC)$.  

\smallskip
To verify the second condition in Theorem~\ref{thm:center-lax-functor-2} we need a couple of technical lemmas.

A functor $F:\CZ\to\CC$ is {\em dominant} if every object of $\CC$ is a retract (i.e. a direct summand) of an object of the form $F(X)$ for some object $X$ of $\CC$.

\ble
Let $\CM$ be an indecomposable $\CC$-module. Then $F:\CZ\to\CC$ is a dominant monoidal functor.
Then the pullback $F^*(\CM)$ is an indecomposable $\CZ$-module.
\ele
\bpf
Indecomposability of $\CC$-module $\CM$ can be reformulated as follows: for a fixed $M\in\CM$ any $N\in\CM$ is a retract of $X*M$ for some $X\in\CC$.
To see the indecomposability of the $\CZ$-module $F^*(\CM)$ take the same $M\in\CM$. Then since by dominance of $F$ any $X\in\CC$ is a retract of $F(Z)$ for some $Z\in\CZ$ we have that any 
$N\in\CM$ is a retract of $F(Z)*M$ for some $Z\in\CZ$.
\epf

\ble
For a fusion category $\CC$, the forgetful functor $F: \CZ(\CC) \to \CC$ is dominant. 
\ele
\bpf
Let $I: \CC \to \CZ(\CC)$ be the right adjoint functor of $F$. By \cite[Lem.\,3.5]{dmno}, 
$I(\one)$ is a connected separable commutative algebra in $\CZ(\CC)$. By \cite[Lem.\,3.2]{ENO2008}, $I$ defines a monoidal equivalence from $\CC$ to the category $\CZ(\CC)_{I(\one)}$ of $I(\one)$-modules in $\CZ(\CC)$, and
the following diagram
$$
\xymatrix{
\CZ(\CC)  &   & \CZ(\CC)_{I(\one)}  \ar[ll]_{forget} \\
& \CC \ar[lu]^I \ar[ur]_{\simeq_\otimes}^{I} &
}
$$
is commutative. Notice that the left adjoint functor of the forgetful functor $\CZ(\CC)_{I(\one)} \to \CZ(\CC)$ is given by $-\otimes I(\one)$, which is dominant because $I(\one)$ is separable. Then it is clear that $F$ is dominant. 
\epf

Thus in particular $\Fun_\CC(\CM, \CN)$ is an indecomposable $\CZ(\CC)$-module. Therefore, we can apply Proposition\,\ref{prop:ev-M-iso} to the case of the semisimple indecomposable $\CZ(\CC)$-module $\Fun_\CC(\CM, \CN)$. We obtain, for $k$-linear $\CC$-module functors $F, G, H: \CM \to \CN$, the morphism $\overline{\comp}:[G, H] \otimes_{[G,G]} [F,G] \xrightarrow{} [F, H]$ is an isomorphism. 
In other words, the second condition in Theorem~\ref{thm:center-lax-functor-2} is satisfied for $\BM(\CC)=\Mod^o(\CC)$.

\medskip
Now we will show that the third condition in Theorem~\ref{thm:center-lax-functor-2} is satisfied for $\BM(\CC)=\Mod^o(\CC)$. Namely, the morphism $m_{F,F',G,G'}$ in \eqref{eq:m-phipsi-def} is an isomorphism. Before we prove this result, we need a few lemmas. 

Given $\CC$-module functors $\mathcal{L} \xrightarrow{F, F'} \CM \xrightarrow{G,G'} \CN$, 
the internal homs $[F, F']_{\CC_\CM^\ast}$ and $[G, G']_{\CC_\CM^\ast}$ in $\CC_\CM^\ast$ are defined by, \bea
\text{Hom}_{\Fun_\CC(\CM, \CN)}(X\circ F, F') &\cong& \text{Hom}_{\CC_\CM^\ast}(X, [F, F']_{\CC_\CM^\ast})\nn
\text{Hom}_{\Fun_\CC(\CM, \CN)}(G \circ X, G') &\cong& \text{Hom}_{\CC_\CM^\ast}(X, [G, G']_{\CC_\CM^\ast})  \nonumber
\eea
for $X\in \CC_\CM^\ast$. We have the canonical evaluation maps: 
$$
[F,F']_{\CC_\CM^\ast} \circ F \xrightarrow{\ev_F}  F',  \quad\quad G \circ [G,G']_{\CC_\CM^\ast}  \xrightarrow{\ev_G} G'. 
$$ 
Another way to understand $[F, F']_{\CC_\CM^\ast}$ and $[G, G']_{\CC_\CM^\ast}$ is given in the following Lemma.
\begin{lemma} \label{lemma:GG-GdG}  {\rm 
Let $F^\vee : \CM \to \CL$ be the left adjoint of $F$ and
$G^\vee: \CN\to \CM$ be the right adjoint of $G$. We have 
$$
[F, F']_{\CC_\CM^\ast} \simeq F' \circ F^\vee, \quad \quad [G,G']_{\CC_\CM^\ast}\simeq G^\vee \circ G'.
$$ 
}
\end{lemma}
\pf
We will only show the prove of  $[G,G']_{\CC_\CM^\ast}\simeq G^\vee \circ G'$. The proof of the other isomorphism is similar. 

It is enough to show that $G\circ G^\vee \circ G' \xrightarrow{\epsilon G'} G'$, where $G\circ G^\vee \xrightarrow{\epsilon} \id_\CN$  is the counit of the adjoint pair $(G,G^\vee)$, satisfies the following universal property: 
$$
\xymatrix{
& G \circ (G^\vee G') \ar[rd]^{\epsilon G'} & \\
G\circ X \ar[rr]^f \ar@{.>}[ur]^{\exists !\, 
G\bar{f}}  & & G'
}
$$
One can construct $\bar{f}$ as the composition $X\xrightarrow{\eta X} G^\vee \circ G \circ X \xrightarrow{G^\vee f} G^\vee \circ G'$ where $\eta: \id_\CM \to G^\vee G$ is the unit of the adjunction pair $(G, G^\vee)$. Let $g$ be another morphism $g: X\to G^\vee G'$ such that $f=\epsilon G' \circ Gg$. Then we have 
$$
\bar{f} = (G^\vee f) \circ (\eta X) = (G^\vee \epsilon G') \circ (G^\vee G g) \circ (\eta X) 
= (G^\vee \epsilon G') \circ (\eta G^\vee G') \circ g =g. 
$$
We have proved the uniqueness of $\bar{f}$.  
\epf

For $Z\in \CZ(\CC)$ and $F,F'\in \Fun_\CC(\CM, \CN)$, notice that
\bea
&&\Hom_{\Fun_\CC(\CM, \CN)}(Z\ast F, F') \simeq  \Hom_{\CC_\CM^\ast}( \alpha(Z), [F, F']_{\CC_\CM^\ast})
\simeq \Hom_{\CZ(\CC)}(Z, \alpha^\vee([F, F']_{\CC_\CM^\ast})),  \nn
&&\Hom_{\Fun_\CC(\CM, \CN)}(G \ast Z, G') \simeq \Hom_{\CC_\CM^\ast}( \alpha(Z), [G, G']_{\CC_\CM^\ast}) \simeq \Hom_{\CZ(\CC)}(Z, \alpha^\vee([G, G']_{\CC_\CM^\ast})). \nonumber
\eea
Therefore, we obtain
\be
[F, F'] \cong \alpha^\vee([F,F']_{\CC_\CM^\ast}) \quad \mbox{and} \quad [G,G'] \cong \alpha^\vee([G, G']_{\CC_\CM^\ast}).
\ee
\begin{lemma} \label{lemma:ev-map-factors}
{\rm 
The evaluation maps $[F, F'] \ast F \xrightarrow{\ev_F} G$ and $G\circ [G,G'] \xrightarrow{\ev_G} G'$ factor  as follows: 
\bea
&&[F, F'] \ast F \cong \alpha\alpha^\vee([F,F']_{\CC_\CM^\ast}) \circ F \xrightarrow{\epsilon\circ F} 
[F,F']_{\CC_\CM^\ast} \circ F \xrightarrow{\ev_F} F',  \nn
&& G \ast [G,G'] \cong G \circ \alpha\alpha^\vee([G,G']_{\CC_\CM^\ast})
\xrightarrow{G\circ \epsilon} G\circ [G,G']_{\CC_\CM^\ast} \xrightarrow{\ev_G} G',
\eea
where $\epsilon: \alpha\alpha^\vee \to \id_{\CC_\CM^\ast}$ is the counit of the adjoint pair $(\alpha, \alpha^\vee)$. 
}
\end{lemma}
\pf
It follows immediately from the universal property of $(\alpha\alpha^\vee, \alpha\alpha^\vee\xrightarrow{\epsilon} \id)$: 
$$
\xymatrix{
& \alpha\alpha^\vee(Y) \ar[rd]^{\epsilon_Y}  & \\
\alpha(X) \ar[rr]^f \ar@{.>}[ru]^{\exists ! \alpha(\bar{f}) }& & Y 
}
$$
and that of $([F,F'], \ev_F)$ and that of $([G,G'], \ev_G)$.  
\epf

\begin{lemma} \label{thm:ENO}  {\rm 
The right adjoint $\alpha^\vee$ gives a monoidal equivalence between 
$\CC_\CM^\vee$
and the category of $Z(\CM)$-modules in $\CZ(\CC)$.  In particular, we have a canonical isomorphism:
\be  \label{eq:alpha-FG}
[F,F'] \otimes_{Z(\CM)} [G,G'] \simeq \alpha^\vee( [G, G']_{\CC_\CM^\ast} \circ [F, F']_{\CC_\CM^\ast} ).
\ee
}
\end{lemma}

\pf
Consider the forgetful functor $\CZ(\CC_\CM^\ast) \to \CC_\CM^\ast$. According to \cite[Lem.\,3.2]{ENO2008}, its adjoint $I$ induces an equivalence of $\CC_\CM^\ast$ and $I(\id_{\CM})$-modules in $\CZ(\CC_\CM^\ast)$. Since $\CC$ is fusion, the functors $L_\CM$ and $R_\CM$ in (\ref{diag:Z(M)}) are equivalences \cite{eo}. Therefore, the functor $\tilde{\alpha}$ defined in Proposition\,\ref{prop:aind-factor} is also a monoidal equivalence. 
Thus, $\CC_\CM^\ast$ is equivalent to $\alpha^\vee(\id_{\CM})$-modules in $\CZ(\CC)$. But from \eqref{eq:adjoint-of-alpha} we see that $\alpha^\vee(\id_{\CM}) = \CZ(\CM)$.
\epf

\begin{lemma}  \label{lemma:u=iso}  {\rm
The morphism $n_{F,F',G,G'} : [F,F'] \otimes_{Z(\CM)} [G,G'] \to [GF, G'F']$ defined in the diagram (\ref{eq:nFFGG-def}) is an isomorphism. Moreover, $m_{F,F',G,G'}$ defined in (\ref{eq:m-phipsi-def}) is an isomorphism.  
}
\end{lemma}
\pf
We consider the following sequence of natural isomorphisms: 
\bea
\text{Hom}_{\CZ(\CC)}(X, [F,F']\otimes_{Z(\CM)} [G,G'])
&\overset{\eqref{eq:alpha-FG}}{\cong}& \text{Hom}_{Z(\CC)}(X, \alpha_{\mathcal{M}}^\vee([G, G']_{\CC_\CM^\ast} \circ [F, F']_{\CC_\CM^\ast} )) 
\nonumber \\[.5em]
&\cong& \text{Hom}_{\mathrm{Fun}_\CC(\CM, \CM)}( \alpha_{\mathcal{M}}(X), [G, G']_{\CC_\CM^\ast} \circ [F, F']_{\CC_\CM^\ast}) 
\nonumber \\[.1em]
&\overset{\text{Lem.\,\ref{lemma:GG-GdG}}}\cong& \text{Hom}_{\mathrm{Fun}_\CC(\CM, \CM)}( \alpha_{\mathcal{M}}(X), (G^\vee \circ G') \circ [F, F']_{\CC_\CM^\ast}) 
\nonumber \\[.1em]
&\overset{\text{Lem.\,\ref{lemma:beta-iso}}}\cong& \text{Hom}_{\mathrm{Fun}_\CC(\CM, \CM)}( \alpha_{\mathcal{M}}(X), [F, (G^\vee \circ G') \circ F']_{\CC_\CM^\ast}) 
\nonumber \\[.5em]
&\cong& \text{Hom}_{\mathrm{Fun}_\CC(\mathcal{L}, \CM)}( \alpha_{\mathcal{M}}(X)) \circ F, G^\vee \circ G' \circ F') 
\nonumber \\[.5em]
&\cong& \text{Hom}_{\mathrm{Fun}_\CC(\mathcal{L}, \CN
)}( G \circ \alpha_{\mathcal{M}}(X) \circ F, G'\circ F') \nn
&\overset{\text{Lem.\,\ref{lemma:2-actions}}}\cong& \text{Hom}_{\mathrm{Fun}_\CC(\mathcal{L}, \CN)}(X \ast (G\circ F),  G'\circ F') 
\nonumber \\[.5em]
&\cong& \text{Hom}_{\CZ(\CC)}(X, [GF,G'F']) \ .   \label{eq:long-iso}
\eea
Now let us take $X=[F,F']\otimes_{Z(\CM)} [G,G']$. We examine the image of $\id_X$ in (\ref{eq:long-iso}) step by step. In the 1st step, we obtain the following composition of maps: 
\bea
[F,F']\otimes_{Z(\CM)} [G,G'] &\cong& \alpha_\CM^\vee([F, F']_{\CC_\CM^\ast}) \otimes_{Z(\CM)} 
\alpha^\vee ([G, G']_{\CC_\CM^\ast}) \nn
&\xrightarrow{\eta}& \alpha_\CM^\vee \alpha_\CM (\alpha_\CM^\vee([F, F']_{\CC_\CM^\ast}) \otimes_{Z(\CM)} 
\alpha_\CM^\vee ([G, G']_{\CC_\CM^\ast})) \nn
&\cong& \alpha_\CM^\vee( \alpha_\CM\alpha_\CM^\vee ([G, G']_{\CC_\CM^\ast}) \circ 
\alpha_\CM\alpha_\CM^\vee([F, F']_{\CC_\CM^\ast})) \nn
&\xrightarrow{\alpha_\CM^\vee(\epsilon \epsilon)}& 
\alpha_{\CM}^\vee([G, G']_{\CC_\CM^\ast} \circ [F, F']_{\CC_\CM^\ast})  \nonumber
\eea
In the 2nd step, we obtain
\bea
\alpha_\CM([F,F']\otimes_{Z(\CM)} [G,G']) &\cong& 
\alpha_\CM (\alpha_\CM^\vee([F, F']_{\CC_\CM^\ast}) \otimes_{Z(\CM)} \alpha_\CM^\vee ([G, G']_{\CC_\CM^\ast}))  \nn
&\cong& \alpha_\CM\alpha_\CM^\vee ([G, G']_{\CC_\CM^\ast}) \circ \alpha_\CM \alpha_\CM^\vee([F, F']_{\CC_\CM^\ast}) \nn
&\xrightarrow{\epsilon\epsilon}& [G, G']_{\CC_\CM^\ast} \circ [F, F']_{\CC_\CM^\ast}  \nonumber
\eea
Using Lemma \ref{lemma:ev-map-factors}, we obtain in the 6th step the following morphism: 
$$
G\circ \alpha_\CM([F,F']\otimes_{Z(\CM)} [G,G']) \circ F = 
(G \ast [G,G']) \circ ([F,F'] \ast F) \xrightarrow{\ev_G \ev_F} G'F'.
$$
In the 7th step, we obtain the canonical action 
$$
([F,F'] \otimes [G,G']) \ast GF \to G'F'.
$$ 
Then we obtain $n_{F,F',G,G'}$ in the last step by the definition of $n_{F,F',G,G'}$. Therefore, $n_{F,F',G,G'}$ is an isomorphism. 

\medskip
It remains to show that $m_{F,F',G,G'}$ is an isomorphism. By (\ref{eq:m-phipsi-def}), it is enough to prove that $m'_{F,F',G,G'}$ is an isomorphism. By the definition of $m'_{F,F',G,G'}$ in (\ref{diag:def-m'}), it is enough to show that the map $Z(G'F') \otimes_{Y'} [GF, G'F'] \xrightarrow{f} [GF,G'F']$, which is naturally induced by the universal property of $\otimes_{Y'}$ from the map $Z(G'F') \otimes [GF, G'F'] \xrightarrow{\comp} [GF,G'F']$ in (\ref{diag:def-m'}), is an isomorphism. Since we have proved that $n_{(F,F',G,G')}$ is an isomorphism, in particular, we have that
$$
n_{F',F',G',G'}=m_{F',G'}: Y'=Z(F')\otimes_{Z(\CM)} Z(G') \to Z(G'F')
$$ 
is an algebra isomorphism. By re-examining the action of $Y'$ on $[GF, G'F']$ and $Z(G'F')$, it is easy to see that it factors through the action of $Z(G'F')$ via the algebra isomorphism $m_{F',G'}$. Therefore, we have the following factorization of $f$:
$$
\xymatrix{
Z(G'F') \otimes_{Y'} [GF, G'F'] \ar[rr]^{\cong}  \ar[rd]_{f} & & Z(G'F') \otimes_{Z(G'F')} [GF, G'F'] \ar[ld]^{\cong} \\
&  [GF, G'F']  &
}
$$
We obtain that $f$ is an isomorphism. 
\epf

Since the $m_{F,G} = n_{F,F,G,G}$ are isomorphisms, the multiplication transformation (see \eqref{eq:m_(F,G)-def} and \eqref{eq:composition-nat-xfer-m-diag-1cat-level}) is an isomorphism.
Altogether, we have shown that all three conditions in Theorem~\ref{thm:center-lax-functor-2} are satisfied for $\BM(\CC)=\Mod^o(\CC)$.
\begin{thm}  \label{thm:fusion-cat-Z-functor}
Let $\CC$ be a fusion category over a ground field $k$. Then 
$$\BZ: \Mod^o(\CC) \rightarrow \CALGu(\CZ(\CC))$$ 
defines a (non-lax) 2-functor between bicategories. 
\end{thm}

\begin{rema} {\rm
In a rational CFT, the category $\CC$ is not only a fusion category but also a modular tensor category. In these cases, $\CZ(\CC) = \CC_+\boxtimes \CC_-$. The full center $Z(\CM)$ of a $\CC$-module $\CM$ in $\Mod^o(\CC)$ is the bulk fields (or a closed CFT) associated to $\CM$. 
}
\end{rema}

\appendix

\section{Appendix}

\subsection{Module categories} \label{app:module-cat}

Here we briefly state our conventions for module categories, some references are \cite{qu,janelidze:2001,ostrik}.

\begin{defn} \label{def:mod-cat}
Let $\CC$ be a monoidal category with tensor unit $1_\CC$. A {\em left module category} over $\CC$, or $\CC$-{\em module} for short, is a category $\CM$ together with
\begin{enumerate}
\item {\em action functor}: A functor $\ast : \CC \times \CM \to \CM$.
\item {\em associator}: A natural isomorphism $a : \ast \circ (\id_\CC \times \ast) \to \ast \circ (\otimes \times \id_\CM)$ between functors $\CC \times \CC \times \CM \to \CM$. That is, 
a family of isomorphisms $a_{X,Y,M} : X \ast (Y \ast M) \to (X\otimes Y) \ast M$, natural in 
$X,Y \in \CC$ and $M \in \CM$.
\item {\em unit isomorphism}: a natural isomorphism $l : 1_\CC \ast (-) \to \id_\CM$ between endofunctors on $\CM$. That is, a family of isomorphisms $l_M : 1_\CC \ast M \rightarrow M$, natural in $M \in \CM$.
\end{enumerate}
These must satisfy the coherence conditions
\begin{enumerate}
\item {\em pentagon}: The diagram\\[-1em]
\be
\xygraph{ !{0;/r4.5pc/:;/u4.5pc/::}[]*+{X*(Y*(Z*M))} (
  :[u(.7)r(1.5)]*+{(X\otimes Y)*(Z*M)} ^{a_{X,Y,Z*M}}
  :[d(.7)r(1.5)]*+{((X\otimes Y)\otimes Z)*M}="r" ^{a_{XY,Z,M}}
  ,
  :[r(.5)d(.8)]*+!R(.3){X*((Y\otimes Z)*M)} _{\id_X*a_{Y,Z,M}}
  :[r(2)]*+!L(.3){(X\otimes(Y\otimes Z))*M} _{a_{X,YZ,M}}
  : "r" _{\alpha_{X,Y,Z}*\id_M}
)}
\ee
commutes for all $X,Y,Z \in \CC$, $M \in \CM$. Here $\alpha_{X,Y,Z}$ is the associator of $\CC$.
\item {\em triangle}: The diagram
\be
  \xymatrix@R=1em@C=.7em{X \ast (1_\CC \ast M) \ar[dr]_{\id_X \ast l_M} \ar[rr]^{a_{X,1_\CC,M}} && (X \otimes 1_\CC) \ast M \ar[dl]^{r_X \ast \id_M} \\
  & X \ast M }
\ee
commutes for all $X \in \CC$, $M \in \CM$. Here $r_{X}$ is the right unit isomorphism of $\CC$.
\end{enumerate}
\end{defn}

\begin{defn} \label{def:C-module-functor}
Let $\CC$ be a monoidal category and let $\CM$, $\CN$ be $\CC$-modules. A $\CC$-{\em module functor} from $\CM$ to $\CN$ is a functor $F : \CM \rightarrow \CN$ together with 
\begin{itemize}
\item a natural isomorphism $F^{(2)} : F \circ \ast \rightarrow \ast \circ (\id_\CC \times F)$ between functors $\CC \times \CM \to \CN$.
That is, a family of isomorphisms $F_{X,M}^{(2)} : F(X\ast M)\to X \ast F(M)$, natural in $X \in \CC$, $M \in \CM$.
\end{itemize}
These must satisfy the compatibility conditions
\begin{enumerate}
\item {\em associator compatibility}: The diagram\\[-1em]
\be
\xygraph{ !{0;/r4.5pc/:;/u4.5pc/::}[]*+{F(X\ast (Y\ast M))} (
  :[u(.7)r(1.5)]*+{F((X\otimes Y)\ast M)} ^{F(a_{X,Y,M})}
  :[d(.7)r(1.5)]*+{(X\otimes Y)\ast F(M)}="r" ^{F^{(2)}_{XY,M}}
  ,
  :[r(.5)d(.8)]*+!R(.3){X\ast F(Y\ast M)} _{F^{(2)}_{X,Y\ast M}}
  :[r(2)]*+!L(.3){X\ast (Y\ast F(M))} _{\id_X\ast F^{(2)}_{Y,M}}
  : "r" _{a_{X,Y,F(M)}}
)}
\ee
commutes for all $X,Y \in \CC$ and $M \in \CM$.
\item {\em unit compatibility}: The diagram
\be
  \xymatrix@R=1em@C=.7em{F(1_\CC \ast M) \ar[dr]_{F(l_M)} \ar[rr]^{F^{(2)}_{1_\CC,M}} && 1_\CC \ast F(M) \ar[dl]^{l_{F(M)}} \\
  & F(M) }
\ee
commutes for all $M \in \CM$. 
\end{enumerate}
\end{defn}

\begin{defn} \label{def:C-mod-nat-xfer}
Let $\CC$ be a monoidal category, $\CM$, $\CN$ be $\CC$-modules, and $F, G : \CM \to \CN$ be $\CC$-module functors. A $\CC$-{\em module natural transformation} from $F$ to $G$ is a natural transformation $\phi : F \rightarrow G$ such that the diagram
\be \label{eq:c-mod-nat-xfer-condition}
\xymatrix{
F(X \ast M) \ar[r]^{F^{(2)}_{X,M}} \ar[d]^{\phi_{X\ast M}} & X \ast F(M) \ar[d]^{\id_X \ast \phi_{M}} \\
G(X \ast M) \ar[r]^{G^{(2)}_{X,M}} & X \ast G(M) \\
}
\ee
commutes.
\end{defn}

\subsection{Bicategories} \label{app:bicategories}

In this appendix we recall the definition of bicategories and related notions, see \cite{be} or \cite{Leinster:1998}. Let $\one$ be a category with only one object and only the identity morphism.

\begin{defn} \label{def:bicat} 
A bicategory $\mathbf{S}$ consists of a set of objects (in a given universe) and a category of morphisms $\CHom(A,B)$ for each pair of objects $A$ and $B$ together with 
\bnu
\item {\it identity morphism}: $\one_A:  \one \to \Mor(A,A)$ for all $A\in \mathbf{S}$.
\item {\it composition functor}: 
\bea
\circledcirc_{A,B,C}: \CHom(B, C) \times \CHom(A, B) &\to& \CHom(A, C) \nn
(T, S) &\mapsto& T\circ S   \nonumber 
\eea
\item {\it associativity isomorphisms}: for $A, B, C, D\in \mathbf{S}$, there is a natural isomorphism $\alpha_{(A,B,C,D)}$, or $\alpha$ for simplicity,
$$
\alpha:  \circledcirc_{A,C,D} \circ (\circledcirc_{A,B,C} \times \id) \rightarrow 
\circledcirc_{A,B,D} \circ (\id \times \circledcirc_{B,C,D}),
$$
which consists of a family of morphisms $\{ \alpha_{(S,T,U)} \}_{S\in \CHom(A,B), T\in \CHom(B,C), U\in \CHom(C,D)}$. Occasionally, we will also abbreviate $\alpha_{(S,T,U)}$ as $\alpha$ in some diagrams (for example (\ref{diag:asso-lax})) for simplicity. 

\item {\it left and right unit isomorphisms}:
$l_{(A,B)}: \circledcirc_{A,A,B} \circ (\one_A \times \id ) \to \id$ and 
$r_{(A,B)}: \circledcirc_{A,B,B} \circ (\id \times \one_B)\to  \id$. These two natural transformations, also denoted by $l$ and $r$ for simplicity, consist of the following two families of maps 
$\{ \one_A \circ T \xrightarrow{l_T} T\}_{T\in \CHom(A, B)}$ and $\{ T \circ \one_B \xrightarrow{r_T} T\}_{T\in \CHom(A, B)}$. 
\enu
satisfying the following coherence conditions:
\bnu

\item {\it associativity coherence}: 
\be  \label{diag:asso-bicat}
\xymatrix{
((S \circ T)\circ U) \circ V \ar[rr]^{\alpha_{(S,T,U)}1}  
\ar[d]_{\alpha_{(ST,U,V)}}  & & (S\circ (T \circ U)) \circ V \ar[d]^{\alpha_{(S, TU, V)}} \\
(S\circ T) \circ (U\circ V) \ar[rd]_{\alpha_{(S,T,UV)}} &  & S\circ ((T\circ U) \circ V) \ar[ld]^{1\alpha_{(T,U,V)}} \\
& S\circ (T \circ (U\circ V)) &  
} 
\ee

\item {\it identity coherence}: 
\be  \label{diag:triangle-bicat}
\xymatrix{
(S \circ \one_B) \circ T  \ar[rd]_{r_S 1} \ar[rr]^{\alpha_{(S,\one_B, T)}} & & S\circ (\one_B \circ T) \ar[ld]^{1l_T}  \\
& S\circ T & 
}
\ee
\enu
\end{defn}

\begin{defn} \label{def:truncated-cat}
For a bicategory $\mathbf{S}$, we can truncate it to a $1$-category $\underline{\mathbf{S}}$ by ignoring 2-morphisms and defining $1$-morphisms in $\underline{\mathbf{S}}$ by the equivalence classes of $1$-morphisms in $\mathbf{S}$.  
\end{defn}

\begin{defn} \label{def:lax-functor}
Let $\bfC$ and $\bfD$ be two bicategories. A lax functor (or a lax $2$-functor) $\bfF: \bfC \rightarrow \bfD$ is a quadruple $\bfF=(F, \{\bfF_{(A,B)} \}_{A,B\in \bfC}, i, m)$ where 
\bnu
\item $F$ is a map of objects $X \mapsto F(X)$ for each object $X$ in $\bfC$; 
\item $\bfF_{(A,B)}: \CHom_\bfC(A, B)\to \CHom_\bfD(F(A), F(B))$ for each pair of objects $A,B\in \bfC$ is functor;
\item {\it unit transformation}: $\forall A$, $i_A:  \one_{F(A)} \to \bfF_{(A,A)} \circ \one_A$ where $\one_{F(A)}$ and $\bfF_{(A,A)} \circ \one_A$ are functors: $\one \to \CHom_\bfD(F(A), F(A))$;
\item {\it multiplication transformation}: 
$m: \circledcirc_\bfD \circ (\bfF_{(B,C)} \times \bfF_{(A,B)}) \to \bfF_{(A,C)} \circ \circledcirc_\bfC$, i.e.\ a collection of morphisms $m_{S,T}: \bfF_{(B,C)}(S) \circ \bfF_{(A,B)}(T) \to \bfF_{(A,C)}(S \circ T)$ natural in
$S \in \CHom_\bfC(B,C), T\in \CHom_\bfC(A, B)$.
\enu
satisfying the following commutative diagrams: 
\bnu
\item {\it associativity}: for $S \in \CHom_\bfC(C,D), T\in \CHom_\bfC(B,C), U\in \CHom_\bfC(A,B)$, 
\be \label{diag:asso-lax}
\xymatrix{
(\bfF_{(C,D)}(S) \circ \bfF_{(B,C)}(T)) \circ \bfF_{(A,B)}(U) \ar[r]^\alpha \ar[d]_{m1} 
& \bfF_{(C,D)}(S) \circ (\bfF_{(B,C)}(T) \circ \bfF_{(A,B)}(U)) \ar[d]^{1m} \\
\bfF_{(B,D)}(S\circ T) \circ \bfF_{(A,B)}(U) \ar[d]_m & \bfF_{(C,D)}(S) \circ \bfF_{(A,C)}(T\circ U) \ar[d]^m\\
\bfF_{(A,D)}((S\circ T) \circ U)
   \ar[r]^{\bfF_{(A,D)}(\alpha)} & 
\bfF_{(A,D)}(S\circ (T\circ U)) 
}
\ee
\item {\it unit properties}: for $S\in \CHom_\bfC(A, B)$, 
\be  \label{diag:unit-lax-l}
\xymatrix{
\one_{F(B)} \circ \bfF_{(A,B)}(S) \ar[r]^{\hspace{0.5cm}l_{F(S)}} \ar[d]_{i_B1}  & \bfF_{(A,B)}(S) \\
\bfF_{(B,B)}(\one_B) \circ \bfF_{(A,B)}(S) \ar[r]^m & \bfF_{(A,B)}(\one_B \circ S) \ar[u]_{\bfF_{(A,B)}(l_S)} 
}
\ee
\be  \label{diag:unit-lax-r}
\xymatrix{
\bfF_{(A,B)}(S) \circ \one_{F(A)} \ar[r]^{\hspace{0.5cm}r_{F(S)}}
    \ar[d]_{1i_A}  
& \bfF_{(A,B)}(S) \\
 \bfF_{(A,B)}(S) \circ \bfF_{(A,A)}(\one_A)  \ar[r]^m & \bfF_{(A,B)}(S\circ \one_A)\, . \ar[u]_{\bfF_{(A,B)}(r_S)} 
}
\ee
\enu
\end{defn}

If we reverse all arrows, we obtain the notion of an {\it oplax functor} (or an oplax $2$-functor). Given a lax functor $\bfF$, if the natural transformations $i_A, \forall A\in \bfC$ and $m$ are actually isomorphisms, then $\bfF$ is called a functor or a $2$-functor. 

Let $P$ be a property of a functor between $1$-categories like full, faithful, essentially surjective, etc. We say that a (lax, oplax or neither) functor is {\em locally} $P$, if for all objects $A,B$ the functors $\bfF_{(A,B)}$ have property $P$. 

If $\mathbf{S}$ is a bicategory and $\mathbf{S'}$ is a sub-bicategory (i.e.\ a subset of objects and collection of subcategories $\mathbf{S}'(A,B)$ for each $\mathbf{S}(A,B)$, such that the resulting embedding $\iota : \mathbf{S} \to \mathbf{S'}$ defines a (locally faithful) functor), then $\mathbf{S'}$ is {\em locally full} if the embedding functor $\iota$ is locally fully faithful. Less precisely said, a locally full sub-bicategory may be missing some objects and $1$-morphisms, but it will still contain all 2-morphisms between any two $1$-morphisms.

\begin{rema} \label{rema:truncated-functor}
If the unit and multiplication transformations of a lax $2$-functor $\bfF: \bfC \to \bfD$ are isomorphisms, 
we naturally obtain a $1$-functor $\underline{\bfF}: \underline{\bfC} \to \underline{\bfD}$ 
between two $1$-categories. In other words, a $2$-functor $\bfF$ automatically defines a $1$-functor $\underline{\bfF}$.
\end{rema}

\begin{defn}  \label{def:nat-trans-fun-bicat}  
A natural transformation $\sigma: \bfF \rightarrow \bfG$ between two lax functors $\bfF, \bfG: \bfC \to \bfD$ between two bicategories $\bfC$ and $\bfD$ contains the following data:
\bnu
\item 1-cell $F(A) \xrightarrow{\sigma_A} G(A)$; 
\item a natural transformation 
\be  \label{diag:lax-nat-1cell}
\xymatrix{ 
\bfC(A, B) \ar[rr]^{\bfF_{(A,B)}} \ar[dd]_{\bfG_{(A,B)}}  & & \bfD(F(A), F(B)) \ar[dd]^{(\sigma_B)_\ast}  \\
&  \swarrow \, \sigma_{AB} &  \\
\bfD(G(A), G(B)) \ar[rr]_{(\sigma_A)^\ast} && \bfD(F(A), G(B)). 
}
\ee
i.e for each 1-cell $A \xrightarrow{f} B$, we assign a 2-cell 
$\sigma_f: \sigma_B \circ \bfF(f) \to \bfG(f) \circ \sigma_A$ such that for all 2-cells $ f\xrightarrow{\phi} g$ in $\bfC$ we have the following commutative diagram:
\be  \label{diag:lax-nat-2cell}
\xymatrix{
\sigma_B \circ \bfF(f) \ar[r]^{\sigma_f}  \ar[d]_{1\bfF(\phi)} & \bfG(f) \circ \sigma_A \ar[d]^{\bfG(\phi)1} \\
\sigma_B \circ \bfF(g) \ar[r]^{\sigma_g}  & \bfG(g) \circ \sigma_A \, ,
}
\ee
\enu
satisfying the following axioms: 
\bnu
\item Omitting ``$\circ$" in the following diagram: For all $A \xrightarrow{f} B \xrightarrow{g} C$ in $\bfC$,
\be  \label{diag:lax-nat-axiom-1}
\xymatrix{
\sigma_C (\bfF(g) \bfF(f)) \ar[d]_{1m}  \ar[r]^{\alpha^{-1}} &  (\sigma_C \bfF(g) ) \bfF(f) \ar[r]^{\sigma_g 1} &
 (\bfG(g) \sigma_B)  \bfF(f) \ar[r]^{\alpha} &  \bfG(g)  (\sigma_B \bfF(f))  \ar[d]^{1\sigma_f}    \\ 
\sigma_C \bfF(g \circ f) \ar[r]^{\sigma_{gf}} & \bfG(g \circ f)  \sigma_A & 
(\bfG(g) \bfG(f))  \sigma_A \ar[l]_{m1} & \bfG(g)  (\bfG(f)  \sigma_A) \ar[l]_{\alpha^{-1}}
} 
\ee

\item For all $A \in \bfC$,
\be \label{diag:lax-nat-axiom-2}
\xymatrix{
\sigma_A \circ \one_{F(A)} \ar[d]_{1i_{F(A)}} \ar[r]^{\hspace{0.5cm}r} & \sigma_A \ar[r]^{\hspace{-0.5cm}l^{-1}} & \one_{G(A)} \circ \sigma_A 
\ar[d]^{i_{G(A)}1} \\
\sigma_A \circ F\one_A \ar[rr]^{\sigma_{\one_A}} & & G\one_A \circ \sigma_A 
}
\ee

\enu
\end{defn}

\begin{rema} \label{rema:truncated-nat-transf}
If each $\sigma_f$ is an isomorphism for all $1$-cells $A \xrightarrow{f} B$, then we obtain naturally an ordinary natural transformation $\underline{\sigma}: \underline{\bfF} \to \underline{\bfG}$ where $\underline{\bfF}, \underline{\bfG}: \underline{\bfC} \to \underline{\bfD}$ are two ordinary functors between two $1$-categories. 
\end{rema}

\subsection{Proof of Lemma~\ref{lem:beta}}  \label{app:proof-lem:beta}

\begin{proof} 
(i) We abbreviate $S \equiv S_2$ and $T \equiv T_2$ in this proof. Consider the following diagram:
\be \label{diag:exists-3-cell-m-1}
\raisebox{4em}{\small\xymatrix{
(M'BN')(SBT)(MBN) \ar@<+.7ex>[r]^{L_B^3}\ar@<-.7ex>[r]_{R_B^3} & 
  (M'N')(ST)(MN) \ar[r]^(.35){\rho_B^3} \ar@<+.7ex>[d]^{R_{ST}}\ar@<-.7ex>[d]_{L_{ST}}  & 
  (M'\otimes_B N')(S\otimes_B T)(M\otimes_B N)
\ar@{-->}@<+.7ex>[d]^{\exists! \,R_{S\otimes_B T}}\ar@{-->}@<-.7ex>[d]_{\exists! \,L_{S\otimes_B T}}  
\\
(M' B N') (M B N) \ar@<+.7ex>[r]^{L_B^2}\ar@<-.7ex>[r]_{R_B^2} 
& M' \otimes N' \otimes M \otimes N  \ar[r]^(.4){\rho_B^2}  \ar[d]_{1c_{N',M}1}  
&
   (M'\otimes_B N') \otimes (M\otimes_B N) \ar[dd]^{\rho_{S\otimes_B T}} 
   \ar@{-->}[ldd]_{\exists !\, \beta_1} 
\\
& M'\otimes M \otimes N' \otimes N \ar[d]_{\rho_B \circ (\rho_S\rho_T)} 
& 
 \\
& 
(M'\otimes_S M)\otimes_B (N'\otimes_T N) & (M'\otimes_B N') \otimes_{S \otimes_B T} (M\otimes_B N) 
\ar@{-->}[l]_{\exists ! \, \beta} 
}}
\ee
where $L_B$ and $R_B$ are the left and right action of $B$, and $L_{ST}$, $R_{ST}$ are the actions of $S \otimes T$ (involving a braiding), and $L_{S\otimes_BT}$ and $R_{S\otimes_BT}$ are given by the universal properties of $\rho_B^3$, which is the coequalizer of the pair $(L_B^3, R_B^3)$, 
and $\rho_{S\otimes_BT}$ is the coequalizer of the pair $(L_{S\otimes_BT}, R_{S\otimes_BT})$,
and $\rho_S\rho_T$ is the coequalizer of the pair $(L_SL_T,R_SR_T)$. 

Now we explain how to construct the maps $\beta_1$ and $\beta$. First, using graphic calculus and the defining properties of 2-diagram (\ref{eq:2diagram-cond}), we can check that 
\be
\rho_B \circ (\rho_S\rho_T) \circ (1c_{N'M}1) \circ L_B^2
=
\rho_B \circ (\rho_S\rho_T) \circ (1c_{N'M}1) \circ R_B^2.
\ee
Since $\rho_B^2$ is the coequalizer of the pair $(L_B^2, R_B^2)$, we obtain the unique map $\beta_1$ in Diagram (\ref{diag:exists-3-cell-m-1}) such that 
\be \label{eq:beta-1}
\beta_1 \circ \rho_B^2 = \rho_B \circ (\rho_S\rho_T) \circ (1c_{N'M}1).
\ee 
Secondly, using graphic calculus, it is easy to see that 
$\rho_B \circ (\rho_S\rho_T) \circ (1c_{N'M}1) \circ L_{ST}=
\rho_B \circ (\rho_S\rho_T) \circ (1c_{N'M}1) \circ R_{ST}$
which implies immediately that 
\be \label{eq:const-m-1}
\beta_1 \circ (\rho_B^2 \circ L_{ST}) = \beta_1 \circ (\rho_B^2 \circ R_{ST}).
\ee
Since $\rho_B^3$ is epi, it is easy to see that $\rho_{S\otimes_BT}$ is also the coequalizer of the pair $(L_{S\otimes_BT}\circ \rho_B^3, R_{S\otimes_BT}\circ \rho_B^3)$. By the commutativity of the upper square in Diagram (\ref{diag:exists-3-cell-m-1}), we obtain that 
$\rho_{S\otimes_BT}$ is also the coequalizer of the pair $(\rho_B^2\circ L_{ST}, \rho_B^2\circ R_{ST})$.
Then the identity (\ref{eq:const-m-1}) implies that the existence and the uniqueness of morphism $\beta$, which is shown in diagram in (\ref{diag:exists-3-cell-m-1}), such that 
\be  \label{eq:const-m-2}
\beta_1 = \beta \circ \rho_{S\otimes_BT}. 
\ee
We define $\beta_{(M'N'),(MN)}:=\beta$.

\medskip

\noindent
(ii) We will prove that $\beta$ is an isomorphism by giving an explicit construction of the inverse map $\beta_3$ of $\beta$ below similar to that of $\beta_1$ and $\beta$. Consider the following diagram:
\be \label{eq:diag:exists-3-cell-m-2}
\raisebox{4em}{\small\xymatrix{
(M'SM)B(N'TN) \ar@<+.7ex>[r]^{L_S\,1\,L_T}\ar@<-.7ex>[r]_{R_S\,1\,R_T} & 
  (M'M)B(N'N) \ar[r]^(.35){\rho_S\,1\,\rho_T} \ar@/^5pt/[d]^{1\,R_{B}1}\ar@/_5pt/[d]_{1\,L_{B}1}  & 
  (M'\otimes_S M) \otimes B \otimes (N' \otimes_{T} N)
\ar@{-->}@/^5pt/[d]^{\exists! \,R}\ar@{-->}@/_5pt/[d]_{\exists! \,L}  
\\
(M'SM)(N'TN) \ar@<+.7ex>[r]^{L_SL_T}\ar@<-.7ex>[r]_{R_SR_T} 
& M' \otimes M \otimes N' \otimes N  \ar[r]^(.4){\rho_S\rho_T} \ar[d]_{1c_{N',M}^{-1}1}  &
  (M'\otimes_S M) \otimes (N' \otimes_{T} N) \ar[dd]^{\rho_B} \ar@{-->}[ldd]_{\exists !\, \beta_2}
\\
& M'\otimes N' \otimes M \otimes N \ar[d]_{\rho_{S\otimes_BT}\circ \rho_B^2}  & \\
&(M'\otimes_B N') \otimes_{S \otimes_B T} (M\otimes_B N)  & 
(M' \otimes_{S} M) \otimes_B (N'\otimes_{T} N) \ar@{-->}[l]_{\exists !\, \beta_3}
}}
\ee
where $L$ and $R$ are given by the universal property of $\rho_S\,1\,\rho_T$ as the coequalizer of the pair $(L_S\,1\,L_T, R_S\,1\,R_T)$. The construction of $\beta_2$ and $\beta_3$ follow from the similar argument as that of $\beta_1$ and $\beta$. First, using graphic calculus, we would like to prove the following identity:
\be  \label{eq:const-m-3}
\rho_{S\otimes_BT} \circ \rho_B^2 \circ (1c_{N',M}^{-1}1) \circ (L_SL_T)
=
\rho_{S\otimes_BT} \circ \rho_B^2 \circ (1c_{N',M}^{-1}1) \circ (R_SR_T). 
\ee
It follows by first composing both sides by the isomorphism $(111c_{T,M}1) \circ (1c_{N', SM}1)$, then applying the commutativity of upper square in (\ref{diag:exists-3-cell-m-1}). The identity (\ref{eq:const-m-3}), together with the universal property of $\rho_S\rho_T$, implies the existence and uniqueness of $\beta_2$ such that
\be \label{eq:const-m-4}
\beta_2\circ (\rho_S\rho_T)=\rho_{S\otimes_BT} \circ \rho_B^2 \circ (1c_{N',M}^{-1}1).
\ee 
Secondly, using graphic calculus and definition properties of 2-diagram (\ref{eq:2diagram-cond}), it is easy to prove that the following identity:
$\rho_{S\otimes_BT} \circ \rho_B^2 \circ (1L_B1)=\rho_{S\otimes_BT} \circ \rho_B^2 \circ (1R_B1)$
holds. By (\ref{eq:const-m-4}), we obtain 
\be \label{eq:const-m-5}
\beta_2 \circ ( (\rho_S\rho_T) \circ (1L_B1) ) 
=\beta_2 \circ ( (\rho_S\rho_T) \circ (1R_B1) ). 
\ee
On the other hand, $\rho_B$ is the coequalizer of the pair $(L,R)$, hence also the coequalizer of 
the pair $(L\circ (\rho_S1\rho_T), R \circ (\rho_S1\rho_T))$, and that of the pair $((\rho_S\rho_T) \circ (1L_B1), (\rho_S\rho_T) \circ (1R_B1))$. Then (\ref{eq:const-m-5}), together with the universal property of $\rho_B$, implies the existence and uniqueness of $\beta_2$ such that 
\be \label{eq:beta-4-3}
\beta_3 \circ \rho_B = \beta_2. 
\ee

\medskip
To see that $\beta$ is invertible, it is enough to prove that 
\be  \label{eq:m-invertible}
\beta \circ \beta_3=\id_{(M'\otimes_SM) \otimes_B (N'\otimes_TN)}
\quad \mbox{and} \quad 
\beta_3 \circ \beta = \id_{(M'\otimes_B N') \otimes_{S\otimes_B T} (M\otimes_B N) }.
\ee 
We will only prove the second
identity in (\ref{eq:m-invertible}). 
The proof for the second identity in (\ref{eq:m-invertible}) is exactly same. We have
\bea
\beta_3 \circ \beta_1 \circ \rho_B^2 &\overset{\eqref{eq:beta-1}}{=}& \beta_3 \circ \rho_B \circ (\rho_S\rho_T) \circ (1c_{N',M}1) \nn
&\overset{\eqref{eq:beta-4-3}}{=}& \beta_2 \circ (\rho_S \rho_T) \circ (1c_{N',M}1) \nn
&\overset{\eqref{eq:const-m-4}}{=}& \rho_{S\otimes_BT} \circ \rho_B^2 \circ (1c_{N',M}^{-1}1) \circ (1c_{N',M}1) \nn
&=& \rho_{S\otimes_BT} \circ \rho_B^2 \ . 
\eea
By the fact that $\rho_B^2$ is epi, we obtain $\beta_3 \circ \beta_1 = \rho_{S\otimes_BT}$, which further implies the identity $\beta_3 \circ \beta \circ \rho_{S\otimes_BT}=\rho_{S\otimes_BT}$. Again by the fact that $\rho_{S\otimes_BT}$ is epi, we obtain the second
identity in (\ref{eq:m-invertible}). 

\medskip
\noindent 
(iii) According to Definition~\ref{def:2-diagram+3-cell}\,(iii) we need to check that $\beta_{(M'N'),(MN)}$ is an $(S_3 \otimes_B T_3)$-$(S_1 \otimes_B T_1)$-bimodule map, 
and that the two sub-diagrams in \eqref{eq:Cosp-comp-3cells} commute. For convenience, we again set $S:=S_2, T:=T_2$.

To prove that $\beta$ is an $S_1 \otimes_B T_1$-module map, we need prove the commutativity of the lower square of the following diagram: 
\be  \label{diag:m-module-map}
\xymatrix{
(M'\otimes N') (M\otimes N) (S_1\otimes T_1) \ar[rr]^{\rho_B^3}  \ar[d]_{f} & & (M'\otimes_B N') (M\otimes_BN) (S_1\otimes_BT_1) \ar[d]_{\rho_{S\otimes_BT}} \\
(M'\otimes_S M)\otimes_B (N'\otimes_TN)(S_1\otimes_BT_1) \ar[d]_{R_{S_1\otimes_BT_1}}  \ar[d]_{}& &
(M'\otimes_B N')\otimes_{S\otimes_BT}(M\otimes_BN)(S_1\otimes_BT_1)
\ar[d]^{R_{S_1\otimes_BT_1}}
\ar[ll]_{\beta1} \\
(M'\otimes_S M) \otimes_B
(N'\otimes_TN) && 
(M'\otimes_B N')\otimes_{S\otimes_BT}(M\otimes_BN) \ar[ll]_{\beta}
}
\ee
where $f:=\rho_B^2\circ (\rho_S\rho_T1)\circ ((1c_{N',M}1)1)$.
The commutativity of the upper square is just the lower square in \eqref{diag:exists-3-cell-m-1}.
Using the definition of $R_{S_1\otimes_BT_1}$, it is easy to see that we can switch the order of the right action and $f$ in 
the composed map $R_{S_1\otimes_BT_1} \circ f$, i.e.
$$
R_{S_1\otimes_BT_1} \circ f = \rho_B^2\circ (\rho_S\rho_T1)\circ (1c_{N',M}1) \circ R_{S_1\otimes T_1}.
$$
Similarly, we have
$$
R_{S_1\otimes_BT_1} \circ \rho_{S\otimes_BT} \circ \rho_B^3 =
\rho_{S\otimes_B T} \circ \rho_B^2 \circ R_{S_1\otimes T_1}. 
$$
Then, by the commutativity of the lower square in (\ref{diag:exists-3-cell-m-1}), it is clear that the outer square in (\ref{diag:m-module-map}) is also commutative. Then the commutativity of the lower square follows from the fact that a coequalizer map is an epimorphism. 

Similarly, we can show that $\beta$ is an $S_3 \otimes_B T_3$-module map. We omit the details. Therefore, $\beta$ is a $(S_3 \otimes_B T_3)$-$(S_1 \otimes_B T_1)$-bimodule map. 

\medskip
To show that the upper sub-diagrams in \eqref{eq:Cosp-comp-3cells} commute, we consider the following diagram
\be \label{eq:m-up-triangle-pf}
\xymatrix{
S_1\otimes T_1 \ar[rr]^{\rho_B} \ar[d]_{\tilde{u}} & & S_1\otimes_B T_1 \ar[d]^{u}  \\
(M'\otimes N') \otimes (M\otimes N) \ar[rr]^{\rho_B^2} \ar[d]_{\rho_B\circ (\rho_S\rho_T) \circ (1c_{N',M}1)}& & (M'\otimes_B N') \otimes (M\otimes_B N)
\ar[d]_{\rho_{S\otimes_B T}} \\
(M'\otimes_SM) \otimes (N'\otimes_TN) &&
(M'\otimes_SM) \otimes_{S\otimes_BT} (N'\otimes_TN) \ar[ll]_{\beta}
}
\ee
where $u$ is defined as the $u$ in (\ref{eq:2-diag-com-u-def}) and $\tilde{u}$ is defined 
similarly as $u$ such that the upper square in (\ref{eq:m-up-triangle-pf}) commutes. The lower square in (\ref{eq:m-up-triangle-pf}) is nothing but the commutative lower square in (\ref{eq:diag:exists-3-cell-m-2}). Notice that the right column $\rho_{S\otimes_BT} \circ u$ in (\ref{eq:m-up-triangle-pf}) is just the left arrow in the upper triangle in \eqref{eq:Cosp-comp-3cells}. Moreover, the left column in (\ref{eq:m-up-triangle-pf}) can be shown to give the right arrow in the upper triangle in \eqref{eq:Cosp-comp-3cells} composed with $\rho_B: S_1\otimes T_1 \to S_1\otimes_BT_1$. Since $\rho_B$ is an epimorphism, we obtain that the upper triangle in \eqref{eq:Cosp-comp-3cells} commutes. 

The proof of the commutativity of the lower triangle in \eqref{eq:Cosp-comp-3cells} is similar. We omit it. 

\medskip
\noindent (iv) The naturality of $\beta_{(M'N'),(MN)}$  amounts to commutativity of the diagram (write $S \equiv S_2$ and $T \equiv T_2$)
\be
\raisebox{2.3em}{\xymatrix{
(M' \otimes_B N') \otimes_{S \otimes_B T} (M\otimes_B N) 
  \ar[rr]^{\beta_{(M'N'),(MN)}} \ar[d]_{(\phi' \otimes_B \psi')\otimes_{S\otimes_B T}(\phi\otimes_B \psi)}  & &
(M'\otimes_{S} M) \otimes_B (N'\otimes_{T} N) 
  \ar[d]^{(\phi' \otimes_S \phi) \otimes_B (\psi'\otimes_{T} \psi)} 
\\
(\tilde{M}'\otimes_B \tilde{N}') \otimes_{S \otimes_B T} (\tilde{M} \otimes_B \tilde{N}) \ar[rr]^{\beta_{(\tilde M'\tilde N'),(\tilde M\tilde N)}} & &
(\tilde{M}' \otimes_{S} \tilde{M}) \otimes_B (\tilde{N}'\otimes_{T} \tilde{N}) \ ,
}}
\ee
where $M\xrightarrow{\phi}\tilde{M}$, $M'\xrightarrow{\phi'} \tilde{M}'$, $N\xrightarrow{\psi} \tilde{N}'$ and $N'\xrightarrow{\psi'} \tilde{N}'$ are 3-cells. Commutativity holds before passing to the tensor products over $B$ and $S \otimes_B T$, so that the commutativity of the above diagram follows from the universal property of the coequalizers.
\end{proof}

\subsection{Proof of Proposition~\ref{prop:C_ABC-lax-functor}}  \label{app:proof-prop:C_ABC-lax-functor}

\begin{proof}
Below we establish the associativity \eqref{diag:asso-lax} and identity conditions \eqref{diag:unit-lax-l}, \eqref{diag:unit-lax-r}. Once this is done, it is clear that $\Coco$ will be a 2-functor
 from $\Cosp(B,C) \times \Cosp(A,B)$ to $\Cosp(A,C)$, as the associativity morphisms are  isomorphisms by Lemma~\ref{lem:beta} and unit morphisms \eqref{eq:laxfun-data3} are just identities.
\\[.5em]
\nxt associativity condition:\\
We need to pick four objects in the source category $\Cosp(B,C) \times \Cosp(A,B)$, say
\be
\raisebox{1.5em}{\xymatrix@R=1em@C=1em{ & S_i && T_i \\ A\ar[ru] && B \ar[lu] \ar[ru] && C \ar[lu]}}
\qquad ; ~~ i=1,2,3,4 \ .
\ee
Next we pick three 1-morphisms, each between the object of index $i$ and $i+1$. That is, we pick an object in $\tdiag_{BC}(T_i,T_{i+1}) \times \tdiag_{AB}(S_i,S_{i+1})$ for $i=1,2,3$:
\be\label{eq:lax-functor-proof-aux1}
\raisebox{2.6em}{\xymatrix@R=1em@C=1em{
& S_i \ar[d] && T_i \ar[d] \\
A \ar[ru] \ar[rd] & M_{i+1,i} & B \ar[ru] \ar[lu] \ar[rd] \ar[ld] & N_{i+1,i} & C \ar[lu] \ar[ld] \\
& S_{i+1} \ar[u] && T_{i+1} \ar[u]}}
\qquad ; ~~ i=1,2,3 \ .
\ee
Then we need to prove commutativity of the diagram \eqref{diag:asso-lax}. Substituting the definition of $\Coco_{(S,T),(S',T')}$ in \eqref{eq:laxfun-data2} makes this a somewhat cumbersome expression. For the left hand column of \eqref{diag:asso-lax} this results in
\be
\raisebox{4em}{\xymatrix{
\big[ (M_{43} \otimes_B N_{43}) \otimes_{S_3 \otimes_B T_3} (M_{32} \otimes_B N_{32}) \big]
\otimes_{S_2 \otimes_B T_2} (M_{21} \otimes_B N_{21})
\ar[d]^{\beta_{(M_{43}N_{43}),(M_{32}N_{32})}\,1}
\\
\big[ (M_{43} \otimes_{S_3} M_{32}) \otimes_B (N_{43} \otimes_{T_3} N_{32}) \big]
\otimes_{S_2 \otimes_B T_2} (M_{21} \otimes_B N_{21})
\ar[d]^{\beta_{((M_{43} \otimes_{S_3} M_{32})\, (N_{43} \otimes_{T_3} N_{32})),(M_{21}\, N_{21})}}
\\
\big[ (M_{43} \otimes_{S_3} M_{32}) \otimes_{S_2} M_{21} \big] \otimes_B
\big[ (N_{43} \otimes_{T_3} N_{32}) \otimes_{T_2} N_{21} \big]  \ ,
}}
\ee
where $\beta$ was defined in \eqref{eq:Cosp-comp-3cells-def}.
For the right hand column, one obtains
\be
\raisebox{4em}{\xymatrix{
(M_{43} \otimes_B N_{43}) \otimes_{S_3 \otimes_B T_3} \big[ (M_{32} \otimes_B N_{32}) 
\otimes_{S_2 \otimes_B T_2} (M_{21} \otimes_B N_{21}) \big]
\ar[d]^{1\,\beta_{(M_{32}\, N_{32}),(M_{21}\, N_{21})}}
\\
(M_{43} \otimes_B N_{43}) \otimes_{S_3 \otimes_B T_3} \big[ (M_{32} \otimes_{S_2} M_{21}) 
\otimes_B (N_{32} \otimes_{T_2} N_{21}) \big]
\ar[d]^{\beta_{(M_{43}\, N_{43} ), ((M_{32} \otimes_{S_2} M_{21})\, (N_{32} \otimes_{T_2} N_{21}))}}
\\
\big[ M_{43} \otimes_{S_3} (M_{32} \otimes_{S_2} M_{21}) \big] \otimes_B
\big[ N_{43} \otimes_{T_3} (N_{32} \otimes_{T_2} N_{21}) \big]  \ .
}}
\ee
The top and bottom entry of the two columns are linked by the obvious associator isomorphism.
To show that the hexagon \eqref{diag:asso-lax} commutes in the present case, one first considers
the corresponding diagram before passing to the fibered products, which by 
\eqref{eq:Cosp-comp-3cells-def} reads
(omitting all associators)
\be
\raisebox{3.7em}{\xymatrix{
M_{43} \, N_{43} \,  M_{32} \,  N_{32} \, M_{21} \, N_{21}
\ar[d]^{1\,c_{N_{43} ,  M_{32}}\,1\,1\,1}
\ar[r]^{=}
&
M_{43} \, N_{43} \,  M_{32} \,  N_{32} \, M_{21} \, N_{21}
\ar[d]^{1\,1\,1c_{N_{32} ,  M_{21}}\,1}
\\
M_{43} \, M_{32} \,  N_{43} \,  N_{32} \, M_{21} \, N_{21}
\ar[d]^{1\,1\,c_{N_{43} \otimes  N_{32},M_{21}}\,1}
&
M_{43} \, N_{43} \,  M_{32} \,  M_{21} \, N_{32} \, N_{21}
\ar[d]^{1\,c_{N_{43} , M_{32} \otimes  M_{21}}\,1\,1}
\\
M_{43} \, M_{32} \,  M_{21} \, N_{43} \,  N_{32} \, N_{21}
\ar[r]^{=}
&
M_{43} \, M_{32} \,  M_{21} \, N_{43} \,  N_{32} \, N_{21}
}}
\ee
This diagram indeed commutes (by the properties of the braiding), 
so that the associativity coherence can be established 
using the universal property of coequalizers.
\\[.5em]
\nxt identity condition:\\
We use the notation \eqref{eq:lax-functor-proof-aux1} for $i=1,2$. Recall the identity 2-diagram from \eqref{eq:2diag-identities}. The first identity condition, given in \eqref{diag:unit-lax-l}, expands to 
\be
\raisebox{3.7em}{\xymatrix{
(S_2 \otimes_B T_2) \otimes_{S_2 \otimes_B T_2} (M_{21} \otimes_B N_{21})
  \ar[d]^{=}
  \ar[rrr]^{\lambda_{M_{21} \otimes_B N_{21}}}
&&&
  M_{21} \otimes_B N_{21}
  \ar[d]^{\lambda^{-1}_{M_{21}} \otimes_B \lambda^{-1}_{N_{21}}}
\\
(S_2 \otimes_B T_2) \otimes_{S_2 \otimes_B T_2} (M_{21} \otimes_B N_{21})
\ar[rrr]^{\beta_{(S_2,T_2),(M_{21},N_{21})}}
&&&
(S_2 \otimes_{S_2} M_{21}) \otimes_B (T_2 \otimes_{T_2} N_{21}) \ ,
}}
\ee
where $\lambda_M$, for a left $A$-module $M$, is the canonical isomorphism $A \otimes_A M \to M$.
Upon taking the inverses of the vertical arrows, one again finds that the square commutes before taking the fibered tensor products (i.e.\ one finds an equality of two morphisms $S_2 \otimes T_2 \otimes M_{21} \otimes N_{21} \to M_{21} \otimes N_{21}$). As above, one then employs the universal property of coequalizers to show \eqref{diag:unit-lax-l}. The second condition, given in \eqref{diag:unit-lax-r}, is checked analogously.
\end{proof}

\newcommand\arxiv[2]      {\href{http://arXiv.org/abs/#1}{#2}}
\newcommand\doi[2]        {\href{http://dx.doi.org/#1}{#2}}
\newcommand\httpurl[2]    {\href{http://#1}{#2}}

\end{document}